\documentclass[11pt,reqno]{amsart}
\usepackage[utf8]{inputenc}
\usepackage{amssymb}
\usepackage[margin=1in]{geometry}
\usepackage{graphicx}
\usepackage{booktabs}
\usepackage[numbers]{natbib}
\usepackage{xcolor}
\usepackage{caption}
\usepackage{subcaption}
\usepackage[colorlinks=true]{hyperref}
\definecolor{linkcol}{HTML}{1A4C8B}  
\definecolor{citecol}{HTML}{2E7D32}  
\definecolor{urlcol}{HTML}{9C1B2E}   
\hypersetup{linkcolor=linkcol, citecolor=citecol, urlcolor=urlcol}
\usepackage{tikz}
\usepackage{pgfplots}
\usepackage[ruled,vlined]{algorithm2e}
\usepackage{graphicx}
\usepackage[all,cmtip]{xy}
\pgfplotsset{compat=1.18}
\usetikzlibrary{positioning, arrows.meta}
\usepackage{eulervm,charter}
\DeclareMathOperator{\rank}{rank}
\newcommand{\R}{\mathbb{R}}
\newcommand{\Z}{\mathbb{Z}}
\newcommand{\Alpha}{\mathrm{Alpha}}

\makeatletter
\renewcommand{\paragraph}{\@startsection{paragraph}{4}%
  {\z@}{.5\linespacing\@plus.7\linespacing}{-.5em}%
  {\normalfont\bfseries}}
\makeatother

\usepackage{enumitem}
\usepackage{tcolorbox}
\tcbuselibrary{breakable}

\newtheorem{theorem}{Theorem}[section]
\newtheorem{lemma}[theorem]{Lemma}
\newtheorem{proposition}[theorem]{Proposition}
\newtheorem{corollary}[theorem]{Corollary}

\theoremstyle{definition}
\newtheorem{definition}[theorem]{Definition}

\newtheorem{example}[theorem]{Example}

\theoremstyle{remark}
\newtheorem{remark}[theorem]{Remark}

\title[The Intersection Euler Characteristic Profile]{The Intersection Euler Characteristic Profile: Euler Calculus and Stability for Topological Interaction of Ball Unions}

\author{Kazuhiro Kawamura}
\address{Department of Mathematics, University of Tsukuba, Japan}
\email{kawamura@math.tsukuba.ac.jp}

\author{Sushovan Majhi}
\address{Data Science Program, George Washington University, Washington D.C., USA}
\email{sushovan@gwu.edu}

\author{Atish Mitra}
\address{Department of Mathematical Sciences, Montana Technological University, Butte, MT, USA}
\email{amitra@mtech.edu}



\begin{document}

\begin{abstract}
The \emph{Intersection Euler Characteristic Profile} (Intersection ECP) of $k$ colored point clouds $X_1, \ldots, X_k \subset \R^d$ is the Euler characteristic $\chi(\bigcap_{i=1}^k \mathcal{U}(X_i; t_i))$ of the overlap of their ball unions---an integer-valued, multiparameter invariant of their topological interaction across scales.
Its organizing framework is the Euler calculus on constructible functions: the profile is equally the Euler integral $\int \prod_{i=1}^k \mathbf{1}_{\mathcal{U}(X_i; t_i)}\, d\chi$ of the product of the $k$ data-dependent offsets, and this identity---our \emph{Intersection Theorem}---is a commuting square interchanging geometric intersection and algebraic product.
The invariant is rigid-motion invariant, scale-equivariant, and $L^1$-stable, and it is canonical: among pointwise-Euler interaction profiles it is the one forced by separation and normalization, the top floor of a spectrum of descriptors graded by how many clouds meet.
For $n$ points a single sorted Alpha-complex sweep computes it in $O(n^{\lceil d/2\rceil}\log n)$ time with \emph{no} persistence reduction, worst-case optimal in even dimensions.
Where the Euler characteristic cancels, a relative-homology refinement resolves the finer interaction and is stable in the two-parameter interleaving distance.
Finally, for increasingly dense samples the profile and its refinement are consistent, recovering the (relative) homology and Euler characteristic of the underlying shapes---in the inverse limit for compact sets, and, under positive reach, persistently and with explicit sample complexity.
\end{abstract}

\maketitle

\section{Introduction}\label{sec:introduction}

A fundamental problem in topological data analysis (TDA) is to quantify how two---or more---point clouds interact across spatial scales.
Given finite clouds $X_1, \ldots, X_k \subset \R^d$, each thickened to its scale-parametrized ball union $\mathcal{U}(X_i; r) = \bigcup_{p \in X_i} B(p, r)$, we ask for a single \emph{scale-dependent integer} $I(r)$ that measures their topological interaction.
A useful such invariant should satisfy:
\begin{itemize}
    \item[\textbf{(P1)}] \textbf{Separation.} $I(r) = 0$ while the clouds remain disjoint at scale $r$: no interaction registers before the thickenings meet.
    \item[\textbf{(P2)}] \textbf{Consistency.} When the clouds are dense samples of a common shape $K$, $I(r)$ stabilizes to a topological invariant of $K$---one fixed by its shape, not its size or position.
        \item[\textbf{(P3)}] \textbf{Symmetry and $k$-independence.} $I(r)$ treats the clouds symmetrically---no cloud is a privileged ``foreground''---and is defined for every $k$ with no arbitrary choices.
    \item[\textbf{(P4)}] \textbf{Stability.} $I(r)$ changes by a controlled amount under perturbation of the clouds---ideally Lipschitz in the Hausdorff distance.
    \item[\textbf{(P5)}] \textbf{Tractability and inference.} $I(r)$ is cheap to compute and, being a curve rather than a diagram, enters statistical inference and learning directly---admitting a hypothesis test that separates genuine interaction from chance---with no arbitrary vectorization step.
\end{itemize}

Persistence barcodes fail (P1)--(P2) in opposite ways: the diagram of a single cloud is blind to the presence of another---two Gaussian clusters have identical diagrams however they are positioned---while the diagram of the union $X \cup Y$ responds to the clouds' relative position while they are still disjoint and conflates each cloud's intrinsic topology with the interaction we wish to isolate.
A natural categorified response is \emph{multiparameter persistence}, filtering the two clouds jointly by their separate scales $(r, s)$; but two-parameter persistence modules are notoriously intractable---they admit no interval decomposition, and hence no barcode \citep{carlsson2009theory}, carry no complete discrete invariant \citep{botnan2022introduction}, and their interleaving distance is NP-hard to compute \citep{bjerkevik2018computing}---so they furnish no usable scale-dependent number of the kind (P1)--(P5) demand.
Two recent frameworks instead measure interaction directly.
\emph{Mixup barcodes} \citep{wagner2024mixup} track how the persistence bars of one cloud are shortened by the presence of another through image persistence; \emph{chromatic} TDA \citep{cultrera2024chromatic} colors the points and reads a \emph{six-pack} of persistence diagrams off chromatic Delaunay and alpha complexes \citep{montesano2022chromatic}.
Both capture genuine---and, in the chromatic case, remarkably fine---interaction, and the mixup barcode is stable.
Measured against (P1)--(P5), however, each leaves gaps that its own authors record as open.
Mixup barcodes are \emph{asymmetric}---the inclusion $L \hookrightarrow K$ privileges one cloud (P3)---admit \emph{no canonical extension to $k > 2$ clouds} and \emph{no hypothesis test} (both stated as open problems by \citet{wagner2024mixup}), and rely on an image-persistence reduction cubic in the simplex count of a Vietoris--Rips filtration that only approximates ball-union geometry (P5).
Chromatic TDA is tied to the Euclidean Delaunay--Voronoi construction, returns a full six-pack rather than a single integer, and develops no stability, sampling, or inferential theory.
Both, moreover, are diagram-valued: to enter a statistical test or a learning algorithm they must first be \emph{vectorized}---through persistence images, landscapes, landmarks, or ad hoc scalar summaries such as total mixup or the MST-ratio---a step that is lossy and non-canonical. The situation is worse for the chromatic framework, whose companion scalar---the MST-ratio, the combined length of the
monochromatic spanning trees over that of the multichromatic
one~\citep{cultrera2024mstratio}---is not a summary of the six-pack at all, but
an independent geometric construction: real-valued, purely $\beta_0$-level, and
returning a single number rather than a scale-indexed profile.
In each case the \emph{intuition} is right (to measure the interaction itself) but no single invariant is at once symmetric, $k$-natural, geometrically exact, cheap, and testable.

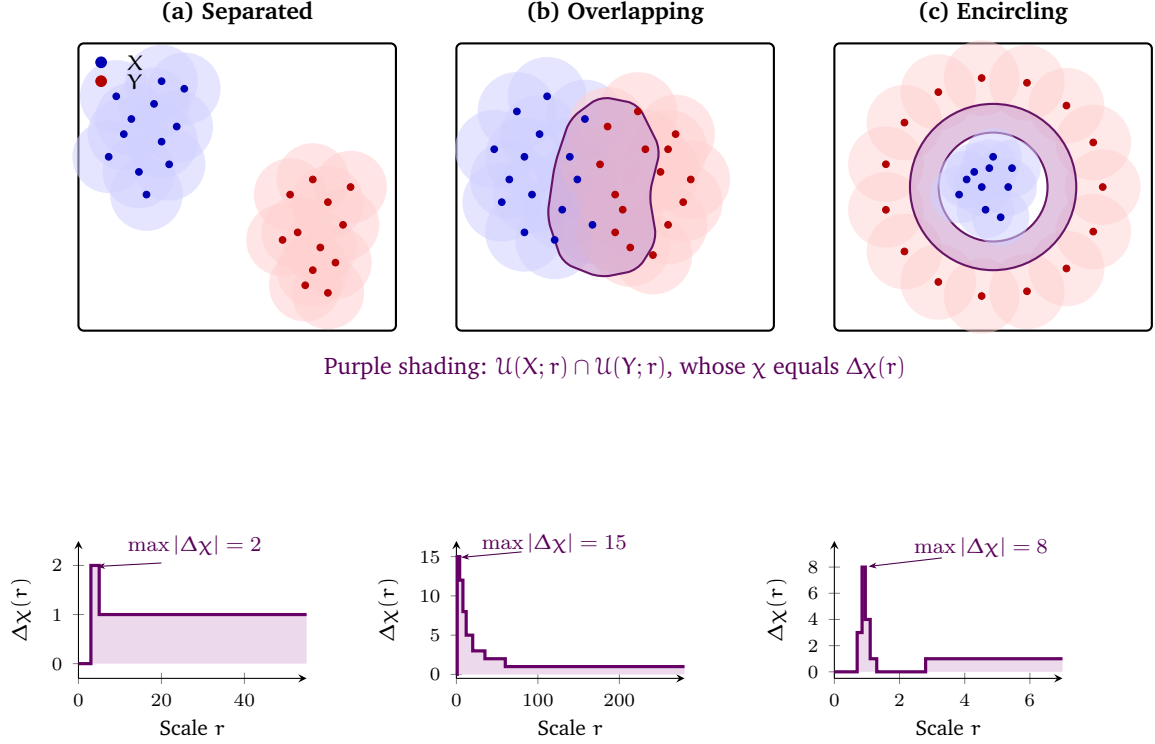
\begin{figure}[t]
\centering
\begin{tikzpicture}[
    every node/.style={font=\footnotesize},
    bluept/.style={fill=blue!70!black, draw=blue!70!black, circle, inner sep=0pt, minimum size=2.5pt},
    redpt/.style={fill=red!70!black, draw=red!70!black, circle, inner sep=0pt, minimum size=2.5pt},
]

\def\panelW{4.2}
\def\panelH{3.8}
\def\gapX{5.0}

\begin{scope}[shift={(0,0)}]
  \node[font=\footnotesize\bfseries] at (0.5*\panelW, \panelH+0.4) {(a) Separated};
  \draw[thick, rounded corners=2pt] (0,0) rectangle (\panelW,\panelH);
  \foreach \x/\y in {0.6/2.6, 1.0/3.0, 0.8/2.1, 1.3/2.7, 1.1/3.3, 0.5/3.1, 1.2/2.2, 0.9/1.8, 1.4/3.2, 0.7/2.8, 1.1/2.5, 0.4/2.3}{
    \fill[blue!18, opacity=0.55] (\x,\y) circle (0.48);
  }
  \foreach \x/\y in {0.6/2.6, 1.0/3.0, 0.8/2.1, 1.3/2.7, 1.1/3.3, 0.5/3.1, 1.2/2.2, 0.9/1.8, 1.4/3.2, 0.7/2.8, 1.1/2.5, 0.4/2.3}{
    \node[bluept] at (\x,\y) {};
  }
  \foreach \x/\y in {2.9/1.3, 3.3/1.7, 3.1/0.8, 3.5/1.4, 2.8/1.8, 3.2/1.1, 3.6/1.9, 3.0/0.6, 3.4/0.9, 2.7/1.2, 3.3/0.5, 3.1/2.0}{
    \fill[red!18, opacity=0.55] (\x,\y) circle (0.48);
  }
  \foreach \x/\y in {2.9/1.3, 3.3/1.7, 3.1/0.8, 3.5/1.4, 2.8/1.8, 3.2/1.1, 3.6/1.9, 3.0/0.6, 3.4/0.9, 2.7/1.2, 3.3/0.5, 3.1/2.0}{
    \node[redpt] at (\x,\y) {};
  }
  \node[bluept, minimum size=4pt] at (0.3,3.55) {};
  \node[anchor=west, font=\scriptsize] at (0.5,3.55) {$X$};
  \node[redpt, minimum size=4pt] at (0.3,3.3) {};
  \node[anchor=west, font=\scriptsize] at (0.5,3.3) {$Y$};
\end{scope}

\begin{scope}[shift={(\gapX,0)}]
  \node[font=\footnotesize\bfseries] at (0.5*\panelW, \panelH+0.4) {(b) Overlapping};
  \draw[thick, rounded corners=2pt] (0,0) rectangle (\panelW,\panelH);
  \foreach \x/\y in {0.7/2.0, 1.1/2.6, 1.4/1.6, 0.8/2.9, 1.5/2.3, 0.9/1.3, 1.7/2.8, 0.5/2.4, 1.3/1.2, 1.0/1.8, 1.6/2.0, 0.6/1.7, 1.2/3.1, 1.8/1.4, 0.9/2.3}{
    \fill[blue!18, opacity=0.55] (\x,\y) circle (0.52);
  }
  \foreach \x/\y in {2.1/1.8, 2.5/2.4, 2.8/1.4, 2.0/2.7, 2.7/2.1, 2.3/1.1, 2.9/2.6, 1.9/2.2, 2.6/1.0, 2.2/1.6, 2.4/2.9, 3.0/1.7, 2.1/1.3, 2.8/2.4, 3.1/2.0}{
    \fill[red!18, opacity=0.55] (\x,\y) circle (0.52);
  }
  \fill[violet!35, opacity=0.7, rounded corners=4pt]
    (1.4,0.9) -- (1.9,0.7) -- (2.4,0.8) -- (2.65,1.2) -- (2.55,1.8) -- (2.7,2.4) -- (2.4,3.0) -- (1.9,3.1) -- (1.5,2.8) -- (1.3,2.2) -- (1.2,1.5) -- cycle;
  \draw[violet!70!black, thick, opacity=0.9, rounded corners=4pt]
    (1.4,0.9) -- (1.9,0.7) -- (2.4,0.8) -- (2.65,1.2) -- (2.55,1.8) -- (2.7,2.4) -- (2.4,3.0) -- (1.9,3.1) -- (1.5,2.8) -- (1.3,2.2) -- (1.2,1.5) -- cycle;
  \foreach \x/\y in {0.7/2.0, 1.1/2.6, 1.4/1.6, 0.8/2.9, 1.5/2.3, 0.9/1.3, 1.7/2.8, 0.5/2.4, 1.3/1.2, 1.0/1.8, 1.6/2.0, 0.6/1.7, 1.2/3.1, 1.8/1.4, 0.9/2.3}{
    \node[bluept] at (\x,\y) {};
  }
  \foreach \x/\y in {2.1/1.8, 2.5/2.4, 2.8/1.4, 2.0/2.7, 2.7/2.1, 2.3/1.1, 2.9/2.6, 1.9/2.2, 2.6/1.0, 2.2/1.6, 2.4/2.9, 3.0/1.7, 2.1/1.3, 2.8/2.4, 3.1/2.0}{
    \node[redpt] at (\x,\y) {};
  }
\end{scope}

\begin{scope}[shift={(2*\gapX,0)}]
  \node[font=\footnotesize\bfseries] at (0.5*\panelW, \panelH+0.4) {(c) Encircling};
  \draw[thick, rounded corners=2pt] (0,0) rectangle (\panelW,\panelH);
  \def\cx{2.1} \def\cy{1.9}
  \foreach \angle in {0,24,...,336}{
    \pgfmathsetmacro{\rx}{\cx+1.45*cos(\angle)}
    \pgfmathsetmacro{\ry}{\cy+1.45*sin(\angle)}
    \fill[red!18, opacity=0.5] (\rx,\ry) circle (0.5);
  }
  \fill[violet!35, opacity=0.65, even odd rule] (\cx,\cy) circle (1.1) (\cx,\cy) circle (0.72);
  \draw[violet!70!black, thick, opacity=0.9] (\cx,\cy) circle (1.1);
  \draw[violet!70!black, thick, opacity=0.9] (\cx,\cy) circle (0.72);
  \foreach \x/\y in {1.85/2.1, 2.3/1.9, 2.0/1.6, 2.1/2.3, 1.65/1.8, 2.35/2.15, 1.95/1.9, 2.2/1.5, 2.05/2.15, 1.75/2.0}{
    \fill[blue!18, opacity=0.55] (\x,\y) circle (0.38);
  }
  \foreach \x/\y in {1.85/2.1, 2.3/1.9, 2.0/1.6, 2.1/2.3, 1.65/1.8, 2.35/2.15, 1.95/1.9, 2.2/1.5, 2.05/2.15, 1.75/2.0}{
    \node[bluept] at (\x,\y) {};
  }
  \foreach \angle in {0,24,...,336}{
    \pgfmathsetmacro{\rx}{\cx+1.45*cos(\angle)}
    \pgfmathsetmacro{\ry}{\cy+1.45*sin(\angle)}
    \node[redpt] at (\rx,\ry) {};
  }
\end{scope}

\node[font=\footnotesize, violet!70!black] at (0.5*\panelW + \gapX, -0.5) {%
  Purple shading: $\mathcal{U}(X;r) \cap \mathcal{U}(Y;r)$, whose $\chi$ equals $\Delta\chi(r)$};

\def\plotShift{-4.6}

\begin{scope}[shift={(0,\plotShift)}]
  \begin{axis}[
    width=4.6cm, height=3.4cm,
    xmin=0, xmax=55, ymin=-0.3, ymax=2.5,
    xlabel={Scale $r$}, ylabel={$\Delta\chi(r)$},
    xlabel style={font=\scriptsize, at={(axis description cs:0.5,-0.22)}},
    ylabel style={font=\scriptsize},
    tick label style={font=\tiny},
    axis lines=left, ytick={0,1,2}, xtick={0,20,40}, clip=false,
  ]
    \addplot[const plot, fill=violet!15, draw=none, forget plot] coordinates {
      (0,0) (3,0) (3,2) (5,2) (5,1) (55,1) (55,0)} \closedcycle;
    \addplot[const plot, violet!80!black, very thick] coordinates {
      (0,0) (3,0) (3,2) (5,2) (5,1) (55,1)};
    \node[font=\scriptsize, violet!70!black] at (axis cs:28,2.4) {$\max|\Delta\chi| = 2$};
    \draw[-{Stealth[length=3pt]}, violet!60!black, thin] (axis cs:20,2.05) -- (axis cs:4.5,1.98);
  \end{axis}
\end{scope}

\begin{scope}[shift={(\gapX,\plotShift)}]
  \begin{axis}[
    width=4.6cm, height=3.4cm,
    xmin=0, xmax=280, ymin=-0.5, ymax=17,
    xlabel={Scale $r$}, ylabel={$\Delta\chi(r)$},
    xlabel style={font=\scriptsize, at={(axis description cs:0.5,-0.22)}},
    ylabel style={font=\scriptsize},
    tick label style={font=\tiny},
    axis lines=left, ytick={0,5,10,15}, xtick={0,100,200}, clip=false,
  ]
    \addplot[const plot, fill=violet!15, draw=none, forget plot] coordinates {
      (0,0) (1,0) (1,15) (4,15) (4,12) (8,12) (8,8) (12,8) (12,5) (20,5) (20,3) (35,3) (35,2) (60,2) (60,1) (280,1) (280,0)} \closedcycle;
    \addplot[const plot, violet!80!black, very thick] coordinates {
      (0,0) (1,0) (1,15) (4,15) (4,12) (8,12) (8,8) (12,8) (12,5) (20,5) (20,3) (35,3) (35,2) (60,2) (60,1) (280,1)};
    \node[font=\scriptsize, violet!70!black] at (axis cs:120,16.6) {$\max|\Delta\chi| = 15$};
    \draw[-{Stealth[length=3pt]}, violet!60!black, thin] (axis cs:80,15.6) -- (axis cs:6,14.9);
  \end{axis}
\end{scope}

\begin{scope}[shift={(2*\gapX,\plotShift)}]
  \begin{axis}[
    width=4.6cm, height=3.4cm,
    xmin=0, xmax=7, ymin=-0.5, ymax=10,
    xlabel={Scale $r$}, ylabel={$\Delta\chi(r)$},
    xlabel style={font=\scriptsize, at={(axis description cs:0.5,-0.22)}},
    ylabel style={font=\scriptsize},
    tick label style={font=\tiny},
    axis lines=left, ytick={0,2,4,6,8}, xtick={0,2,4,6}, clip=false,
  ]
    \addplot[const plot, fill=violet!15, draw=none, forget plot] coordinates {
      (0,0) (0.7,0) (0.7,3) (0.85,3) (0.85,8) (0.95,8) (0.95,4) (1.1,4) (1.1,1) (1.3,1) (1.3,0) (2.8,0) (2.8,1) (7,1) (7,0)} \closedcycle;
    \addplot[const plot, violet!80!black, very thick] coordinates {
      (0,0) (0.7,0) (0.7,3) (0.85,3) (0.85,8) (0.95,8) (0.95,4) (1.1,4) (1.1,1) (1.3,1) (1.3,0) (2.8,0) (2.8,1) (7,1)};
    \node[font=\scriptsize, violet!70!black] at (axis cs:4.5,9.5) {$\max|\Delta\chi| = 8$};
    \draw[-{Stealth[length=3pt]}, violet!60!black, thin] (axis cs:3.3,8.7) -- (axis cs:1.05,8.05);
  \end{axis}
\end{scope}

\end{tikzpicture}
\caption{Three configurations of point clouds $X$ (blue) and $Y$ (red) in $\R^2$, with ball unions shown as shaded disks. Purple shading marks the intersection $\mathcal{U}(X;r) \cap \mathcal{U}(Y;r)$, whose Euler characteristic equals $\Delta\chi(r)$ (Intersection Theorem). Bottom row: the Intersection ECP $\Delta\chi(r)$ for each configuration. The separated configuration~(a) has a small, late-onset signal; the overlapping~(b) has a large peak; the encircling~(c) has a moderate, topologically distinctive signal.}
\label{fig:motivating}
\end{figure}

We propose to take the interaction to be the \emph{Euler-characteristic defect}
\begin{equation}\label{eq:delta_chi_intro}
    \Delta\chi(r) = \chi_X(r) + \chi_Y(r) - \chi_{X \cup Y}(r) = \chi\!\left(\mathcal{U}(X; r) \cap \mathcal{U}(Y; r)\right),
\end{equation}
the Euler characteristic of the region where the thickened clouds overlap (the two expressions agree by the valuation property of $\chi$).
Figure~\ref{fig:motivating} shows three configurations---separated, overlapping, encircling---differing in the clouds' relative position, which single-cloud persistence does not record, but with qualitatively different $\Delta\chi$ profiles.
This single integer meets each of our requirements: it vanishes off the overlap (P1), returns $\chi(K)$ on a shared compact support $K$ of positive reach (P2), is symmetric and defined for all $k$ by replacing the two-fold intersection with a $k$-fold one (P3), is stable under perturbation---its intersection filtration is bottleneck $1$-Lipschitz in the cloudwise-maximum Hausdorff distance $\max_i d_H(X_i,X_i')$ (Theorem~\ref{thm:algebraic_stability}(1)) and the profile itself is $L^1$-stable (P4)---
and is computed by a single linear scan of an Alpha complex; being a curve rather than a diagram, it enters inference and learning with no vectorization step \citep{hacquard2024euler, dlotko2023euler} (P5).
Moreover these properties do not merely hold---they \emph{characterize} $\Delta\chi$: it is the unique top-floor pointwise-Euler interaction profile singled out by the Separation and Normalization axioms (Theorem~\ref{thm:universality}; Definition~\ref{def:floors}).
In particular, the construction resolves the canonical multi-cloud extension left open by \citet{wagner2024mixup} and supplies the stability and sampling theory that the chromatic framework does not; hypothesis testing and estimation for the invariant are left to future work.

This paper develops the foundations of the Intersection ECP.\footnote{Introduced under the name \emph{mixup Euler characteristic profile}, after the mixup barcode of \citet{wagner2024mixup}; we rename it for the intersection it computes, reserving \emph{interaction} for the phenomenon it measures. The companion applications \citep{majhi2026regime} rest on the same invariant.} The invariant is not an abstraction in search of a use. In the dynamical-systems companion \citep{majhi2026regime} it detects regime transitions, benchmarked on Indian monsoon onset and ENSO; it is designed equally for measuring class disentanglement in neural representations, where interaction must be read across many layers, epochs, and seeds at a scale for which the reduction-based mixup barcode is impractical. Neither setting is served by the mixup barcode or the chromatic six-pack: both are pairwise, reduction-heavy, and carry no hypothesis test (Section~\ref{sec:relation_invariants}). The present paper isolates the invariant those applications rest on and proves it is the right one.

The topological invariant in (P2) is not a free choice.
Modeling interaction as an inclusion--exclusion defect makes the underlying summary a \emph{valuation}, and by Hadwiger's theorem \citep{hadwiger1957vorlesungen, klain1997introduction} every continuous, rigid-motion-invariant, scale-invariant valuation on convex bodies is a multiple of the Euler characteristic, and the normalization $\chi(\mathrm{pt}) = 1$ singles out $\chi$ itself; so (P2) is forced---the invariant stabilizes to $\chi(K)$.
Two distinct canonicity statements thus reinforce each other: Hadwiger forces the \emph{summary} to be $\chi$, while the Universality Theorem (Theorem~\ref{thm:universality}) forces the way the $k$ clouds are \emph{combined} to be the product, together pinning the interaction profile to $\Delta\chi$. What a measurement task wants, then, is not the richest topological \emph{descriptor} but a \emph{measurement}---exactly recoverable, canonically forced, and cheap enough to sweep at population scale (Section~\ref{sec:contributions}).

A \emph{real}-valued invariant would instead measure the metric \emph{size} of the overlap; we take the topological extreme because its discreteness buys \emph{exact} sampling recovery and its homotopy-invariance  the universality above, and real-valued expressiveness re-enters only downstream, in the scalar summaries $\max_r |\Delta\chi(r)|$ and $\int |\Delta\chi(r)|\,dr$, the latter over a fixed finite scale window (since $\Delta\chi(r) \to 1$ at large scales, the unwindowed integral diverges).

That these desiderata are met all at once---rather than by a separate device for each---is no accident: it reflects the structure of \emph{Euler calculus}, the theory of integration with respect to the Euler characteristic, developed by Schapira~\citep{schapira1991operations}, Viro~\citep{viro1988some}, and Curry~\citep{curry2014sheaves}; see \citet{curry2012euler} for a comprehensive survey.
Euler calculus has experienced a renaissance in recent years: the Euler characteristic transform (ECT) of \citet{turner2014persistent} and \citet{curry2022directions} provides an injective signature for individual shapes; Lebovici's hybrid transforms \citep{lebovici2022hybrid} mix Euler and Lebesgue integration to produce Fourier- and Laplace-type invariants of constructible functions; Basu \citep{basu2015complexity} developed a complexity theory for constructible functions and sheaves using Euler integration; and \citet{hacquard2024euler} demonstrated the practical power of Euler characteristic profiles in supervised and unsupervised learning.
Closest to our motivation, the Euler characteristic curve---and its multiscale form, the \emph{Euler characteristic surface} \citep{roy2024ecs}---has proven a stable, cheap, and interpretable descriptor of a \emph{single} shape across applied problems, from drying-droplet and fluid dynamics \citep{roy2020droplets, roy2023characterizing} to multiphase flow-regime identification \citep{koenig2026churn} and time-series classification \citep{luwang2026interpretable}, with applications in unsupervised learning \cite{hacquard2024euler}.
The Intersection ECP is the \emph{interaction} analogue of this single-shape descriptor, and that applied track record motivates building its multi-cloud counterpart.
A unifying theme in this literature is that each invariant integrates a constructible function against a \emph{fixed} family of probes: the ECT uses half-spaces $\mathbf{1}_{H^-_{\nu,t}}$, and hybrid transforms use analytic kernels $K(\xi, \cdot)$.
In each case, one factor in the Euler integral is data-dependent and the other is a fixed geometric or analytic test function.
The present paper studies a different algebraic operation: the Euler integral of a product $\prod_{i=1}^k \mathbf{1}_{\mathcal{U}(X_i; t_i)}$ in which \emph{all $k$ factors are data-dependent}---encoding the mutual interaction of $k$ point clouds rather than the response of one shape to a fixed probe.
In this respect chromatic TDA \citep{cultrera2024chromatic} is the closest relative---it too takes the configuration itself, not a fixed probe, as its data.

The construction rests on a single, elementary observation.
The indicator $\mathbf{1}_{\mathcal{U}(X;r)}$ is a constructible function, pointwise multiplication corresponds to intersection ($\prod_i \mathbf{1}_{A_i} = \mathbf{1}_{\cap_i A_i}$), and the Euler integral $\int \mathbf{1}_A \, d\chi = \chi(A)$ extracts the topological summary; the identity $\Delta\chi = \chi(\text{intersection})$ is therefore a \emph{tautology}, not a calculation.
The Intersection ECP is the resulting compositional assignment $\mathcal{M} : \mathbf{PC}_k(\R^d) \to \mathbf{ZFun}_k$, invariant under rigid motions, scale-equivariant under bi-Lipschitz maps, and $L^1$-stable under Hausdorff perturbations (Lemmas~\ref{lem:rigid_invariance}--\ref{lem:scale_equivariance}, Theorem~\ref{thm:algebraic_stability}).
This compositional structure is not an abstraction for its own sake: the factorization through geometric $\cap$ and algebraic $\prod$ is exactly what makes the Intersection Theorem a \emph{commuting square} (Remark~\ref{rem:naturality}), the extension to $k$ clouds immediate (a $k$-fold product), stability a continuity property of each factor, and the construction canonical (Theorem~\ref{thm:universality}).

\paragraph{Why Euler calculus, and not merely inclusion--exclusion?} For $k = 2$ the identity $\Delta\chi = \chi(X \cap Y)$ is elementary, and one may ask what integration against $d\chi$ buys beyond it.
At the level of \emph{definition}, nothing: one may simply set $\Delta\chi(\mathbf{t}) = \chi(\bigcap_i \mathcal{U}(X_i; t_i))$ for every $k$, and the Euler integral $\int \prod_i \mathbf{1}_{\mathcal{U}(X_i;t_i)}\, d\chi$ equals it by the product formula.
The framework, not the identity, is what carries the paper.
Writing the invariant as $\int \varphi(\mathbf{1}_{\mathcal{U}(X_1;\cdot)}, \ldots, \mathbf{1}_{\mathcal{U}(X_k;\cdot)})\, d\chi$ for a pointwise $\varphi$ embeds $\Delta\chi$ in a whole \emph{family} of candidate interaction profiles, within which canonicity becomes a theorem: separation and normalization force $\varphi = \prod_i$ (Theorem~\ref{thm:universality}), and dropping even the Euler-integral hypothesis forces $\chi$ itself (Theorem~\ref{thm:universality_strong}).
A bare $\chi\big(\bigcap_i \mathcal{U}(X_i; t_i)\big)$ has no competitors and hence no such theorem---the integral form is what lets one ask ``why this scalar, and this combination, and not another.''
We are candid about the limits: the product encodes a \emph{modeling} choice---interaction $=$ $k$-fold intersection---which the calculus organizes but does not force, and much of the analytic theory runs on standard arguments, where Euler calculus supplies language rather than the technical engine.
Its indispensability is at the level of \emph{what the invariant is and why it is forced}.
And it is foundational only for the computable core $\Delta\chi$: the finer structure we develop---the relative-homology refinement, its interleaving stability (Theorem~\ref{thm:relative_stability}), and sampling consistency (Section~\ref{sec:consistency})---lives in a broader \v{C}ech/persistent-homology setting where Euler integration need not apply, and $\Delta\chi$ is its computable Euler shadow.

\paragraph{Two design axes.} The constructions of this paper are organized by two independent choices. \emph{(a)~Which set carries the interaction:} (a.i)~the full overlap $\bigcap_{i} \mathcal{U}(X_i; t_i)$; (a.ii)~refining it, the floors $C_j(\mathbf{t}) = \bigcup_{|S| = j} \bigcap_{i \in S} \mathcal{U}(X_i; t_i)$ of the interaction spectrum, with their pairs $(C_j, C_{j+1})$; or (a.iii)~the leave-one-out union pairs $\big(\bigcup_i \mathcal{U}(X_i; t_i),\ \bigcup_{i \neq j} \mathcal{U}(X_i; t_i)\big)$. All agree for $k = 2$. \emph{(b)~How that set is measured:} (b.i)~by the Euler characteristic, or (b.ii)~by persistent homology. Every combination appears below: (a.i)$+$(b.i) is the Intersection ECP itself, the computable core (Sections~\ref{sec:theory} and~\ref{sec:algorithm}); (a.i)$+$(b.ii) is the stability theory of the intersection filtration (Theorem~\ref{thm:algebraic_stability}); (a.ii)$+$(b.i) is the interaction spectrum (Definition~\ref{def:floors}); (a.iii)$+$(b.ii) is the relative interaction module of Section~\ref{sec:relative}; and the sampling theory of Section~\ref{sec:consistency} runs along both axes. The Euler-characteristic column maximizes computability (P5) { at the price of \emph{cancellation}---balanced Betti numbers can cancel in the alternating sum $\chi = \sum_q (-1)^q \beta_q$, so the integer profile can miss interaction that the persistent-homology column, detecting it degree by degree at greater cost, still sees.} The refinement tower of Section~\ref{sec:relative} makes the trade explicit.

\subsection{Our Contributions}\label{sec:contributions}

With the Euler calculus as the organizing language for the computable invariant $\Delta\chi$, the construction factors as
\[
\mathbf{PC}_k(\R^d) \xrightarrow{\;\mathcal{U}\;} \mathbf{CF}(\R^d)^k \xrightarrow{\;\prod\;} \mathbf{CF}(\R^d) \xrightarrow{\;\textstyle\int d\chi\;} \Z,
\]
offset (geometry) $\to$ product (interaction) $\to$ integration (topology), and this factorization is the source of the main results developed below:

\begin{itemize}
    \item \textbf{Intersection Theorem} (Section~\ref{sec:intersection}): $\Delta\chi(\mathbf{t}) = \chi(\bigcap_i \mathcal{U}(X_i; t_i))$, a \emph{commuting square} in which geometric $\cap$ and algebraic $\prod$ interchange under $\int d\chi$.
    \item \textbf{$k$-cloud generality}: the product $\prod_i \mathbf{1}_{A_i}$ is defined for every $k$, so the Intersection ECP handles $k$ clouds with no extra machinery or choice of order.
    \item \textbf{Universality} (Section~\ref{sec:intersection}): for every $k$, $\Delta\chi$ is the \emph{unique} pointwise-Euler profile fixed by separation and normalization (Theorem~\ref{thm:universality}), and $\chi$ itself is forced by a Hadwiger-type multivaluation argument (Theorem~\ref{thm:universality_strong}); each floor $\chi(C_j)$ of the spectrum is characterized likewise (Proposition~\ref{prop:pairwise}).
    \item \textbf{Analytic and geometric properties} (Section~\ref{sec:props}): a dead zone and elementary properties, a Mayer--Vietoris Betti decomposition of the overlap, what $\Delta\chi$ computes for dense samples, geometric bounds---component counts by zone covering numbers, with a matched fractal growth rate---and multiparameter stability of the profile surface (Proposition~\ref{prop:multiparameter_stability}).
    \item \textbf{Algorithm and complexity} (Section~\ref{sec:algorithm}): a single sorted Alpha-complex sweep computes $\Delta\chi$ in $O(n^{\lceil d/2\rceil}\log n)$ { time, where $n = \sum_i |X_i|$ is the total number of sample points,} with \emph{no} persistence reduction---cheaper than the mixup barcode and, for even $d$, worst-case optimal among filtration-materializing algorithms (Proposition~\ref{prop:hardness}).
    \item \textbf{Relative-homology refinement} (Section~\ref{sec:relative}): where $\chi$ cancels, the symmetrized relative interaction module $\mathrm{RPH}_*$ (the sum of the $k$ leave-one-out modules) resolves the finer structure faithfully and is interleaving-stable (Theorem~\ref{thm:relative_stability}); $\Delta\chi$ is the top floor of the interaction spectrum (Definition~\ref{def:floors}), the whole spectrum equally stable (Proposition~\ref{prop:spectrum_stability}).
    \item \textbf{Consistency under sampling} (Section~\ref{sec:consistency}): as samples densify, the invariant recovers the { ground} truth---qualitatively for arbitrary compact sets in the inverse limit, and quantitatively under bounded reach, persistently at every regular scale and exactly (pointwise) at scales regular for the sample, with explicit sample-complexity bounds.
\end{itemize}

\subsection{Outline}\label{sec:outline}

Section~\ref{sec:theory} constructs the Intersection ECP: the data and algebraic stage, the Euler integral, the construction itself, and the Intersection Theorem.
Section~\ref{sec:props} establishes its analytic and geometric properties; Section~\ref{sec:relative} develops the relative-homology refinement that recovers what $\Delta\chi$ compresses; Section~\ref{sec:consistency} shows the invariant is consistent under sampling; Section~\ref{sec:algorithm} gives the computational realizations, the algorithm, and the complexity analysis; and Section~\ref{sec:discussion} discusses limitations and open problems.

\subsection{Relation to other interaction invariants}\label{sec:relation_invariants}

Two prior invariants also read the topological interaction of colored clouds: the \emph{mixup barcode} of \citet{wagner2024mixup}---the image persistence of the inclusion $\mathcal{U}(Y;r) \hookrightarrow \mathcal{U}(X\cup Y;r)$---and the \emph{chromatic six-pack} \citep{cultrera2024chromatic}, six persistence diagrams read off a chromatic Delaunay complex. Both retain graded homological detail that $\Delta\chi$ collapses into a single alternating sum, and it is worth saying plainly why one would nonetheless want the Euler profile.
The reason is that they answer different questions. $\Delta\chi$ is a deliberate coarsening---one integer per scale---and the coarsening is strict at each level.

\begin{proposition}[Strict coarsening]\label{prop:hierarchy}
For disjoint pairs $(X, Y)$, each reduction in the chain
\[
    \{\text{graded overlap homology } H_*(\mathcal{U}(X;r)\cap\mathcal{U}(Y;s))\}
    \;\longrightarrow\; \Delta\chi(\cdot)
    \;\longrightarrow\; \max_r |\Delta\chi(r)|
\]
strictly loses information:
\begin{enumerate}
    \item \emph{Overlap homology $\succ$ profile.} The interacting and non-interacting pairs of Figure~\ref{fig:saturation} have \emph{identical} profiles $\Delta\chi \equiv 0$ across a whole band, yet their overlap Betti curves differ---one carries the annulus (resp.\ torus-shell) homology, the other none. The Euler alternating sum forgets the balanced Betti numbers those curves resolve.
    \item \emph{Profile $\succ$ peak scalar.} Rescaling a configuration by $\lambda > 0$ leaves $\max_r |\Delta\chi|$ unchanged (scale-invariance, Lemma~\ref{lem:scale_equivariance}) while dilating the profile along $r$; two configurations can thus share the peak statistic yet have different profiles.
\end{enumerate}
(Neither the mixup barcode nor the six-pack sits on this chain: each keeps graded homology $\Delta\chi$ discards, but neither determines $\Delta\chi$, which needs $\chi(\mathcal{U}(X))$ of one cloud alone---so the comparison with them is by \emph{properties}, below, not by refinement.)
\end{proposition}

So $\Delta\chi$ is a genuine coarsening. What it buys for the coarsening is a bundle of properties that the finer invariants do \emph{not} possess at once {(in the table, $N = O(n^{\lceil d/2\rceil})$ denotes the total number of simplices in the underlying complex, so the reduction-based invariants' $O(N^3)$ cost is cubic in that count)}:
\begin{center}
\renewcommand{\arraystretch}{1.2}
\begin{tabular}{@{}lccc@{}}
\toprule
 & \textbf{Intersection ECP} $\Delta\chi$ & Mixup barcode & Chromatic six-pack \\
\midrule
Symmetric in the clouds        & yes & no ($L \hookrightarrow K$) & yes \\
Native to $k > 2$              & yes & no (open) & colored only \\
Reduction-free to compute      & yes (signed count) & no ($O(N^3)$) & no ($O(N^3)$) \\
{Exact integer readout}\footnotemark   & {yes (signed count)} & {no (real diagram)} & {no (real diagram)} \\
Axiomatically forced           & yes (Thms~\ref{thm:universality},~\ref{thm:universality_strong}) & no & no \\
Inference-ready (no vectorization) & yes (scalar curve) & no (diagram) & no (diagram) \\
\bottomrule
\end{tabular}
\end{center}
\footnotetext{ This row is not a deficiency of the diagram invariants at their own task---recovering a stable summary of the interaction---but a structural difference: $\Delta\chi$ is integer-valued, so at scales regular for the sample it returns the exact target value (Section~\ref{sec:consistency}), whereas a real-valued persistence diagram is guaranteed only Lipschitz-continuously, never exactly.}
Every richer invariant forfeits at least one row. The mixup barcode is asymmetric, pairwise, reduction-bound, and untested; the six-pack is symmetric but tied to the Euclidean chromatic Delaunay construction, six times as large, and likewise carries no stability, sampling, or inference theory for the interaction. $\Delta\chi$ is the \emph{richest} interaction invariant compatible with the whole bundle: an integer-valued, scale-indexed \emph{curve} one can do exact recovery and statistics on directly, forced by the axioms of Section~\ref{sec:intersection}.

This is decisive precisely in the measurement regime. Because $\Delta\chi$ needs no boundary-matrix reduction, it scales as $O(n^{\lceil d/2\rceil}\log n)$ rather than cubically in the simplex count---the difference between feasible and infeasible when a study must sweep many measurements across layers, epochs, and seeds. And being defined for every $k$, it can \emph{pose} questions the pairwise invariants cannot---whether three classes are jointly entangled beyond what their pairs already show.

We do not pretend the coarsening is free. Where it costs---an overlap whose Betti numbers cancel---the relative-homology refinement of Section~\ref{sec:relative} recovers { most of} what $\chi$ discards, at the price of the computational simplicity that motivated $\Delta\chi$ in the first place. The profile is the shipped invariant; the refinement is its principled escalation.

\section{Construction of the Intersection ECP}\label{sec:theory}

This section defines the Intersection ECP through constructible functions and the Euler calculus.
We set up the data and the algebraic stage through which the construction factors (Sections~\ref{sec:categories}--\ref{sec:euler_integral}), define the Intersection ECP and prove the Intersection Theorem (Section~\ref{sec:mixup_def}), and display the theorem as a commuting square (Section~\ref{sec:intersection}). The construction accommodates an arbitrary number $k$ of point clouds; the two-cloud case ($k = 2$) is a specialization.

\subsection{Data and Algebraic Stage}\label{sec:categories}

We fix the data the construction acts on, the algebraic stage through which it factors, and the space in which its output lives.
The Intersection ECP is invariant under rigid motions and scale-equivariant and $L^1$-stable under bi-Lipschitz maps (Lemma~\ref{lem:rigid_invariance}, Lemma~\ref{lem:scale_equivariance}, Theorem~\ref{thm:algebraic_stability}).

\paragraph{ Homology conventions.} { Throughout the paper we fix a field $\mathbb{F}$, and all homology groups are taken with coefficients in $\mathbb{F}$. We fix a field for two reasons: it makes the regularity condition of Definition~\ref{def:regular_window} the weakest that suffices---an integral-homology condition would additionally forbid torsion changes, which $\chi$ never sees\footnote{Though which windows qualify can depend on the choice of $\mathbb{F}$.}---and it matches the barcode notion of critical scale used in Theorem~\ref{thm:algebraic_stability}(2), where a critical scale is a bar endpoint and interval decompositions exist only over a field. Every conclusion about $\Delta\chi$ is nonetheless coefficient-free, $\chi$ being the alternating Betti sum over every field.}

\begin{definition}[Disjoint point-cloud tuples and their maps]\label{def:cat_pcpairs}
Let $\mathbf{PC}_k(\R^d)$ denote the collection of $k$-tuples $\mathcal{X} = (X_1, \ldots, X_k)$ of pairwise disjoint\footnote{Disjointness is a normalization, not a hypothesis---the Intersection Theorem (Theorem~\ref{thm:intersection}) and the algebraic formula hold for arbitrary finite point sets.
We impose it so that interaction vanishes at scale zero ($\Delta\chi(\mathbf{0}) = \chi(\bigcap_i X_i) = 0$, i.e.\ interaction is geometric, not an artifact of shared points), the minimum cross-cloud distance is positive, and the cloudwise map condition below is well posed.
Relaxing it changes only the zero baseline to $\Delta\chi(\mathbf{0}) = |\bigcap_i X_i|$.} finite subsets of $\R^d$.
We compare two tuples $\mathcal{X}, \mathcal{X}'$ through ambient maps that match the clouds cloudwise: a homeomorphism $f : \R^d \to \R^d$ with $f(X_i) = X'_i$ for each $i$.
We use two classes:
\begin{itemize}
    \item \textbf{Rigid motions} ($C = 1$): $\mathcal{X}'$ is a rigid motion of $\mathcal{X}$, i.e.\ $f$ is an isometry.
    \item \textbf{Ambient $C$-bi-Lipschitz homeomorphisms}: there exists $C \geq 1$ with
    \[
        C^{-1}\|x - y\| \;\leq\; \|f(x) - f(y)\| \;\leq\; C\|x - y\| \qquad \text{for all } x, y \in \R^d.
    \]
\end{itemize}
\end{definition}

The single global constant $C$ is what matters: it controls both the within-cloud distances ($x, y$ in the same $X_i$) and the \emph{cross-cloud} distances ($x \in X_i$, $y \in X_j$, $i \neq j$) that determine when ball unions meet, and hence the relative geometry governing interaction.
We work with the ambient map throughout; its restriction $f|_{\sqcup_i X_i}$ recovers a cloudwise description.

\begin{definition}[The algebraic stage: constructible functions]\label{def:cat_cf}
Let $\mathbf{CF}(\R^d)$ be the commutative ring (non unital, since
$\mathbf{1}_{\R^d}$ is not compactly supported) of constructible functions $\varphi : \R^d \to \Z$---finite $\Z$-linear combinations of indicators of compact semi-algebraic sets---under pointwise addition and multiplication.
\end{definition}

It is the intermediate stage of the factorization~\eqref{eq:factorization}, entered through the offset map and exited through the Euler integral.
Two features of its structure are used:
\begin{itemize}
    \item The product step $\prod$ multiplies indicators, and the product of indicators is the indicator of the intersection, $\mathbf{1}_A \cdot \mathbf{1}_B = \mathbf{1}_{A \cap B}$; this is the idempotent set $\{\varphi : \varphi^2 = \varphi\}$ where the offset map $\mathcal{U}$ lands.
    \item The Euler integral $\int d\chi : (\mathbf{CF}(\R^d), +) \to (\Z, +)$ is a homomorphism of additive groups (Section~\ref{sec:euler_integral}) but \emph{not} order-preserving for the pointwise order $\varphi \leq \psi$: e.g.\ $\mathbf{1}_{\{0,1\}} \leq \mathbf{1}_{[0,1]}$ yet $\int d\chi$ gives $2 > 1$.
\end{itemize}

All sets arising below are finite intersections of finite unions of closed balls, hence compact semi-algebraic; we use the algebraic stage in the standard generality since the sampling theory of Section 5 works with arbitrary compact (semi-algebraic) sets.

\begin{definition}[Target: integer-valued profiles with the $L^1$ metric]\label{def:cat_bifunc}
Let $\mathbf{ZFun}_k$ be the collection of piecewise-constant integer-valued profiles $\R_{\geq 0}^k \to \Z$, equipped with the pointwise partial order $\leq$ and the $L^1$ distance $d_{L^1}$ of Definition~\ref{def:l1_distance}, with respect to which the \emph{diagonal} Intersection ECP is stable (Theorem~\ref{thm:algebraic_stability}).
\end{definition}

\subsection{The Euler Integral as a Valuation}\label{sec:euler_integral}

The bridge between the geometric (point cloud) and algebraic (integer-valued profile) worlds is the Euler integral; see \citet{schapira1991operations}, \citet{curry2014sheaves}, and the survey \citet{curry2012euler} for comprehensive treatments.

\begin{definition}[Euler integral]\label{def:euler_integral}
The \textbf{Euler integral} is the $\Z$-linear functional
\[
    \int d\chi : (\mathbf{CF}(\R^d),\, +) \;\to\; (\Z,\, +)
\]
defined by $\int \mathbf{1}_A \, d\chi = \chi(A)$ for compact semi-algebraic $A$, and extended $\Z$-linearly: $\int \sum_i c_i \mathbf{1}_{A_i} \, d\chi = \sum_i c_i \chi(A_i)$.
\end{definition}

This extension is well-defined.
On the \emph{lattice} of compact semi-algebraic sets (closed under finite $\cup$ and $\cap$, but not under complement, so not a Boolean algebra), $\chi$ is a \emph{valuation} \citep{viro1988some}: it satisfies the inclusion--exclusion identity $\chi(A \cup B) = \chi(A) + \chi(B) - \chi(A \cap B)$, which holds by compatible semi-algebraic triangulation.
The relations among indicators of such sets are generated precisely by $\mathbf{1}_{A \cup B} + \mathbf{1}_{A \cap B} = \mathbf{1}_A + \mathbf{1}_B$ (together with $\mathbf{1}_\emptyset = 0$), which $\chi$ respects; hence the $\Z$-linear extension is independent of the representation chosen for a constructible function (the valuation extension theorem; see \citep{schapira1991operations, curry2012euler}).
Compactness is essential: on compact semi-algebraic sets the ordinary Euler characteristic agrees with its compactly supported version, for which inclusion--exclusion holds, whereas the identity fails for the ordinary Euler characteristic of non-compact semi-algebraic sets (an open interval has $\chi = 1$ but compactly supported Euler characteristic $-1$).

A first worked evaluation, for compact semi-algebraic $A$ and $B$: although $B^c$ is not compact, the product $\mathbf{1}_A \mathbf{1}_{B^c} = \mathbf{1}_A - \mathbf{1}_{A \cap B}$ is constructible, and
\[
    \int \mathbf{1}_A \mathbf{1}_{B^c}\, d\chi \;=\; \chi(A) - \chi(A \cap B),
\]
the part of $A$'s Euler characteristic not accounted for by its overlap with $B$.

Although $\int d\chi$ is \emph{not} a ring homomorphism for pointwise multiplication,\footnote{$\int \mathbf{1}_A \cdot \mathbf{1}_B \, d\chi = \chi(A \cap B) \neq \chi(A) \cdot \chi(B)$ in general: e.g.\ if $A = B$ is a two-point set, then $\chi(A) \cdot \chi(A) = 4$ but $\chi(A \cap A) = 2$.} the following product identity holds.

\begin{proposition}[Euler product formula]\label{prop:euler_product}
For indicator constructible functions $\mathbf{1}_{A_1}, \ldots, \mathbf{1}_{A_k}$ with each $A_i$ compact and semi-algebraic:
\begin{equation}\label{eq:euler_product}
    \int_{\R^d} \prod_{i=1}^{k} \mathbf{1}_{A_i} \, d\chi = \chi\!\left(\bigcap_{i=1}^{k} A_i\right).
\end{equation}
\end{proposition}

\begin{proof}
Pointwise $\prod_i \mathbf{1}_{A_i} = \mathbf{1}_{\bigcap_i A_i}$, so $\int \prod_i \mathbf{1}_{A_i}\, d\chi = \chi(\bigcap_i A_i)$ by definition of the Euler integral.
\end{proof}

Despite its elementary proof, this formula is the engine of the theory: it converts pointwise multiplication in $\mathbf{CF}(\R^d)$ into set intersection under the Euler integral, and the Intersection Theorem (Theorem~\ref{thm:intersection}) is a direct consequence.

\subsection{The Intersection ECP}\label{sec:mixup_def}

We now define the central construction.

Given a finite point cloud $\mathcal{P} \subset \R^d$ and scale $t \geq 0$, the ball union $\mathcal{U}(\mathcal{P}; t) = \bigcup_{p \in \mathcal{P}} B(p, t)$ defines an element $\mathbf{1}_{\mathcal{U}(\mathcal{P}; t)} \in \mathbf{CF}(\R^d)$, giving a scale-parametrized constructible function $\mathbf{1}_{\mathcal{U}(\mathcal{P}; \cdot)}$ associated to $\mathcal{P}$.

\begin{definition}[Intersection Euler Characteristic Profile]\label{def:mixup_ec}
The \textbf{Intersection Euler Characteristic Profile} (Intersection ECP) is the assignment $\mathcal{M}$ that sends each tuple $\mathcal{X} = (X_1, \ldots, X_k) \in \mathbf{PC}_k(\R^d)$ to the function $\mathcal{M}(\mathcal{X}) : \R_{\geq 0}^k \to \Z$ defined by
\begin{equation}\label{eq:delta_chi}
    \mathcal{M}(\mathcal{X})(\mathbf{t}) \;=\; \Delta\chi(\mathbf{t};\, \mathcal{X}) \;=\; \int_{\R^d} \prod_{i=1}^{k} \mathbf{1}_{\mathcal{U}(X_i;\, t_i)} \, d\chi,
\end{equation}
where $\mathbf{t} = (t_1, \ldots, t_k) \in \R_{\geq 0}^k$.
\end{definition}

We take this profile as \emph{the} Intersection ECP.
It is the top floor $\chi(C_k)$ of the interaction spectrum $\{\chi(C_j)\}_{j=2}^{k}$; the full spectrum is a canonical and equally stable refinement, developed in Definition~\ref{def:floors} and recommended when finer $k$-cloud resolution ($k\ge3$) is wanted.
For $k=2$ the spectrum has a single floor, so it coincides with {$\Delta\chi = \mathcal{M}(\mathcal{X})(\mathbf{t})$}.

The assignment factors through three compositional steps:
\begin{equation}\label{eq:factorization}
    \mathbf{PC}_k(\R^d) \xrightarrow{\;\mathcal{U}\;} \mathbf{CF}(\R^d)^k \xrightarrow{\;\prod\;} \mathbf{CF}(\R^d) \xrightarrow{\;\textstyle\int d\chi\;} \Z,
\end{equation}
where $\mathcal{U}(\mathcal{X}) = (\mathbf{1}_{\mathcal{U}(X_1;\cdot)}, \ldots, \mathbf{1}_{\mathcal{U}(X_k;\cdot)})$ is the \textbf{offset map} (assigning scale-parametrized indicator functions), $\prod$ is pointwise multiplication in $\mathbf{CF}(\R^d)$, and $\int d\chi$ is the Euler integral.

The factorization~\eqref{eq:factorization} separates three conceptual steps: the offset map $\mathcal{U}$ encodes the \emph{geometry} of each cloud at each scale, the product $\prod$ the \emph{interaction} between clouds, and the Euler integral $\int d\chi$ the \emph{topological summary}.
Each is invariant under rigid motions and metrically continuous under bi-Lipschitz maps, so their composite $\mathcal{M}$ is rigid-motion invariant (Lemma~\ref{lem:rigid_invariance}), scale-equivariant under bi-Lipschitz homeomorphisms (Lemma~\ref{lem:scale_equivariance}), and $L^1$-stable under Hausdorff perturbations (Theorem~\ref{thm:algebraic_stability}).

The name \emph{Euler Characteristic Profile} also appears in \citet{dlotko2023euler}, denoting $\mathbf{p} \mapsto \chi(K_{\mathbf{p}})$ for a multifiltration of a \emph{single} complex $K$---an \emph{intrinsic} multiparameter structure.
Our Intersection ECP is different in kind: its $k$ parameters are independent scales for $k$ \emph{separate} clouds, an \emph{extrinsic} structure measuring the interaction \emph{between} datasets rather than the topology of one.

The central identity is available at once, so we state it here; we display it as a commuting square in Section~\ref{sec:intersection}.

\begin{theorem}[Multiparameter Intersection Theorem]\label{thm:intersection}
For pairwise disjoint finite point clouds $X_1, \ldots, X_k \subset \R^d$ and any $\mathbf{t} \in \R_{\geq 0}^k$:
\begin{equation}\label{eq:intersection}
    \Delta\chi(\mathbf{t};\, X_1, \ldots, X_k) = \chi\!\left(\bigcap_{i=1}^{k} \mathcal{U}(X_i;\, t_i)\right).
\end{equation}
\end{theorem}

\begin{proof}
Apply the Euler product formula (Proposition~\ref{prop:euler_product}) to Definition~\ref{def:mixup_ec}:
\[
    \Delta\chi(\mathbf{t}) = \int_{\R^d} \prod_{i=1}^{k} \mathbf{1}_{\mathcal{U}(X_i;\, t_i)} \, d\chi = \chi\!\left(\bigcap_{i=1}^{k} \mathcal{U}(X_i;\, t_i)\right). \qedhere
\]
\end{proof}

\begin{lemma}[Rigid-motion invariance]\label{lem:rigid_invariance}
If $\mathcal{X}'$ is a rigid motion of $\mathcal{X}$ ({the image under} an ambient isometry $f$ with $f(X_i) = X'_i$), then $\mathcal{M}(\mathcal{X}) = \mathcal{M}(\mathcal{X}')$: the Intersection ECP depends only on the rigid-motion class of $\mathcal{X}$.
\end{lemma}

\begin{proof}
An isometry $f$ maps each ball $B(x, t_i)$ to $B(f(x), t_i)$, hence carries $\mathcal{U}(X_i; t_i)$ onto $\mathcal{U}(X'_i; t_i)$ and, being a bijection, $\bigcap_i \mathcal{U}(X_i; t_i)$ onto $\bigcap_i \mathcal{U}(X'_i; t_i)$.
Since $\chi$ is a homeomorphism invariant, $\Delta\chi(\mathbf{t}; \mathcal{X}) = \chi(\bigcap_i \mathcal{U}(X_i; t_i)) = \chi(\bigcap_i \mathcal{U}(X'_i; t_i)) = \Delta\chi(\mathbf{t}; \mathcal{X}')$ for every $\mathbf{t}$, using the Intersection Theorem (Theorem~\ref{thm:intersection}).
\end{proof}

The next lemma, and the stability theory below (Theorem~\ref{thm:algebraic_stability}), need a name for the scale regions on which the topology of the overlap does not move.

\begin{definition}[Window free of homologically critical scales]\label{def:regular_window}
Let $\mathcal{X} \in \mathbf{PC}_k(\R^d)$. A window $W = \prod_i [a_i, b_i] \subset \R_{\geq 0}^k$ is \textbf{free of homologically critical scales} of $\mathcal{X}$ if for all $\mathbf{s}, \mathbf{t} \in W$ with $s_i \leq t_i$ for every $i$, the inclusion
\[
    \bigcap_{i} \mathcal{U}(X_i;\, s_i) \;\hookrightarrow\; \bigcap_{i} \mathcal{U}(X_i;\, t_i)
\]
induces isomorphisms on homology {(over the fixed field $\mathbb{F}$)} in every degree.
\end{definition}

On such a window all Betti numbers of the overlap, and hence $\Delta\chi$, are constant.
{ Dually, a \textbf{critical scale} of $\mathcal{X}$ is a homological critical value of the intersection filtration $\{\bigcap_i \mathcal{U}(X_i; t)\}_t$---a birth or death scale of some bar, i.e.\ a scale where the field homology of the intersection changes---\emph{not} merely a scale at which the integer $\Delta\chi$ jumps; a simultaneous $\beta_1$-death and $\beta_3$-birth, say, is critical even though $\chi$ is unchanged.}
The condition concerns the \emph{intersection filtration} of $\mathcal{X}$, not the profile $\mathcal{M}(\mathcal{X})$: the profile can be constant across a scale at which the homology moves.

We write $\mathrm{PH}_*(\mathcal{X})$ for the persistence module $\{H_*(\bigcap_i \mathcal{U}(X_i; t_i);\, \mathbb{F})\}_{\mathbf{t}}$ of this intersection filtration, and $d_I$ for the interleaving distance on persistence modules; on a one-parameter module---in particular the diagonal $t_1 = \cdots = t_k$---$d_I$ coincides with the bottleneck distance of barcodes by the isometry theorem \citep{chazal2016structure}. We use $\mathrm{PH}_*$ and $d_I$ from here on, including for the relative modules of Section~\ref{sec:relative}. Both attach to the \emph{module}, not to the integer profile: $\Delta\chi$ itself admits no interleaving distance (Remark~\ref{rem:interleaving_comparison}).

\begin{lemma}[Scale-equivariance under bi-Lipschitz maps]\label{lem:scale_equivariance}
Let $f : \R^d \to \R^d$ be an ambient $C$-bi-Lipschitz homeomorphism with $f(X_i) = X'_i$.
Then for every $\mathbf{t}$,
\begin{equation}\label{eq:functor_morphism}
    \bigcap_{i} \mathcal{U}(X'_i;\, C^{-1}t_i) \;\subseteq\; f\!\Big(\bigcap_{i} \mathcal{U}(X_i;\, t_i)\Big) \;\subseteq\; \bigcap_{i} \mathcal{U}(X'_i;\, C\,t_i),
\end{equation}
whose middle set is homeomorphic to $\bigcap_i \mathcal{U}(X_i; t_i)$. If moreover the window $\prod_i[C^{-2}t_i,\, C^{2}t_i]$ is free of homologically critical scales of both $\mathcal{X}$ and $\mathcal{X}'$ (Definition~\ref{def:regular_window}), then
\[
    \Delta\chi(C^{-1}\mathbf{t};\, \mathcal{X}') \;=\; \Delta\chi(\mathbf{t};\, \mathcal{X}) \;=\; \Delta\chi(C\mathbf{t};\, \mathcal{X}').
\]

\end{lemma}

\begin{proof}
For a $C$-bi-Lipschitz homeomorphism $f$ and any $x$, $B(f(x), C^{-1}r) \subseteq f(B(x,r)) \subseteq B(f(x), Cr)$.
Taking unions over $x \in X_i$ and using $f(X_i) = X'_i$ gives $\mathcal{U}(X'_i; C^{-1}t_i) \subseteq f(\mathcal{U}(X_i; t_i)) \subseteq \mathcal{U}(X'_i; C t_i)$; intersecting over $i$ and using $f(\bigcap_i A_i) = \bigcap_i f(A_i)$ (as $f$ is a bijection) yields \eqref{eq:functor_morphism}.
Write $F_{\mathbf{s}} = \bigcap_i \mathcal{U}(X_i;\, s_i)$ and $G_{\mathbf{s}} = \bigcap_i \mathcal{U}(X'_i;\, s_i)$.
Since $f^{-1}$ is also $C$-bi-Lipschitz with $f^{-1}(X'_i) = X_i$, applying \eqref{eq:functor_morphism} to $f^{-1}$ at the scales $C^{\pm 1}\mathbf{t}$ extends the sandwich to the interleaved tower
\[
    f(F_{C^{-2}\mathbf{t}}) \;\subseteq\; G_{C^{-1}\mathbf{t}} \;\subseteq\; f(F_{\mathbf{t}}) \;\subseteq\; G_{C\mathbf{t}} \;\subseteq\; f(F_{C^{2}\mathbf{t}}).
\]
On $H_*(-;\mathbb{F})$ the window hypotheses make the outer composites isomorphisms: for $\mathcal{X}$ the two composites along the $f$-images (which agree with those of $F_{C^{-2}\mathbf{t}} \subseteq F_{\mathbf{t}} \subseteq F_{C^{2}\mathbf{t}}$ up to conjugation by $f_*$), and for $\mathcal{X}'$ the composite $G_{C^{-1}\mathbf{t}} \to G_{C\mathbf{t}}$.
By the standard interleaving argument, the two inner inclusions $G_{C^{-1}\mathbf{t}} \subseteq f(F_{\mathbf{t}}) \subseteq G_{C\mathbf{t}}$ induce isomorphisms on $H_*$.
Hence $F_{\mathbf{t}}$, $G_{C^{-1}\mathbf{t}}$, and $G_{C\mathbf{t}}$ have identical Betti numbers, and alternating sums with the Intersection Theorem (Theorem~\ref{thm:intersection}) give the displayed equalities.

For a homothety $f(x) = \lambda x$ ($C = \lambda$) the containment~\eqref{eq:functor_morphism} is an equality---$\lambda B(x,r) = B(\lambda x, \lambda r)$, so $f$ carries $\bigcap_i \mathcal{U}(X_i;t_i)$ \emph{onto} $\bigcap_i \mathcal{U}(\lambda X_i; \lambda t_i)$---and the conclusion $\Delta\chi(\lambda\mathbf{t}; \lambda\mathcal{X}) = \Delta\chi(\mathbf{t};\mathcal{X})$ holds with no window hypothesis.
\end{proof}

\begin{definition}[$L^1$ distance on profiles]\label{def:l1_distance}
For integer-valued profiles $F, G : \R_{\geq 0}^k \to \Z$ that agree outside a bounded set, the \textbf{$L^1$ distance} is
\[
    d_{L^1}(F, G) = \int_{\R_{\geq 0}^k} |F(\mathbf{t}) - G(\mathbf{t})| \, d\mathbf{t}.
\]
\end{definition}

For the \emph{diagonal} profiles $r \mapsto \Delta\chi(r)$ and $r \mapsto \Delta\chi'(r)$, both equal $\chi(\R^d\text{-ball}) = 1$ once $r$ is large ({Lemma~\ref{cor:dead_zone}(3)}, all $t_i = r \to \infty$ together), so its difference has bounded support and the diagonal $d_{L^1}$ is finite. Along a mixed ray---one coordinate held small while another grows---the profile need not tend to $1$, so the unwindowed integral over $\R_{\geq 0}^k$ is generally infinite; accordingly the $L^1$ theory is stated on the diagonal (Theorem~\ref{thm:algebraic_stability}) and, off it, on bounded windows (Proposition~\ref{prop:multiparameter_stability}).
We do \emph{not} define an interleaving distance directly on the profile $\Delta\chi$: the interleaving distance $d_I$ is a metric on persistence \emph{modules}---it compares $\mathrm{PH}_*(\mathcal{X})$ with $\mathrm{PH}_*(\mathcal{X}')$, not the integer profiles (see Remark~\ref{rem:interleaving_comparison}).

Recall that two barcodes are within \emph{bottleneck distance} $\epsilon$ if there is a bijective correspondence between their bars of persistence $> 2\epsilon$ moving every endpoint by at most $\epsilon$, the bars of persistence $\leq 2\epsilon$ left unconstrained.

\begin{theorem}[Stability of the Intersection ECP]\label{thm:algebraic_stability}
Let $\mathcal{X}, \mathcal{X}' \in \mathbf{PC}_k(\R^d)$ with $\epsilon = \max_i d_H(X_i, X'_i)$, and write $\Delta\chi = \Delta\chi(\cdot;\mathcal{X})$, $\Delta\chi' = \Delta\chi(\cdot;\mathcal{X}')$ for the diagonal profiles.
\begin{enumerate}
    \item \textbf{Interleaved filtrations.} For all $r \geq \epsilon$,
    \[
        \textstyle\bigcap_i \mathcal{U}(X'_i; r-\epsilon) \;\subseteq\; \bigcap_i \mathcal{U}(X_i; r) \;\subseteq\; \bigcap_i \mathcal{U}(X'_i; r+\epsilon),
    \]
    and symmetrically with $\mathcal{X}, \mathcal{X}'$ exchanged.
Hence the two intersection-filtration modules are $\epsilon$-interleaved, $d_I(\mathrm{PH}_*(\mathcal{X}), \mathrm{PH}_*(\mathcal{X}')) \leq \epsilon$, and their persistence diagrams are within bottleneck distance $\epsilon$.
    \item \textbf{Agreement off critical scales.} $\Delta\chi(r) = \Delta\chi'(r)$ at every $r$ whose neighborhood $[r-\epsilon, r+\epsilon]$ contains no critical scale {(Definition~\ref{def:regular_window})} of the intersection filtration of $\mathcal{X}$ or of $\mathcal{X}'$.

    \item \textbf{$L^1$ stability.} $d_{L^1}(\Delta\chi, \Delta\chi') \leq 2\sum_q W_1(\mathrm{Dgm}_q, \mathrm{Dgm}'_q) \leq 2\epsilon\, (N + N')$, where $W_1$ is the $1$-Wasserstein distance (with the $\ell^\infty$ ground metric on the half-plane, {so that the factor $2$ is exactly the $\ell^1$-to-$\ell^\infty$ conversion}) between the degree-$q$ persistence diagrams of the two intersection filtrations and $N = \sum_q |\mathrm{Dgm}_q|$, $N' = \sum_q |\mathrm{Dgm}'_q|$ are the total numbers of persistence bars (birth--death pairs, summed over all homological degrees $q$) of the two intersection filtrations.
  The Wasserstein form is controlled by the \emph{actual} bar displacements; the bound $2\epsilon (N + N')$ is the worst case, charging every matched pair the maximal shift $\epsilon$ and every unmatched bar---necessarily of persistence $\leq 2\epsilon$---its distance to the diagonal.
\end{enumerate}
\end{theorem}

\begin{proof}
Since $d_H(X_i, X'_i) \leq \epsilon$, we have $\mathcal{U}(X'_i; r-\epsilon) \subseteq \mathcal{U}(X_i; r) \subseteq \mathcal{U}(X'_i; r+\epsilon)$ for $r \geq \epsilon$; intersecting over $i$ gives the displayed inclusions, and exchanging $\mathcal{X}, \mathcal{X}'$ gives the symmetric ones.
An $\epsilon$-interleaving of filtrations induces an $\epsilon$-interleaving of the homology persistence modules in each degree {(over the fixed field $\mathbb{F}$)}. The interval decomposition of pointwise finite-dimensional modules \citep{chazal2016structure} is what makes the diagrams $\mathrm{Dgm}_q$ exist and makes $\beta_q(r)$ the number of bars over $r${, on which the identities (2) and (3) run}.

The interleaving gives the bottleneck bound by the isometry theorem \citep{chazal2016structure}; this is~(1).

For (2), the hypothesis excludes critical scales of \emph{both} filtrations from the closed window $[r-\epsilon, r+\epsilon]$, so every bar endpoint of either diagram avoids the window, and every bar of either filtration contains the whole window {$[r-\epsilon, r+\epsilon]$} or is disjoint from it.
Fix the bottleneck matching from (1).
If a matched bar $I$ contains $r$, then $I \supseteq [r-\epsilon, r+\epsilon]$, so its partner $I'$---whose endpoints lie within $\epsilon$ of those of $I$---meets the window; since the endpoints of $I'$ \emph{also} avoid the window, $I' \supseteq [r-\epsilon, r+\epsilon] \ni r$.
{ An unmatched bar has persistence $\leq 2\epsilon$; but a bar covering $r$ whose endpoints avoid the closed window $[r-\epsilon, r+\epsilon]$ has birth $< r-\epsilon$ and death $> r+\epsilon$, hence persistence $> 2\epsilon$---so no unmatched bar covers $r$.}
So the position of $r$ inside a bar never enters: the bars over $r$ of the two diagrams correspond exactly under the matching, and since $\beta_q(r)$ \emph{is} the number of bars over $r$ (the interval decomposition over $\mathbb{F}$ again), $\beta_q(r;\mathcal{X}) = \beta_q(r;\mathcal{X}')$ for every $q$, and hence $\Delta\chi(r) = \Delta\chi'(r)$.

For (3), write $\beta_q(r) = \beta_q(r;\mathcal{X})$ and $\beta'_q(r) = \beta_q(r;\mathcal{X}')$ for the degree-$q$ Betti curves (dimensions over $\mathbb{F}$) of the two intersection filtrations.
By (2) the two profiles differ only on the union of the $\epsilon$-neighborhoods of the critical scales.
The pointwise identity $\Delta\chi(r) - \Delta\chi'(r) = \sum_q (-1)^q\big(\beta_q(r) - \beta'_q(r)\big)$ gives, at each $r$, $|\Delta\chi(r) - \Delta\chi'(r)| \leq \sum_q |\beta_q(r) - \beta'_q(r)|$.
Integrating over the scale $r$ (an integral $\int_0^\infty \! dr$, the $L^1$ distance of the two Betti curves in degree $q$) and bounding it by the $1$-Wasserstein distance of the degree-$q$ diagrams \citep[Prop.~3.1]{dlotko2023euler},
\[
    \int_0^\infty |\beta_q(r) - \beta'_q(r)|\, dr \;\leq\; 2\,W_1(\mathrm{Dgm}_q, \mathrm{Dgm}'_q) \;\leq\; 2\epsilon\, (N_q + N'_q),
\]
where $N_q = |\mathrm{Dgm}_q|$, $N'_q = |\mathrm{Dgm}'_q|$, and the last step evaluates the bottleneck matching from (1): each matched pair costs its endpoint shifts $\leq \epsilon$, and each unmatched bar---of persistence $\leq 2\epsilon$, hence at $\ell^\infty$-distance $\leq \epsilon$ from the diagonal---costs at most $\epsilon$; the number of matched pairs and unmatched bars together is at most $N_q + N'_q$.

Summing over $q$,
\[
    d_{L^1}(\Delta\chi, \Delta\chi') \;=\; \int_0^\infty |\Delta\chi(r) - \Delta\chi'(r)|\, dr \;\leq\; \sum_q \int_0^\infty |\beta_q - \beta'_q|\, dr \;\leq\; 2\sum_q W_1(\mathrm{Dgm}_q, \mathrm{Dgm}'_q) \;\leq\; 2\epsilon\, (N + N'). \qedhere
\]

\end{proof}

\begin{remark}[Why not a persistence-style interleaving]\label{rem:interleaving_comparison}
The stability theorem for persistence modules \citep{chazal2016structure} bounds the bottleneck distance of diagrams by the Hausdorff distance of point clouds, and it is tempting to seek the analogue with $\Delta\chi$ in place of the diagram and an interleaving distance $d_I(\Delta\chi,\Delta\chi') = \inf\{\epsilon : \Delta\chi(r-\epsilon)\le\Delta\chi'(r)\le\Delta\chi(r+\epsilon)\ \forall r\}$ in place of bottleneck distance.
This fails, and the non-monotonicity of $\chi$ (Theorem~\ref{thm:algebraic_stability}) is exactly the obstruction: if $\Delta\chi = \mathbf{1}_{[a,b)}$ records a transient feature and $\Delta\chi'=\mathbf{1}_{[a+\delta,b+\delta)}$ is the same feature shifted by $\delta = d_H$, then at $r = b$ the upper bound forces $1 = \Delta\chi'(b) \le \Delta\chi(b+\epsilon) = 0$ for every $\epsilon \geq 0$; thus $d_I = \infty$ while $d_H = \delta \to 0$.
The genuine, stable counterparts are the bottleneck stability of the underlying intersection filtration (Theorem~\ref{thm:algebraic_stability}(1)), pointwise agreement away from critical scales (2), and the $L^1$ bound (3)---the bottleneck statement being the faithful image of the persistence stability theorem, now applied to the intersection filtration $\bigcap_i \mathcal{U}(X_i;t_i)$ rather than to a single cloud.

For a \emph{single} cloud ($k = 1$), $L^1$/integrated stability of the Euler characteristic curve is established---see \citet[Prop.~7]{hacquard2024euler} and \citet{dlotko2023euler}, with the related directional setting in \citet{curry2022directions}.
For the interaction profile, part~(1) of Theorem~\ref{thm:algebraic_stability} is that same standard framework applied to the intersection filtration $\bigcap_i \mathcal{U}(X_i; t_i)$---the content is the identification of the right filtered object, and that the perturbation must respect the coloring across clouds---not a new stability mechanism; parts~(2) and~(3), the pointwise agreement off critical scales and the $L^1$ bound on the profile itself, are the statements specific to $\Delta\chi$.
\end{remark}

\begin{remark}[Gromov--Hausdorff strengthening]\label{rem:gh_stability}
{ Since the Intersection ECP depends only on the ambient-isometry class of $\mathcal{X}$ (Lemma~\ref{lem:rigid_invariance}), it is natural to measure perturbations in the Gromov--Hausdorff rather than the Hausdorff distance. The bottleneck stability of part~(1)---and the relative-module stability of Theorem~\ref{thm:relative_stability}---hold in the \emph{colored} Gromov--Hausdorff distance of the joint configuration $(\bigsqcup_i X_i,\ \text{coloring})$: a color-respecting $\epsilon$-correspondence induces interleaving maps on the colored \v{C}ech complexes, hence an interleaving of the intersection filtrations, with the usual constants (bottleneck bound $2\epsilon$); the reader will see this at once from the proof of part~(1), read with the point-moving maps as an abstract correspondence. The \emph{colored, joint} distance is essential---per-cloud Gromov--Hausdorff closeness ignores the cross-cloud geometry that carries the interaction. The ambient reading $\Delta\chi = \chi(\bigcap_i \mathcal{U}(X_i; t_i))$ and the $L^1$ bound~(3) remain Euclidean, using the nerve identification in $\R^d$.}
\end{remark}

Special cases: $k = 2$ gives the primary biparameter profile $\Delta\chi(r, s;\, X, Y)$ and its diagonal $\Delta\chi(r) = \Delta\chi(r, r;\, X, Y)$; $k = 1$ recovers the classical ECC $\Delta\chi(t;\, X) = \chi(\mathcal{U}(X; t))$ \citep{roy2024ecs}; and $k \geq 3$ is the $k$-fold interaction.

\subsection{The Intersection Theorem as a Commuting Square}\label{sec:intersection}

The Intersection Theorem (Theorem~\ref{thm:intersection}), stated with the definition in Section~\ref{sec:mixup_def}, asserts that geometric intersection ($\cap$) and algebraic multiplication ($\prod$) are interchangeable under the Euler integral---a consequence of the factorization~\eqref{eq:factorization}, {which we now display as a commuting square}.

Because~\eqref{eq:intersection} holds unconditionally for compact semi-algebraic sets, one could instead \emph{define} the Intersection ECP by its right-hand side $\chi(\bigcap_i \mathcal{U}(X_i; t_i))$ and demote Definition~\ref{def:mixup_ec} to a corollary.
We use that geometric form freely for computation---it is what the stability proofs estimate---but keep the Euler integral of the product as the \emph{primitive}: it is the form that factors through~\eqref{eq:factorization}, is pinned down by the Universality Theorem below,  and makes the commuting square~\eqref{eq:naturality_diagram} a theorem rather than a definitional triviality.
That square is worth displaying explicitly.

\begin{remark}[The Intersection Theorem as a commuting diagram]\label{rem:naturality}
At each fixed $\mathbf{t}$, the theorem asserts commutativity of:
\begin{equation}\label{eq:naturality_diagram}
\begin{array}{ccc}
    \mathbf{PC}_k(\R^d) & \xrightarrow{\;(\mathbf{1}_{\mathcal{U}(X_i;\, t_i)})_i\;} & \mathbf{CF}(\R^d)^k \\[6pt]
    \Big\downarrow\vcenter{\rlap{\scriptsize$\;\bigcap_i \mathcal{U}(X_i; t_i)$}} & & \Big\downarrow\vcenter{\rlap{\scriptsize$\;\int (\prod)\, d\chi$}} \\[6pt]
    \mathbf{SA}_c & \xrightarrow{\;\;\chi\;\;} & \Z
\end{array}
\end{equation}
where $\mathbf{SA}_c$ is the compact semi-algebraic sets on which $\chi$ is defined, the top arrow is the indicator-valued offset map of~\eqref{eq:factorization} at the fixed $\mathbf{t}$, and the right arrow is the two-step composite $\mathbf{CF}^k \xrightarrow{\prod} \mathbf{CF} \xrightarrow{\int d\chi} \Z$.
The left path forms ball unions, intersects geometrically, and takes the Euler characteristic; the right path forms indicator functions, multiplies algebraically, and integrates.
Both yield $\Delta\chi(\mathbf{t})$.  This commuting square is the precise content of the Intersection Theorem; by Theorem~\ref{thm:universality} the Intersection ECP is moreover the unique pointwise-Euler interaction profile obeying Separation and Normalization on the $k$-fold overlap.
\end{remark}

The square shows the product \emph{achieves} the geometric intersection; we now turn the observation into a converse and ask whether it is the \emph{only} construction that does.
It is, and the canonicity comes in two stages that assume progressively less: the first fixes how the clouds are combined while granting the Euler integral, and the second forces the Euler characteristic itself.

\begin{theorem}[Universality of the Intersection ECP]\label{thm:universality}
Let $\Phi : \mathbf{PC}_k(\R^d) \to \mathbf{ZFun}_k$ be an assignment satisfying:
\begin{enumerate}
    \item \textbf{Pointwise Euler form}: there is a \emph{pointwise} {function} $\varphi : \Z^k \to \Z$ with $\varphi(\mathbf{0}) = 0$ such that
    \[
        \Phi(\mathcal{X})(\mathbf{t}) = \int_{\R^d} \varphi\!\left(\mathbf{1}_{\mathcal{U}(X_1;t_1)}, \ldots, \mathbf{1}_{\mathcal{U}(X_k;t_k)}\right) d\chi,
    \]
    applied pointwise: $\varphi(f_1, \ldots, f_k)(x) = \varphi\big(f_1(x), \ldots, f_k(x)\big)$;
    \item \textbf{Separation}: $\Phi(\mathcal{X})(\mathbf{t}) = 0$ whenever the $k$-fold overlap $\bigcap_i \mathcal{U}(X_i; t_i) = \emptyset$---no interaction registers without a common overlap;
    \item \textbf{Normalization}: $\Phi(\mathcal{X})(\mathbf{t}) = 1$ when $\bigcap_i \mathcal{U}(X_i; t_i)$ is a {single nonempty ($k$-fold) lens}.
\end{enumerate}
Then $\Phi = \mathcal{M}$; that is, $\mathcal{M} = \Delta\chi = \chi(C_k)$ is the \emph{unique} pointwise-Euler interaction profile vanishing off the $k$-fold overlap and normalized on a single nonempty lens.
\end{theorem}

Here $\chi$ (hence inclusion--exclusion) is granted through the Euler-integral form~(1), and only the \emph{combination} of the clouds is forced; Theorem~\ref{thm:universality_strong} removes even that.

\begin{proof}
Write $A_i = \mathcal{U}(X_i; t_i)$.
On indicators, $\varphi(\mathbf{1}_{A_1}, \ldots, \mathbf{1}_{A_k})$ takes the constant value $\varphi(\mathbf{s})$ on the level set $R_{\mathbf{s}} = \{x : \mathbf{1}_{A_i}(x) = s_i \text{ for all } i\}$---the points lying in exactly the clouds indexed by $S := \mathrm{supp}(\mathbf{s}) = \{i : s_i = 1\}$---so by $\Z$-linearity of the Euler integral and $\varphi(\mathbf{0}) = 0$,
\begin{equation}\label{eq:univ_levelset}
    \Phi(\mathcal{X})(\mathbf{t}) = \int \varphi(\mathbf{1}_{A_1}, \ldots, \mathbf{1}_{A_k})\, d\chi = \sum_{\mathbf{0} \neq \mathbf{s} \in \{0,1\}^k} \varphi(\mathbf{s})\, \chi(R_{\mathbf{s}}),
\end{equation}
Here the level sets $R_{\mathbf{s}}$ are locally closed but not compact, so the notation abbreviates $\chi(R_{\mathbf{s}}) := \int \mathbf{1}_{R_{\mathbf{s}}}\, d\chi$; writing $\mathbf{1}_{R_{\mathbf{s}}} = \mathbf{1}_{\bigcap_{i \in S} A_i} - \mathbf{1}_{(\bigcap_{i \in S} A_i) \cap (\bigcup_{j \notin S} A_j)}$ gives
\[
    \chi(R_{\mathbf{s}}) \;=\; \chi\Big(\bigcap_{i \in S} A_i\Big) \;-\; \chi\Big(\bigcap_{i \in S} A_i \,\cap \bigcup_{j \notin S} A_j\Big),
\]
a difference of two evaluations on compact semi-algebraic sets, so every term of~\eqref{eq:univ_levelset} stays inside the class on which $\chi$ is defined.
The full-support term is $R_{(1,\ldots,1)} = \bigcap_i A_i$.
(For $k = 2$, \eqref{eq:univ_levelset} is the three-term sum $\varphi(1,1)\,\chi(A_1 \cap A_2) + \varphi(1,0)\,[\chi(A_1) - \chi(A_1 \cap A_2)] + \varphi(0,1)\,[\chi(A_2) - \chi(A_1 \cap A_2)]$, using $\chi(A_i \setminus A_j) = \chi(A_i) - \chi(A_i \cap A_j)$ to keep every term the Euler characteristic of a compact set.)

\emph{Separation forces every proper-support coefficient to vanish}: $\varphi(\mathbf{s}) = 0$ for all $\mathbf{s}$ with $\emptyset \neq S \subsetneq [k]$, by induction on $m = |S|$.
{ Assume the claim holds for every $\mathbf{s}'$ with $0 < |\mathrm{supp}(\mathbf{s}')| < m$.}
Fix such an $\mathbf{s}$ and let the clouds $\{X_i\}_{i \in S}$ meet in $n$ disjoint common lenses while every other cloud $\{X_j\}_{j \notin S}$ lies far away, disjoint from them.
Then $\bigcap_i A_i = \emptyset$, so Separation makes~\eqref{eq:univ_levelset} vanish; every nonempty level set other than $R_{\mathbf{s}}$ has support of size $< m$ (the far clouds and the partial overlaps within $S$), so by the inductive hypothesis those terms drop and $\varphi(\mathbf{s})\,\chi(R_{\mathbf{s}}) = 0$.
Since $\chi(R_{\mathbf{s}}) = n$ is an arbitrary positive integer, $\varphi(\mathbf{s}) = 0$.
(The base case $m = 1$ is the same computation with all clouds mutually separated, where~\eqref{eq:univ_levelset} reduces to $\sum_i \varphi(e_i)\,\chi(A_i) = 0$ with the $\chi(A_i)$ independently arbitrary.)

\emph{Normalization fixes the top}: taking $A_1, \ldots, A_k$ to meet in a single $k$-fold lens---so $\chi(\bigcap_i A_i) = 1$ and all proper-support terms of~\eqref{eq:univ_levelset} vanish---Normalization gives $\varphi(1, \ldots, 1) = 1$.
Hence $\varphi(\mathbf{s}) = \prod_i s_i$ on $\{0,1\}^k$, so $\varphi(\mathbf{1}_{A_1}, \ldots, \mathbf{1}_{A_k}) = \prod_i \mathbf{1}_{A_i} = \mathbf{1}_{\bigcap_i A_i}$ and $\Phi(\mathcal{X})(\mathbf{t}) = \int \mathbf{1}_{\bigcap_i A_i}\, d\chi = \chi\big(\bigcap_i A_i\big) = \mathcal{M}(\mathcal{X})(\mathbf{t})$ at every $\mathbf{t}$.
\end{proof}

The \emph{$k$-fold} normalization is precisely what selects the top floor of the interaction spectrum; a \emph{pairwise} normalization selects the bottom floor $\chi(C_2)$ instead (Definition~\ref{def:floors}, and Proposition~\ref{prop:pairwise}).

Theorem~\ref{thm:universality} took the Euler integral---hence $\chi$---as a hypothesis (condition (1)).
In fact it too is forced: dropping the pointwise-Euler form in favor of neutral additivity and topological axioms still pins down $\Delta\chi$.

\begin{theorem}[The Euler characteristic is forced]\label{thm:universality_strong}
Write $I(A_1, \ldots, A_k) = \Phi(\mathcal{X})(\mathbf{t})$ for $A_i = \mathcal{U}(X_i; t_i)$ and an assignment $\Phi : \mathbf{PC}_k(\R^d) \to \mathbf{ZFun}_k$, with $I$ defined for arguments in the lattice $\mathcal{B}$ generated by finitely many closed balls---finite unions of finite intersections of balls{, closed under the finite unions and intersections over which the proof's inclusion--exclusion expansion runs}.

Suppose:
\begin{enumerate}
    \item \textbf{Multivaluation}: $I$ is a valuation in each argument separately---fixing all arguments but one (written $A, A'$),
    \[
        I(A \cup A') + I(A \cap A') = I(A) + I(A')
    \]
    for $A, A' \in \mathcal{B}$ {(the interaction responds \emph{linearly}, by inclusion--exclusion, to combining each cloud under $\cup$ and $\cap$)};
    \item \textbf{Topological invariance}: $I(A_1, \ldots, A_k)$ depends only on the homotopy type of $\bigcap_i A_i$ (we ask the interaction to capture the \emph{topology} of the overlap, not its metric size);
    \item \textbf{Separation}: $I(A_1, \ldots, A_k) = 0$ when $\bigcap_i A_i = \emptyset$;
    \item \textbf{Normalization}: $I(A_1, \ldots, A_k) = 1$ when $\bigcap_i A_i$ is nonempty and contractible.
\end{enumerate}
Then $I(A_1, \ldots, A_k) = \chi\big(\bigcap_i A_i\big)$, so $\Phi = \mathcal{M}$.
Thus neither the Euler characteristic nor the product is assumed: $\chi$ emerges as the unique homotopy-invariant valuation normalized on contractible sets, and the multivaluation structure propagates it to the interaction $\chi(\bigcap_i A_i)$.
\end{theorem}

The proof is in Appendix~\ref{app:incl-excl}: expanding $I$ by the multivaluation inclusion--exclusion over finite intersections of balls, each leaf overlap is convex---hence empty or contractible---so Separation, Normalization, and Topological invariance pin it to $\chi$, and $\chi(\bigcap_i A_i)$ obeys the same expansion.

\begin{remark}[What universality does, and does not, Claim]\label{rem:what_universality_claims}
The two theorems single out $\chi(\bigcap_i \mathcal{U}(X_i; t_i))$ as the canonical \emph{Euler-level summary} of the overlap: among pointwise-Euler assignments obeying Separation and Normalization on the $k$-fold overlap (Theorem~\ref{thm:universality})---or, assuming less, multivaluations obeying the neutral valuation and topological-invariance axioms together with the same Separation and Normalization (Theorem~\ref{thm:universality_strong})---it is the unique one.
This is a statement about \emph{which} summary, not a claim that the summary's \emph{value} measures interaction strength.
Indeed $\chi$ is signed and non-monotone (Remark~\ref{rem:interleaving_comparison}), so a larger value of $\Delta\chi(\mathbf{t})$ does not mean a ``stronger'' interaction.
Detecting \emph{whether} clouds interact, and resolving the finer structure $\chi$ compresses, is the role of the relative-homology refinement of Section~\ref{sec:relative}.
\end{remark}

\subsection{The Interaction Spectrum}\label{sec:spectrum}

The universality theorems pin $\Delta\chi$ as the canonical \emph{top-floor} summary of the overlap.
It tops a whole \emph{interaction spectrum}---a graded family of $\chi$-level descriptors, each canonical in the same sense and all stable together---which we now develop; it is the Euler-characteristic sibling of the relative-homology refinement of Section~\ref{sec:relative}.

\begin{definition}[The interaction spectrum: floors of the descending chain]\label{def:floors}
The uniqueness in Theorem~\ref{thm:universality} is relative to its normalization (value $1$ on a single $k$-fold lens); weakening that normalization exposes a whole \emph{family} of pointwise-Euler interaction descriptors, of which $\Delta\chi$ is the top member.
For $\mathbf{t} = (t_1,\ldots,t_k)$ set
\[
    C_j(\mathbf{t}) \;=\; \bigcup_{S\subseteq [k],~ |S| = j}\ \bigcap_{i \in S} \mathcal{U}(X_i; t_i), \qquad j = 1,\ldots,k,
\]
the region where \emph{at least $j$} of the clouds overlap.
These form a descending chain
\[
    \textstyle\bigcup_i \mathcal{U}(X_i; t_i) = C_1 \supseteq C_2 \supseteq \cdots \supseteq C_k = \bigcap_i \mathcal{U}(X_i; t_i),
\]
and each $\chi(C_j)$ is again a pointwise-Euler descriptor: writing $\varepsilon_i = \mathbf{1}_{\mathcal{U}(X_i; t_i)}$, the region $C_j$ has indicator $\mathbf{1}_{C_j} = \varphi_j(\varepsilon_1, \ldots, \varepsilon_k)$ for the explicit pointwise function
\[
    \varphi_j(\varepsilon_1, \ldots, \varepsilon_k) \;=\; \max_{|S| = j}\ \prod_{i \in S} \varepsilon_i,
\]
whose value is $1$ exactly when at least $j$ of the $\varepsilon_i$ equal $1$; hence $\chi(C_j) = \int \varphi_j\big(\mathbf{1}_{\mathcal{U}(X_1;\cdot)}, \ldots, \mathbf{1}_{\mathcal{U}(X_k;\cdot)}\big)\, d\chi$.
The extremes recover the two named floors: $\varphi_k = \prod_i \varepsilon_i$ gives $\Delta\chi$, and $\varphi_2 = \max_{i < i'} \varepsilon_i \varepsilon_{i'}$ gives the pairwise floor $\chi(C_2)$ (Proposition~\ref{prop:pairwise}).

The Intersection ECP is the \emph{top floor} $\Delta\chi = \chi(C_k)$ (``all $k$ clouds meet''), canonical under the $k$-fold normalization; the \emph{bottom floor} $\chi(C_2)$ (``some pair meets'') is a distinct canonical descriptor---for $k=3$ exactly $\chi\big((A\cap B)\cup(B\cap C)\cup(C\cap A)\big)$, characterized by its own universality statement (Proposition~\ref{prop:pairwise}).
Thus $\Delta\chi$ is canonical \emph{as the top floor}, one member of the interaction spectrum $\{\chi(C_j)\}_{j=2}^{k}$, not the sole pointwise-Euler interaction invariant.

The floors are informative jointly.
Under the hypothesis $(\ast)$ that \emph{some} pairwise overlap already forces the total overlap ($C_2(\mathbf{t}) \neq \emptyset \Rightarrow C_k(\mathbf{t}) \neq \emptyset$), the floors switch on together and $\Delta\chi$ loses nothing.

When $(\ast)$ fails they separate, and the \emph{delay} between the onset of $C_2$ and of $C_k$ is itself a scale the profile records.
That gap is reflected in the reduced {homology} of the interaction nerve $\mathcal{N}(\{\mathcal{U}(X_i; t_i)\}_i)$---in the encircling configuration of Figure~\ref{fig:motivating}(c), $\mathcal{N} \simeq S^1$, exactly the hole that keeps $C_k$ empty while $C_2$ is not.
\end{definition}

Every floor is canonical in the same sense as the top (Theorem~\ref{thm:universality}, the case $j = k$): each $\chi(C_j)$ is forced by a universality statement, the separation and normalization now referring to the $j$-th floor.

\begin{proposition}[Universality of the $j$-th floor]\label{prop:pairwise}
Fix $2 \le j \le k$, and for $A_i = \mathcal{U}(X_i; t_i)$ let $\Psi : \mathbf{PC}_k(\R^d) \to \mathbf{ZFun}_k$ satisfy:
\begin{enumerate}
    \item \textbf{Pointwise Euler form}: there is a pointwise $\psi : \Z^k \to \Z$ with $\psi(\mathbf{0}) = 0$ and $\Psi(\mathcal{X})(\mathbf{t}) = I(A_1, \ldots, A_k)$, where $I(A_1, \ldots, A_k) := \int_{\R^d} \psi(\mathbf{1}_{A_1}, \ldots, \mathbf{1}_{A_k})\, d\chi$ is defined for any finite unions of balls $A_i$ (including $\emptyset$);
    \item \textbf{Separation}: $\Psi(\mathcal{X})(\mathbf{t}) = 0$ when the $j$-th floor is empty, $C_j(\mathbf{t}) = \emptyset$ (no $j$ of the clouds share a common point);
    \item \textbf{Normalization}: $\Psi(\mathcal{X})(\mathbf{t}) = 1$ when the $k$ offsets share a common point and $C_j(\mathbf{t})$ is contractible;
    \item \textbf{Reduction}: emptying any one offset yields the $j$-th floor of the rest---for each $l$,
    \[
        I(A_1, \ldots, A_{l-1}, \emptyset, A_{l+1}, \ldots, A_k) = \chi\big(C_j\big((A_i)_{i \neq l}\big)\big),
    \]
    vacuous when $k = j$.
\end{enumerate}
Then $\Psi(\mathcal{X})(\mathbf{t}) = \chi(C_j(\mathbf{t}))$, the Euler characteristic of the $j$-th floor.
\end{proposition}

\begin{proof}
By induction on $k \geq j$.
For $k = j$, Reduction is vacuous, $C_j(\mathbf{t}) = \bigcap_i A_i$, and the statement is Theorem~\ref{thm:universality}, whose normalizing configuration (a single $k$-fold lens) satisfies the present Normalization.

Let $k > j$, assume the statement for $k - 1$ clouds, and decompose over the level sets as in~\eqref{eq:univ_levelset}:
\[
    \Psi(\mathcal{X})(\mathbf{t}) = I(A_1, \ldots, A_k) = \sum_{\mathbf{0} \neq \mathbf{s}} \psi(\mathbf{s})\, \chi(R_{\mathbf{s}}),
    \qquad
    \chi(R_{\mathbf{s}}) = \chi\Big(\bigcap_{i \in S} A_i\Big) - \chi\Big(\bigcap_{i \in S} A_i \cap \bigcup_{l \notin S} A_l\Big),
\]
$S = \mathrm{supp}(\mathbf{s})$, every evaluation on a compact semi-algebraic set.
By Definition~\ref{def:floors}, $\mathbf{1}_{C_j} = \varphi_j$ with $\varphi_j(\mathbf{s}) = \mathbf{1}_{[\,|S| \geq j\,]}$, and the $(k{-}1)$-cloud indicator is its restriction: $\varphi_j(s_1, \ldots, s_{k-1}, 0)$ is the $j$-th-floor indicator of the first $k-1$ clouds.
We show $\psi = \varphi_j$.

\emph{Below the floor.} Separation forces $\psi(\mathbf{s}) = 0$ for $0 < |S| < j$, by the same induction on support size as in Theorem~\ref{thm:universality}: place $n$ disjoint common lenses of the clouds in $S$ far from all other clouds, so that fewer than $j$ clouds meet anywhere and $C_j(\mathbf{t}) = \emptyset$; the level-set sum collapses to $\psi(\mathbf{s})\,\chi(R_{\mathbf{s}}) = 0$ with $\chi(R_{\mathbf{s}}) = n$ arbitrary.

\emph{Coefficients with a zero coordinate.} Emptying the $l$-th offset ($A_l = \emptyset$; the pointwise Euler form is defined for all finite unions of balls, including $\emptyset$), Reduction and the induction hypothesis for the remaining $k - 1$ clouds give
\[
    \sum_{\mathbf{s}:\, s_l = 0} \psi(\mathbf{s})\, \chi(R_{\mathbf{s}})
    \;=\; \chi\big(C_j\big((A_i)_{i \neq l}\big)\big)
    \;=\; \sum_{\mathbf{s}:\, s_l = 0} \varphi_j(\mathbf{s})\, \chi(R_{\mathbf{s}})
\]
for every configuration of the remaining clouds; since the compact evaluations in the difference formula for $\chi(R_{\mathbf{s}})$ can be realized with arbitrary integer values, the coefficients match: $\psi(\mathbf{s}) = \varphi_j(\mathbf{s})$ whenever $\mathbf{s}$ has a zero coordinate---in particular $\psi(\mathbf{s}) = 1$ for $j \leq |S| < k$.

\emph{The top coefficient.} Take nested offsets $A_1 \supseteq A_2 \supseteq \cdots \supseteq A_k \neq \emptyset$, single balls around distinct nearby centers.
Then $\bigcap_{i \in S} A_i = A_{\max S}$ for every $S$, so $C_j(\mathbf{t}) = A_j$ is a ball---contractible---and all $k$ offsets share the points of $A_k$: Normalization applies and $I(A_1, \ldots, A_k) = 1$.
In this configuration $\chi(R_{\mathbf{s}}) = 0$ for every $\mathbf{s} \neq (1, \ldots, 1)$: if some $l \notin S$ has $l < \max S$, then $A_{\max S} \cap \bigcup_{l \notin S} A_l = A_{\max S}$ and the difference vanishes; otherwise $S = \{1, \ldots, m\}$ with $m < k$ and the difference is $\chi(A_m) - \chi(A_{m+1}) = 0$.
Only the full-support term survives, so $1 = I(A_1, \ldots, A_k) = \psi(1, \ldots, 1)\, \chi(A_k) = \psi(1, \ldots, 1)$.

Hence $\psi = \varphi_j$, so $\Psi(\mathcal{X})(\mathbf{t}) = \int \varphi_j(\mathbf{1}_{A_1}, \ldots, \mathbf{1}_{A_k})\, d\chi = \chi(C_j(\mathbf{t}))$.
(For $k = 3$, $j = 2$ this is $\chi(A_1 \cap A_2) + \chi(A_2 \cap A_3) + \chi(A_3 \cap A_1) - 2\,\chi(A_1 \cap A_2 \cap A_3)$; for $j = k$ it is Theorem~\ref{thm:universality}.)
\end{proof}

\smallskip

The spectrum is not merely definable but \emph{stable}: the same offset containment that stabilizes $\Delta\chi$ stabilizes every floor and every relative pair at once.

\begin{proposition}[Stability of the interaction spectrum]\label{prop:spectrum_stability}
Let $\mathcal{X} = (X_1,\ldots,X_k)$ and $\mathcal{X}' = (X'_1,\ldots,X'_k)$ be objects of $\mathbf{PC}_k(\R^d)$ with $\epsilon = \max_i d_H(X_i, X'_i)$.
For each $j$ the floor filtrations $\{C_j(\mathbf{t})\}$ and $\{C'_j(\mathbf{t})\}$ are $\epsilon$-interleaved, and so are the pair filtrations $\{(C_j(\mathbf{t}), C_{j+1}(\mathbf{t}))\}$ (with the convention $\mathcal{U}(X; s) = \emptyset$ for $s < 0$, so that the interleaving inclusions are defined at all scales).
Hence the persistence modules $PH_*(C_j)$ and the relative modules $PH_*(C_j, C_{j+1})$ are $\epsilon$-interleaved, and on the diagonal $t_1 = \cdots = t_k$ their barcodes are within bottleneck distance $\epsilon$ \citep{chazal2016structure}.
\end{proposition}

\begin{proof}
Since $d_H(X_i, X'_i) \le \epsilon$, we have $\mathcal{U}(X'_i; t_i - \epsilon) \subseteq \mathcal{U}(X_i; t_i) \subseteq \mathcal{U}(X'_i; t_i + \epsilon)$ for $t_i \ge \epsilon$.
Intersecting over any $S$ with $|S| = j$ preserves both inclusions, and taking the union over such $S$ preserves them again, so
\[
    C'_j(\mathbf{t} - \epsilon\mathbf{1}) \;\subseteq\; C_j(\mathbf{t}) \;\subseteq\; C'_j(\mathbf{t} + \epsilon\mathbf{1}),
\]
and symmetrically with $\mathcal{X}, \mathcal{X}'$ exchanged; this is an $\epsilon$-interleaving of the floor filtrations.
The containment $C_{j+1} \subseteq C_j$ is respected by these inclusions, so the interleaving maps carry pairs to pairs and the pair filtrations are $\epsilon$-interleaved as well.
Passing to homology gives the module interleavings, and the isometry theorem gives the diagonal bottleneck bound.
\end{proof}

\begin{remark}[Computable simplicial model; special cases]\label{rem:spectrum_nerve}
Each floor carries a nerve model extending Proposition~\ref{prop:nerve}: $C_j(\mathbf{t})$ is covered by the collection $\mathcal{L}_j(\mathbf{t})$ of nonempty convex $j$-fold lenses
\[
\bigcap_{i\in S} B(x_i; t_i), \qquad x_{i} \in X_i, \quad |S| = j
\]
(one center $x_i \in X_i$ per index $i \in S$), so $C_j(\mathbf{t}) \simeq \mathcal{N}(\mathcal{L}_j(\mathbf{t}))$ by the Nerve Theorem, and the modules of Proposition~\ref{prop:spectrum_stability} are computable simplicial persistence modules.
The top floor $j = k$ is $\Delta\chi = \chi(\mathcal{N}(\mathcal{L}))$ (Proposition~\ref{prop:nerve}), where Proposition~\ref{prop:spectrum_stability} recovers the intersection-filtration stability of Theorem~\ref{thm:algebraic_stability}(1); for $k = 2$ the relative module $H_*(C_1, C_2)$ recovers the relative interaction module of Theorem~\ref{thm:relative_stability}.
\end{remark}

Computationally the spectrum costs no more than its top floor: by the nerve model of Remark~\ref{rem:spectrum_nerve} each $C_j(\mathbf{t}) \simeq \mathcal{N}(\mathcal{L}_j(\mathbf{t}))$, so every $\chi(C_j)$ is a signed Alpha-complex count---assembled, like $\Delta\chi$ itself, by inclusion--exclusion over the constituent unions (Section~\ref{sec:algorithm})---and the graded modules $H_*(C_j)$ are computable simplicial persistence modules. No machinery beyond that for $\Delta\chi$ is required.

\section{Properties of the Intersection ECP}\label{sec:props}

Having fixed the construction, its stability (Theorem~\ref{thm:algebraic_stability}), and its universality (Theorem~\ref{thm:universality}), we now record the geometric and analytic properties that make $\Delta\chi$ usable as an interaction descriptor.
Section~\ref{sec:separation} locates where the profile vanishes---a triangular dead zone---and records its elementary boundary and limit behavior;
Section~\ref{sec:betti_decomp} refines it to the Betti level through a Mayer--Vietoris decomposition of the overlap;
Section~\ref{sec:discrimination} identifies what $\Delta\chi$ computes for dense samples of fixed geometry;
Section~\ref{sec:geometric_bounds} gives geometric bounds on its {size---how large the integer $|\Delta\chi|$ can grow};
and Section~\ref{sec:multiparameter_stability} extends the stability theory to the full $k$-parameter surface.
Throughout, $\mathcal{X} = (X_1, \ldots, X_k)$ is a tuple of pairwise disjoint finite clouds with $|X_i| = n_i$ and $n = \sum_i n_i$, {$\mathbf{t} = (t_1, \ldots, t_k) \in \R_{\geq 0}^k$ is the parameter, and $\Delta\chi(\mathbf{t}) = \chi\big(\bigcap_i \mathcal{U}(X_i; t_i)\big)$ as in Theorem~\ref{thm:intersection}. With a few exceptions the results are stated for general $k$; a few use a common scale $t_1 = \cdots = t_k = r$ (the diagonal), and one---the single-connecting-map Betti refinement---is intrinsically two-cloud, as noted there.}

\subsection{Separation and Basic Properties}\label{sec:separation}

The most basic questions about an interaction descriptor are when it registers nothing and how large it can grow when it does; we answer both here.
The vanishing side is governed by separation: two clouds cannot overlap until their offsets close the gap between them, and this alone forces $\Delta\chi$ to be zero---the rigorous form of desideratum~(P1), which the lemma and dead-zone corollary below make precise.

\begin{lemma}[Vanishing for separated clouds]\label{lem:separation}
Write $d_{\min}(X_i, X_j) = \min_{x \in X_i, y \in X_j} \|x - y\|$. If for some pair $i \neq j$ we have $t_i + t_j < d_{\min}(X_i, X_j)$, then
\[
    \Delta\chi(\mathbf{t};\, \mathcal{X}) = 0.
\]
\end{lemma}

\begin{proof}
When $t_i + t_j < d_{\min}(X_i, X_j)$, the balls $B(x, t_i)$ and $B(y, t_j)$ are disjoint for all $x \in X_i$, $y \in X_j$, since $\|x - y\| \geq d_{\min}(X_i, X_j) > t_i + t_j$. Hence already $\mathcal{U}(X_i; t_i) \cap \mathcal{U}(X_j; t_j) = \emptyset$, so a fortiori $\bigcap_l \mathcal{U}(X_l; t_l) = \emptyset$, and by Theorem~\ref{thm:intersection}, $\Delta\chi(\mathbf{t}) = \chi(\emptyset) = 0$.
\end{proof}

Both this and the {dead-zone lemma} below are immediate from the Intersection Theorem; we state them as named results because each is invoked later (Lemma~\ref{lem:separation} in Theorem~\ref{thm:discrimination}, the dead zone in the onset-surface discussion that follows).
Summing this pairwise obstruction over the tuple globalizes it to a region of parameter space on which the profile is identically zero. { The dead zone, together with the profile's elementary boundary and limit behavior is collected in a single lemma.}

\begin{lemma}[Dead zone, boundary, and limit]\label{cor:dead_zone}
Let $d_{\min} = \min_{i \neq j} d_{\min}(X_i, X_j)$ be the smallest pairwise cross-cloud distance.
\begin{enumerate}
    \item \textbf{Dead zone.} The Intersection ECP vanishes on the region $\{\mathbf{t} \in \R_{\geq 0}^k : \sum_i t_i < d_{\min}\}$. For $k = 2$ this is the triangle $\{(r,s) : r + s < d_{\min}\}$; on the diagonal it recovers the interval $[0, d_{\min}/2)$. Along any ray from the origin, the first non-zero value of $\Delta\chi$ occurs no earlier than the crossing of the boundary hyperplane $\sum_i t_i = d_{\min}$---exactly at the crossing when $k = 2$, and possibly strictly later when $k \geq 3$.
    \item \textbf{Boundary.} If some $t_i = 0$, the overlap lies in the finite set $X_i$, and
    \[
        \Delta\chi(\mathbf{t}) = \#\{x \in X_i : d(x, X_j) \leq t_j \text{ for all } j \neq i\} \geq 0.
    \]
    In particular $\Delta\chi(\mathbf{0}) = 0$.
    \item \textbf{Limit.} $\Delta\chi(\mathbf{t}) \to 1$ as $\min_i t_i \to \infty$.
\end{enumerate}
\end{lemma}

\begin{proof}
(1) If $\sum_l t_l < d_{\min}$ then for the closest pair $(i,j)$, since the $t_l \geq 0$, $t_i + t_j \leq \sum_l t_l < d_{\min} \leq d_{\min}(X_i, X_j)$, and Lemma~\ref{lem:separation} gives $\Delta\chi(\mathbf{t}) = 0$.
For $k = 2$, at $r + s = d_{\min}$ the closed balls around a closest pair $x \in X$, $y \in Y$ with $\|x - y\| = d_{\min}$ are externally tangent, so the overlap becomes nonempty for the first time---a finite set of tangency points---and $\Delta\chi$ equals their number $\geq 1$.
For $k \geq 3$, $\sum_i t_i = d_{\min}$ makes only the closest \emph{pair} of clouds touch, while the $k$-fold overlap requires a point common to \emph{every} cloud; a third cloud may still be far away, so the first non-zero value can occur strictly later.
(2) $t_i = 0$ gives $\mathcal{U}(X_i; 0) = X_i$, so $\bigcap_l \mathcal{U}(X_l; t_l) \subseteq X_i$ is a finite point set; $x \in X_i$ lies in it iff $d(x, X_j) \leq t_j$ for every $j \neq i$, and $\chi$ of a finite set is its cardinality. Pairwise disjointness kills the count at $\mathbf{t} = \mathbf{0}$.
(3) Fix any $p \in \R^d$; once $t_i \geq \max_{x \in X_i}\|p - x\|$ for every $i$---which holds for $\min_i t_i$ large---each $\mathcal{U}(X_i; t_i)$ is star-shaped about $p$, since for $x \in B(x_0, t_i)$ convexity of the ball and $\|p - x_0\| \leq t_i$ give $[p, x] \subseteq B(x_0, t_i) \subseteq \mathcal{U}(X_i; t_i)$. The intersection of sets star-shaped about the common point $p$ is star-shaped about $p$, hence contractible, so $\Delta\chi(\mathbf{t}) = 1$.
\end{proof}

The dead zone splits the profile's information in two: \emph{where} interaction begins (metric content) and \emph{what} the overlap then looks like (topological content).

Two consequences.
First, the ``where'' is richer than the {lemma's} bound, which is only an outer bound on the full zero set $\{\mathbf{t} : \bigcap_i \mathcal{U}(X_i; t_i) = \emptyset\}$: for $k \geq 3$ the boundary of the zero set (the \emph{onset surface}) carries strictly more than the single scale at which it meets the diagonal.
For $X_1 = \{0\}$, $X_2 = \{1\}$, $X_3 = \{10\}$ in $\R$, Helly's theorem in dimension one gives the zero set as the complement of $\{t_1 + t_2 \geq 1,\ t_2 + t_3 \geq 9,\ t_1 + t_3 \geq 10\}$: the facet offsets are the pairwise distances, so the surface recovers all of $1, 9, 10$, while the diagonal---piercing the surface at $(5,5,5)$, interior to the facet $t_1 + t_3 = 10$---recovers only $10 = 2r^*$, and the {lemma's} hyperplane $\sum_i t_i = d_{\min}$ only $1$.
(For $k = 2$ the onset surface \emph{is} the hyperplane $r + s = d_{\min}$, and the diagonal carries the same information.)

Second, a scalar summary should not \emph{mix} the two kinds of content, and the peak statistic is the one that does not: rescaling a configuration by $\lambda > 0$ sends $\Delta\chi(r)$ to $\Delta\chi(r/\lambda)$ (Lemma~\ref{lem:scale_equivariance}), multiplying $\int |\Delta\chi|\,dr$ by $\lambda$ while leaving $\max_r |\Delta\chi|$ unchanged---the integral carries units of scale, conflating topology with onset and duration, while the max is scale-free, reading the topological content alone.
(Here $\max_r$ is over the diagonal profile, the paper's peak-statistic convention; the same scale-invariance holds verbatim for the multiparameter supremum $\max_{\mathbf{t}} |\Delta\chi(\mathbf{t})|$, which can exceed the diagonal max, off-diagonal scales reaching overlaps the diagonal misses (Figure~\ref{fig:mp_surface}).
Per Remark~\ref{rem:what_universality_claims}, neither summary measures interaction \emph{strength}; the contrast is scale-dependence, not magnitude.)

Figure~\ref{fig:mp_surface} makes the two-parameter structure concrete for an encircling pair (Figure~\ref{fig:motivating}(c), $k = 2$): a triangular dead zone below $r + s = d_{\min}$, an asymmetric onset ridge just beyond it where the interaction first registers, and---already at $k = 2$---off-diagonal values (here $|\Delta\chi|$ up to $10$) that the diagonal slice (maximum $4$) never attains. The diagonal profile $\Delta\chi(r) = \Delta\chi(r,r)$ is a single line through this surface.

\begin{figure}[t]
\centering
\includegraphics[width=\textwidth]{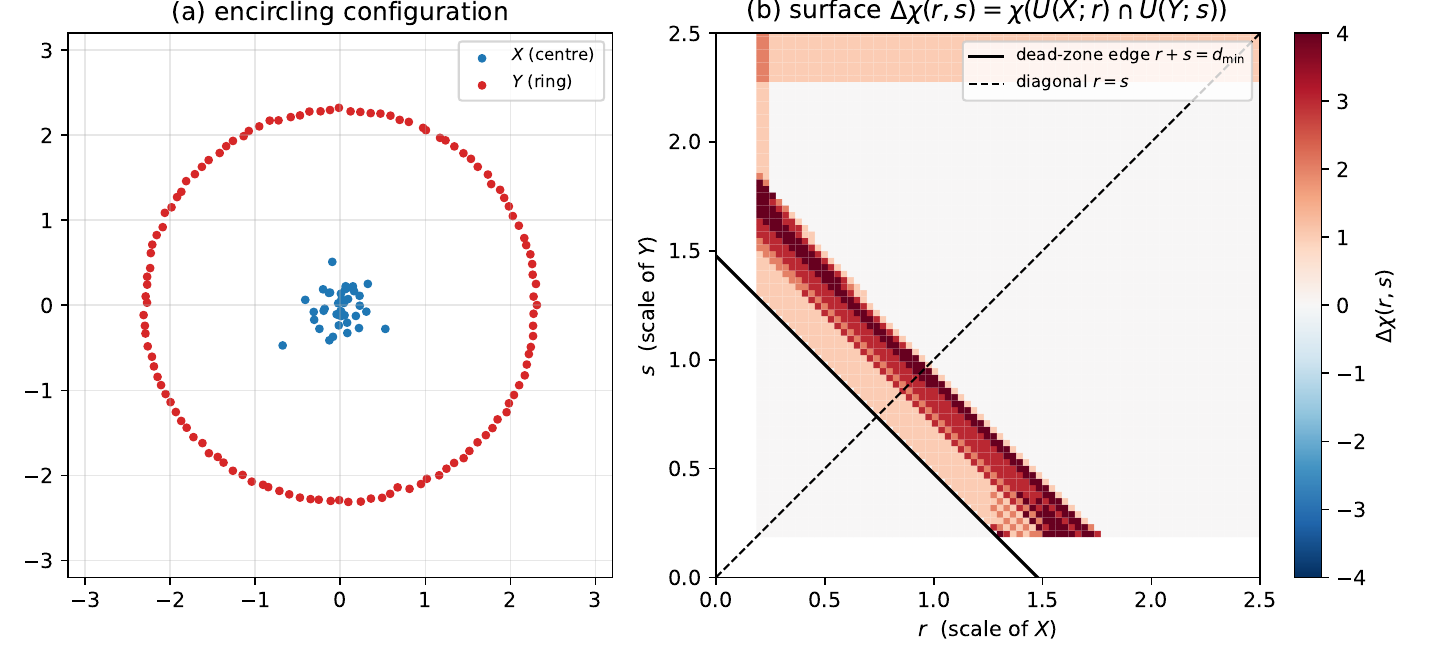}
\caption{The multiparameter surface $\Delta\chi(r,s) = \chi(\mathcal{U}(X;r) \cap \mathcal{U}(Y;s))$ for the encircling configuration \textbf{(a)}, computed exactly by the weighted (power) Alpha complex (Proposition~\ref{prop:alpha_formula}). \textbf{(b)} Below the dead-zone edge $r + s = d_{\min}$ (solid) the profile vanishes; just beyond it an asymmetric onset ridge carries the interaction. The diagonal $r = s$ (dashed) is a single slice: off the diagonal the surface reaches $|\Delta\chi| = 10$, where the diagonal reaches only $4$.}
\label{fig:mp_surface}
\end{figure}

Beyond the dead zone the profile is no longer forced to vanish. {One further elementary feature---the convex-cell decomposition on which the later geometric bounds build---is recorded next}; it recurs throughout Sections~\ref{sec:geometric_bounds} and~\ref{sec:algorithm}.

\begin{lemma}[Cell decomposition and nerve model]\label{lem:cell_decomp}
For a tuple $\mathcal{X} = (X_1, \ldots, X_k)$ of pairwise disjoint clouds with $|X_i| = n_i$, at every $\mathbf{t}$ the overlap $\bigcap_i \mathcal{U}(X_i; t_i)$ is a union of at most $\prod_i n_i$ convex cells $\bigcap_i B(x_i, t_i)$, one per cross-tuple $\mathbf{x} = (x_1, \ldots, x_k) \in \prod_i X_i$.
Writing $\mathcal{B}_{\mathcal{X}}$ for the collection of nonempty such cells, the nerve $\mathcal{N}(\mathcal{B}_{\mathcal{X}})$ is homotopy equivalent to $\bigcap_i \mathcal{U}(X_i; t_i)$.
\end{lemma}

\begin{proof}
Distributing the intersection over the unions,
$\bigcap_i \mathcal{U}(X_i; t_i) = \bigcup_{\mathbf{x} \in \prod_i X_i} \bigcap_i B(x_i, t_i)$,
and each cell $\bigcap_i B(x_i, t_i)$ is an intersection of Euclidean balls, hence convex; there are at most $\prod_i n_i$ of them.
The nonempty cells thus cover the overlap by convex sets whose finite intersections are again convex, so empty or contractible; the Nerve theorem yields $\mathcal{N}(\mathcal{B}_{\mathcal{X}}) \simeq \bigcap_i \mathcal{U}(X_i; t_i)$.
\end{proof}

One consequence of the cell decomposition is used repeatedly below: whenever the non-empty cells $\bigcap_i B(x_i, t_i)$ are pairwise disjoint, each is convex with $\chi = 1$, so by additivity $\Delta\chi(\mathbf{t})$ equals their number---the mechanism behind both the tightness of the covering-number bound and the fractal lower bound of Proposition~\ref{prop:fractal_lower} (Section~\ref{sec:geometric_bounds}), where the underlying sets are self-similar (Cantor-type) dusts.

\subsection{Betti-Number Decomposition}\label{sec:betti_decomp}

The interaction spectrum of Definition~\ref{def:floors} lives at the level of $\chi$; we now refine it one level down, to the Betti numbers of the overlap. This refinement is cleanest for two clouds, where a single Mayer--Vietoris sequence applies, so we fix $k = 2$ and write $A = \mathcal{U}(X_1; t_1)$, $B = \mathcal{U}(X_2; t_2)$, with $\Delta\beta_q = \beta_q(A) + \beta_q(B) - \beta_q(A \cup B)$ ($q$ the homological degree, not to be confused with the number of clouds). Applying $\chi = \sum_q (-1)^q \beta_q$ termwise gives the alternating identity
\begin{equation}\label{eq:betti_decomp}
    \Delta\chi = \sum_{q=0}^{d} (-1)^q \Delta\beta_q .
\end{equation}
The relationship between the individual $\Delta\beta_q$ and the Betti numbers of the overlap
is subtler than~\eqref{eq:betti_decomp} and {is given by} the full Mayer--Vietoris sequence. We record the two-cloud case, which already exposes the connecting-map ``hidden interaction'' terms; whether the analogous graded refinement for $k \geq 3$ admits a comparably clean description---through the Mayer--Vietoris spectral sequence of the cover $\{\mathcal{U}(X_i; t_i)\}_i$ or otherwise---we leave as an open question (Section~\ref{sec:discussion}).

\begin{proposition}[Mayer--Vietoris decomposition of overlap Betti numbers, $k = 2$]\label{prop:mv_betti}
With $A = \mathcal{U}(X_1; t_1)$, $B = \mathcal{U}(X_2; t_2)$, the Mayer--Vietoris long exact sequence yields, for every degree $q$,
\begin{equation}\label{eq:mv_betti}
    \beta_q(A \cap B) = \Delta\beta_q + \delta_q + \delta_{q+1},
\end{equation}
where $\delta_q = \rank\big(\partial_q : H_q(A \cup B) \to H_{q-1}(A \cap B)\big)$ is the rank of the $q$-th connecting homomorphism. In particular $\beta_q(A \cap B) \geq \Delta\beta_q$ for all $q$, with equality if and only if $\delta_q = \delta_{q+1} = 0$.
\end{proposition}

\begin{proof}
The Mayer--Vietoris sequence for $A \cup B$,
\[
    \cdots \to H_{q+1}(A \cup B) \xrightarrow{\partial_{q+1}} H_q(A \cap B) \xrightarrow{\iota_*} H_q(A) \oplus H_q(B) \xrightarrow{j_*} H_q(A \cup B) \xrightarrow{\partial_q} \cdots,
\]
is exact. From exactness at $H_q(A) \oplus H_q(B)$: $\rank(\iota_*) = \beta_q(A) + \beta_q(B) - \rank(j_q)$. From exactness at $H_q(A \cup B)$: $\rank(j_q) = \beta_q(A \cup B) - \delta_q$. Hence $\rank(\iota_*) = \Delta\beta_q + \delta_q$. From exactness at $H_q(A \cap B)$: $\beta_q(A \cap B) = \rank(\iota_*) + \delta_{q+1} = \Delta\beta_q + \delta_q + \delta_{q+1}$. Since $\delta_q, \delta_{q+1} \geq 0$, the inequality follows.
\end{proof}

\begin{corollary}[When the Betti decomposition is exact]\label{cor:exact_betti}
The equality $\beta_q(A \cap B) = \Delta\beta_q$ holds for all $q$ if and only if every connecting homomorphism $\partial_q$ vanishes---equivalently, if and only if the maps $\iota_* : H_q(A \cap B) \to H_q(A) \oplus H_q(B)$ are injective for all $q$.
{ This holds, for instance, (i) when $X_1, X_2$ are well-separated ($d_{\min} > t_1 + t_2$, so $A \cap B = \emptyset$); or (ii) when $A$ or $B$ deformation-retracts onto the overlap $A \cap B$ (one side already forces $\iota_*$ injective)---for example when $X_1, X_2$ are $\epsilon$-dense samples of the same convex compact $K$ at any scale $r \geq \epsilon$: then $K \subseteq A \cap B$, and the straight-line homotopy toward the nearest-point projection $\pi_K$ stays inside every ball constituting $A$ and $B$, so it deformation-retracts $A$, $B$, and $A \cap B$ onto $K$, making both inclusions homotopy equivalences.}
\end{corollary}

When $\Delta\beta_q > 0$ the overlap carries at least $\Delta\beta_q$ independent $q$-dimensional features, and the degree is informative: $\Delta\beta_0 > 0$ signals a fragmented (multi-component) overlap, $\Delta\beta_1 > 0$ a cycle threading the overlap. The connecting-homomorphism terms $\delta_q$ capture features of the overlap that are invisible to either union alone---they are exactly the ``pure interaction'' the alternating sum~\eqref{eq:betti_decomp} conceals, and they are computable from the same Alpha-complex data (the homology of $A$, $B$, $A \cap B$ with the inclusion maps) at no asymptotic overhead beyond the profile itself.

{ It is instructive to recall the connecting map $\partial_{q+1}$: each class $\zeta \in H_{q+1}(A \cup B)$ is represented by a $(q+1)$-cycle $z = c_A + c_B$ with $c_A, c_B$ chains in $A$ and $B$; then $\partial_{q+1}(\zeta)$ is represented by the $q$-cycle $\partial c_A = -\partial c_B$, which lives in $A \cap B$. So $\delta_{q+1} = \rank \partial_{q+1}$ counts the classes of $A \cup B$ whose ``boundary'' is a nontrivial cycle of the overlap---precisely the hidden interaction.}

\begin{example}[Two rings meeting at two spots]\label{ex:linked}
Let $A, B \subset \R^2$ be two circles that cross at exactly two points---for instance the circles of radius one centered at $(\pm\tfrac12, 0)$, which meet at $(0, \pm\tfrac{\sqrt3}{2})$ (Figure~\ref{fig:linked_rings}(a)).
Sample $A, B$ densely by disjoint $X \subset A$, $Y \subset B$ (avoiding the two crossings), and thicken each into a planar ring---the annulus $A' := \mathcal{U}(X; r)$ and likewise $B'$, which for $r < 1$ keeps its central hole---at a scale $r$ just past the first tube contact, small enough that the two meet only near the two crossings.
Then the overlap $A' \cap B'$ is two separate lenses, one at each crossing, each with no hole of its own, so $\beta_0(A' \cap B') = 2$ and $\beta_q(A' \cap B') = 0$ for $q \geq 1$; hence $\Delta\chi = \chi(A' \cap B') = 2$, a plain count of overlap pieces.
The decomposition~\eqref{eq:mv_betti} explains where the two pieces come from: since $A' \cup B'$ is connected, $\Delta\beta_0 = 1 + 1 - 1 = 1$, so $\beta_0(A' \cap B') = \Delta\beta_0 + \delta_1 = 1 + 1$---{one piece is the plain overlap, and the other is the \emph{hidden} piece supplied by the connecting map $\partial_1 : H_1(A' \cup B') \to H_0(A' \cap B')$ ($\delta_1 = \rank \partial_1 = 1$): the loop of the combined shape $A' \cup B'$---followed from one overlap lens through $A'$ to the other and back through $B'$---that removing the overlap breaks.
The same three numbers arise however two rings meet at two spots ($\beta_0(A' \cap B') = 2$, $\delta_1 = 1$, $\Delta\chi = 2$), so what the hidden piece records is this broken loop, not how the rings sit.}
\end{example}

\begin{figure}[t]
\centering
\includegraphics[width=\textwidth]{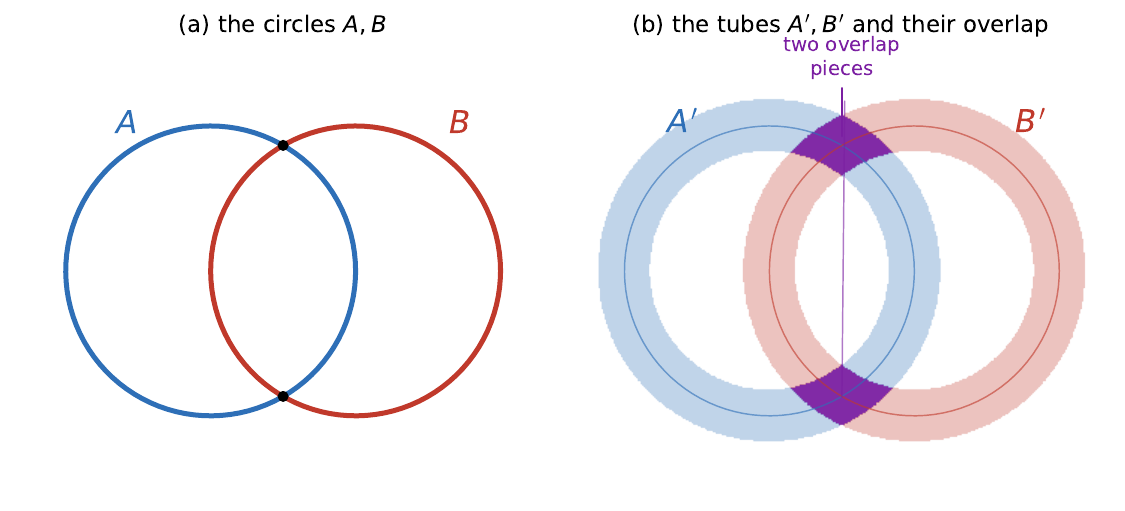}
\caption{The two rings of Example~\ref{ex:linked}, in the plane. \textbf{(a)} The circles $A$ (blue) and $B$ (red) crossing at two points. \textbf{(b)} Their thickenings to annuli $A', B'$: the overlap $A' \cap B'$ (violet) is two contractible lenses at the crossings, so $\Delta\chi = 2$. One lens is the plain overlap; the other is the hidden piece forced by the loop of $A' \cup B'$ that removing the overlap breaks.}
\label{fig:linked_rings}
\end{figure}

\subsection{Geometric Discrimination}\label{sec:discrimination}

We now identify what $\Delta\chi$ computes for dense samples of fixed geometry---the semantic content behind the peak statistic. This is the finite-sample, reach-based counterpart of the qualitative sampling consistency of Section~\ref{sec:consistency_qualitative}; the quantitative (rate/reach) form is Theorem~\ref{thm:quantitative}.

The governing principle is that once the samples are dense relative to the reach of the underlying sets, $\Delta\chi$ reports the topology of those sets rather than of the finite clouds, and thereby \emph{discriminates} between geometric configurations. The same-geometry case, in which every cloud samples one common set, is the special case recorded as part~(3) below.
Throughout, the offset notation extends verbatim to an arbitrary compact center set: $\mathcal{U}(A; \rho) = \bigcup_{a \in A} B(a, \rho) = \{x : d(x, A) \leq \rho\}$ for compact $A \subset \R^d$, agreeing with the finite-cloud definition. { For compact subsets $A_1, \ldots, A_k \subset \R^d$ and point clouds $X_i \subseteq A_i$ with $d_H(X_i, A_i) \leq \epsilon_i$, set $\epsilon = \max_i \epsilon_i$ and write
\[
    T_\rho = \bigcap_i \mathcal{U}(A_i; \rho), \qquad M_\rho = \bigcap_i \mathcal{U}(X_i; \rho),
\]
for the population and sample overlaps respectively. Throughout this section we assume $H_*(T_\rho)$ is finite-dimensional at each scale $\rho$ under consideration.  Call $r$ a \emph{regular scale} if $\rho \mapsto T_\rho$ has no homological critical value in $[r - \epsilon,\, r]$ and $\rho \mapsto M_\rho$ has none in $[r,\, r + \epsilon]$---each filtration constrained only on the side the proof uses.}

\begin{theorem}[Geometric discrimination]\label{thm:discrimination}
{ With $A_i$, $X_i$, $T_\rho$, $M_\rho$, $\epsilon$, and regular scales as above:}
\begin{enumerate}
    \item \textbf{below the contact scale}: whenever $2r < d_{\min}(A_i, A_j)$ for some pair $i \neq j$ (the minimum distance of Lemma~\ref{lem:separation}, attained by compactness), $\Delta\chi(r) = 0$; this uses only $X_i \subseteq A_i$---no window restriction on $r$, and no density or regularity hypothesis;
    \item \textbf{general}: for $r \geq \epsilon$ at a regular scale $r$, $\ \Delta\chi(r) = \chi(T_r) = \chi\big(\bigcap_i \mathcal{U}(A_i; r)\big)$;
    \item \textbf{same geometry}: if $A_1 = \cdots = A_k = A$ with positive reach $\tau > 0$, then for $\epsilon < r < \tau - \epsilon$ at a regular scale $r$, $\Delta\chi(r) = \chi(A)$.
\end{enumerate}
\end{theorem}

\begin{proof}
Since $X_i \subseteq A_i$, every $x \in \mathcal{U}(X_i; \rho)$ has $d(x, A_i) \leq \rho$, so $\mathcal{U}(X_i; \rho) \subseteq \mathcal{U}(A_i; \rho)$. Conversely, if $d(x, A_i) \leq \rho - \epsilon$ choose $a \in A_i$ with $\|x - a\| \leq \rho - \epsilon$ and, by $\epsilon$-density, $p \in X_i$ with $\|a - p\| \leq \epsilon$; then $\|x - p\| \leq \rho$, so $x \in \mathcal{U}(X_i; \rho)$. Intersecting over $i$,
\[
    T_{\rho - \epsilon} \;\subseteq\; M_\rho \;\subseteq\; T_\rho \qquad \text{for every } \rho,
\]
so the filtrations $M$ and $T$ are $\epsilon$-interleaved.

(1) A point $x$ with $d(x, A_i) \leq r$ and $d(x, A_j) \leq r$ would give $d_{\min}(A_i, A_j) \leq 2r$; so under the hypothesis $\mathcal{U}(A_i; r) \cap \mathcal{U}(A_j; r) = \emptyset$, whence $M_r \subseteq T_r = \emptyset$ (only $X_i \subseteq A_i$ used) and $\Delta\chi(r) = \chi(\emptyset) = 0$.

(2) The inclusions $T_{r-\epsilon} \subseteq M_r \subseteq T_r$ factor the map $H_*(T_{r-\epsilon}) \to H_*(T_r)$ through $H_*(M_r)$, and $M_r \subseteq T_r \subseteq M_{r+\epsilon}$ factors $H_*(M_r) \to H_*(M_{r+\epsilon})$ through $H_*(T_r)$. Regularity of $T$ on $[r-\epsilon,\, r]$ makes the first map an isomorphism, hence $H_*(M_r) \to H_*(T_r)$ is surjective; regularity of $M$ on $[r,\, r+\epsilon]$ makes the second an isomorphism, hence $H_*(M_r) \to H_*(T_r)$ is injective. So $H_*(M_r) \cong H_*(T_r)$; {both are finite-dimensional (the $H_*(T_\rho)$ by the standing hypothesis of this section, $H_*(M_r)$ automatically as a finite ball-union),} so their Euler characteristics agree and $\Delta\chi(r) = \chi(M_r) = \chi(T_r)$ by the Intersection Theorem (Theorem~\ref{thm:intersection}).

(3) With all $A_i = A$ we have $T_\rho = \mathcal{U}(A; \rho)$, and for $\rho < \tau$ the nearest-point projection $\pi_A$ deformation-retracts $\mathcal{U}(A; \rho)$ onto $A$ along segments on which $d(\cdot, A)$ is non-increasing; hence $T$ has no critical value below $\tau$ and $\chi(T_r) = \chi(A)$. Now apply~(2).
\end{proof}

Regularity is genuinely a \emph{joint} hypothesis: the interleaving ties every $M$-bar of persistence $> 2\epsilon$ to the critical values of $T$ (Theorem~\ref{thm:algebraic_stability}(1)), but $M$ may in addition carry ephemeral bars---persistence at most $2\epsilon$, sampling artifacts---at scales unrelated to $\mathrm{Crit}(T)$, and at such a scale $\Delta\chi(r) = \chi(M_r)$ reports the artifact, not the geometry. Whether these artifact scales vanish in the dense-sample limit is the regularity question of Section~\ref{sec:consistency}, whose sampling statements therefore carry the $M$-regularity event explicitly.

Theorem~\ref{thm:discrimination} shows the Intersection ECP \emph{discriminates geometric configurations}: when the clouds are drawn from distinct geometries $A_i$, it reads $\chi\big(\bigcap_i \mathcal{U}(A_i; r)\big)$, which generically differs from the value $\chi(A)$ {it would return in the coincident case $A_1 = \cdots = A_k = A$ of part~(3)},
and the statistic $\max_r |\Delta\chi(r)|$ captures the discrepancy. Its magnitude reflects the topological relationship of the $A_i$: any well-separated pair forces zero interaction at moderate scales, while overlapping geometries encode $\chi\big(\bigcap_i \mathcal{U}(A_i; r)\big)$.

\subsection{Sharp Geometric Bounds}\label{sec:geometric_bounds}

The elementary properties of Section~\ref{sec:separation} are cardinality-based; in the reach-controlled setting of Theorem~\ref{thm:discrimination} one can bound $|\Delta\chi|$ by intrinsic geometry of the overlap instead. Write the \emph{interaction zone} of the $i$-th set at scale $r$ as $Z_i(r) = A_i \cap \bigcap_{j \neq i} \mathcal{U}(A_j; 2r)$; the zones grow with $r$.

{\emph{What these bounds are, and are not.} They bound the \emph{size} of the interaction---the overlap's component count $\beta_0(T_r)$, and $|\Delta\chi|$ when those components are acyclic---by the covering complexity of the interaction zones, and are not a route to computing $\Delta\chi$, which is always the Alpha-complex sweep of Section~\ref{sec:algorithm}, never a volume or a covering number. Two regimes are worth separating. For a fixed set of positive reach $|\Delta\chi(r)|$ is bounded, and as $r \to 0$ it stabilizes to $\chi$ of the underlying overlap rather than growing (Theorem~\ref{thm:discrimination}); here the bounds simply cap the size in terms of the zone geometry. The interesting regime is \emph{rough} zones, where the component count $\beta_0(T_r)$ can diverge as $r \to 0$ at a rate keyed to the box dimension of the zones---sharp as a rate (Proposition~\ref{prop:fractal_lower}) and confirmed directly in Figure~\ref{fig:fractal_bound}. In neither case is the right-hand side a quantity one evaluates in practice (the zones $Z_i(r)$ are themselves interaction-determined); what the bounds deliver is the intrinsic-dimension content of the interaction and, in the rough regime, its growth rate---a question the barcode invariants leave untouched.}

Recall that for a compact subset $S \subset \R^d$ and $\rho > 0$, the \emph{$\rho$-covering number} $\mathcal{N}(S, \rho)$ is the least number of balls of radius $\rho$ \emph{centered in $S$} whose union contains $S$.

\begin{theorem}[Geometric bounds on the Intersection ECP]\label{thm:geometric_bounds}
Let $A_1, \ldots, A_k \subset \R^d$ be compact and let $X_i \subseteq A_i$ be $\epsilon$-dense. For every regular scale $r > \epsilon$ in the sense of Theorem~\ref{thm:discrimination}:
\begin{enumerate}
    \item \textbf{Betti bound}: $\displaystyle |\Delta\chi(r)| \leq \sum_{q=0}^{d} \beta_q\Big(\bigcap_i \mathcal{U}(A_i; r)\Big)$;
    \item \textbf{component bound}: write $T_\rho = \bigcap_i \mathcal{U}(A_i; \rho)$ as in Theorem~\ref{thm:discrimination}. If the map $H_0(T_r) \to H_0(T_{3r})$ induced by inclusion is injective (no two components of the overlap merge before the scale triples), then
    \[
        \beta_0(T_r) \;\leq\; {\min_{i=1,\ldots,k}} \mathcal{N}(Z_i(r),\, r),
    \]
    where $\mathcal{N}(Z_i(r), r)$ is the $r$-covering number of the interaction zone $Z_i(r)$. A single zone already covers the whole overlap, so the smallest of the $k$ covering numbers governs. In particular, whenever every component of $T_r$ is acyclic---e.g.\ a disjoint union of convex cells---$|\Delta\chi(r)| \leq \min_i \mathcal{N}(Z_i(r), r)$, with no dimensional constant.
\end{enumerate}
\end{theorem}

\begin{proof}
(1) By Theorem~\ref{thm:discrimination}, $\Delta\chi(r) = \chi\big(\bigcap_i \mathcal{U}(A_i; r)\big)$ at such $r$, and $|\chi(S)| = |\sum_q (-1)^q \beta_q(S)| \leq \sum_q \beta_q(S)$.

(2) \emph{Covering.} {Fix any single index $i$. Any $x \in T_r$ has a nearest point $a_i \in A_i$ with $\|x - a_i\| \leq r$; since $x \in \mathcal{U}(A_j; r)$ for every $j$, also $d(a_i, A_j) \leq \|a_i - x\| + d(x, A_j) \leq 2r$ for $j \neq i$, so $a_i \in Z_i(r)$. If $c$ is a center of a minimal $r$-covering of $Z_i(r)$ with $\|a_i - c\| \leq r$, then $x \in B(c, 2r)$. Hence $T_r$ is covered by $\mathcal{N}(Z_i(r), r)$ balls of radius $2r$---and, since $i$ was arbitrary, by $\min_i \mathcal{N}(Z_i(r), r)$ such balls.}

\emph{Separation of components.} If $C_1, C_2$ are distinct components of $T_r$ at distance $\rho = d(C_1, C_2)$, realized by $x_1 \in C_1$, $x_2 \in C_2$, then every point $y$ of the segment $[x_1, x_2]$ satisfies $d(y, A_i) \leq r + \min(\|y - x_1\|, \|y - x_2\|) \leq r + \rho/2$ for every $i$, so $C_1$ and $C_2$ lie in one component of $T_{r + \rho/2}$. If $\rho \leq 4r$ they merge by scale $3r$, contradicting the injectivity of $H_0(T_r) \to H_0(T_{3r})$. Hence distinct components are more than $4r$ apart, so a ball of radius $2r$---of diameter $4r$---meets at most one component. Every component meets some ball of the covering, giving {$\beta_0(T_r) \leq \mathcal{N}(Z_i(r), r)$ for each $i$, hence $\beta_0(T_r) \leq \min_i \mathcal{N}(Z_i(r), r)$}.

For the last claim, acyclic components give $\sum_q \beta_q(T_r) = \beta_0(T_r)$, and (1) applies.
\end{proof}

Both hypotheses in bound~(2) are essential, and the bound is sharp. Without no-early-merging it genuinely fails---two smooth curves at distance $2r-\eta$ ($\eta \ll r$) meet in ephemeral pockets spaced $\sqrt{r\eta} \ll r$ apart, outrunning any bound linear in the covering numbers as $\eta \to 0^+$, and those pockets merge just above $r$, which the hypothesis excludes. With it, the linear dependence cannot be improved: for any $N$, choose sites $p_1, \ldots, p_N$ spaced $8r$ apart on a line and place one point of each cloud inside each ball $B(p_\ell, r/2)$. This yields $N$ disjoint convex overlap cells, so $\Delta\chi(r) = N$ while $\mathcal{N}(Z_i(r), r) = N$ for every $i$ (and components lie $\geq 5r$ apart, merging only above $3.5r$, so the hypothesis holds).

Two sharpenings make the component bound's geometric content explicit: a volume estimate for smooth zones (Remark~\ref{rem:volume_bound}), and a fractal growth rate as $r \to 0^+$ that is matched from below.

{
\begin{remark}[Volume bound for smooth zones]\label{rem:volume_bound}
When each $A_i$ is a closed (compact, without boundary) smooth $k_i$-dimensional submanifold of $\R^d$ with reach $\tau_i$ and $X_i \subset A_i$ is $\epsilon$-dense ($\epsilon < \tau/2$, $\tau = \min_i \tau_i$), a packing argument makes the covering number in Theorem~\ref{thm:geometric_bounds}(2) explicit in the smooth case: for $r \in (\epsilon, \tau/2)$ there are dimensional constants $c_m > 0$ with
\[
    \mathcal{N}(Z_i(r), r) \;\leq\; c_{k_i}\, \frac{\mathrm{vol}_{k_i}\big(A_i \cap \mathcal{U}(Z_i(r); r)\big)}{r^{k_i}}
    \;\leq\; c_{k_i}\, \frac{\mathrm{vol}_{k_i}(A_i)}{r^{k_i}}.
\]
Indeed, a maximal $r$-separated subset $P \subseteq Z_i(r)$ is an $r$-covering, so $\mathcal{N}(Z_i(r), r) \leq |P|$; the disjoint half-balls $B(p, r/2)$, $p \in P$, cut caps $B(p, r/2) \cap A_i \subseteq A_i \cap \mathcal{U}(Z_i(r); r)$ of volume $\geq c'_{k_i}\, r^{k_i}$ (for $r/2 < \tau_i/2$, a dimensional constant, \citealp[Lemma~5.3]{niyogi2008finding}, the curvature correction keeping $c'_{k_i} < \omega_{k_i}(1/2)^{k_i}$), whence $|P|\,c'_{k_i} r^{k_i} \leq \mathrm{vol}_{k_i}(A_i \cap \mathcal{U}(Z_i(r); r))$ with $c_{k_i} = 1/c'_{k_i}$; the second inequality is monotonicity. Substituted into Theorem~\ref{thm:geometric_bounds}(2), this fixes the smooth covering exponent as the integer dimension $k_i$.
\end{remark}
}

{As stressed above, the estimate is asymptotic in $r$ and valid only inside the window $r \in (\epsilon, \tau/2)$: for \emph{large} $r$ the no-early-merging hypothesis of Theorem~\ref{thm:geometric_bounds}(2) can fail and the bound need not hold, while at a fixed \emph{small} $r$ it is deliberately loose---the right-hand side $\sim r^{-k_i}$ can dwarf $|\Delta\chi(r)|$, which a finite sample caps by its overlap-cell count.

{Remark~\ref{rem:volume_bound}} fixes the rate for \emph{smooth} zones, where the covering exponent is the integer dimension $k_i$. The same packing argument extends to rough zones once the box-counting dimension replaces $k_i$, giving a growth rate keyed to fractal dimension.

\begin{corollary}[Fractal growth rate]\label{cor:fractal_bound}
Fix a reference scale $r_0$ and let $s = \max_i \overline{\dim}_B Z_i(r_0)$ be the largest \emph{upper} box-counting dimension of the reference zones. Then for every $\eta > 0$ there is $C_\eta$ with
\[
    \beta_0(T_r) \;\leq\; C_\eta\, r^{-s-\eta}
\]
at every scale $r \leq r_0$ satisfying the no-early-merging hypothesis of Theorem~\ref{thm:geometric_bounds}(2); when in addition the components of $T_r$ are acyclic, $|\Delta\chi(r)| \leq C_\eta\, r^{-s-\eta}$. For smooth zones $s = \max_i k_i$ and the loss $\eta$ can be dropped, recovering the rate $r^{-\max_i k_i}$ of Remark~\ref{rem:volume_bound}.
\end{corollary}

\begin{proof}
The zones are monotone in the scale, $Z_i(r) \subseteq Z_i(r_0)$ for $r \leq r_0$; since our covering number counts balls \emph{centered in the set}, a minimal covering of the larger $Z_i(r_0)$ need not place its centers in the smaller $Z_i(r)$, but the centered-in-set and unrestricted covering numbers differ by at most a dimensional constant and share the same box dimension, so up to that constant $\mathcal{N}(Z_i(r), r) \leq \mathcal{N}(Z_i(r_0), r)$.

{Fix $\eta > 0$. The upper box dimension is $\overline{\dim}_B Z_i(r_0) = \limsup_{\rho \to 0^+} \log \mathcal{N}(Z_i(r_0), \rho)/\log(1/\rho) \leq s$, so there is a threshold $r_\eta \in (0, r_0]$---depending on $\eta$---with $\mathcal{N}(Z_i(r_0), r) \leq r^{-s-\eta}$ for all $r < r_\eta$ and every $i$. On the complementary bounded range $r \in [r_\eta, r_0]$ the covering number is finite and monotone, at most $\mathcal{N}(Z_i(r_0), r_\eta)$, while $r^{-s-\eta} \geq r_0^{-s-\eta}$; hence, defining
\[
    C_\eta \;=\; \max\!\Big(1,\ \textstyle\max_i \mathcal{N}(Z_i(r_0), r_\eta)\cdot r_0^{\,s+\eta}\Big)
\]
(and enlarging it once more to absorb the dimensional constant above) gives $\mathcal{N}(Z_i(r), r) \leq C_\eta\, r^{-s-\eta}$ for \emph{all} $r \leq r_0$ and every $i$. Substituting into the component bound of Theorem~\ref{thm:geometric_bounds}(2), $\beta_0(T_r) \leq \min_i \mathcal{N}(Z_i(r), r) \leq C_\eta\, r^{-s-\eta}$, and, when the components of $T_r$ are acyclic, $|\Delta\chi(r)| \leq C_\eta\, r^{-s-\eta}$. In the smooth case a $k_i$-dimensional piece has box dimension exactly $k_i$ with $\mathcal{N}(Z_i(r), r) = \Theta(r^{-k_i})$, so the loss $\eta$ drops and Remark~\ref{rem:volume_bound} gives the exact power.}
\end{proof}

Corollary~\ref{cor:fractal_bound} is only a \emph{ceiling}: it bounds the growth by $r^{-s}$ but leaves open whether any geometry forces that rate or the estimate is merely loose. The next proposition supplies the matching \emph{floor}, exhibiting sets whose overlap component count---and hence, its components being acyclic, $\Delta\chi$ itself---grows \emph{exactly} at the box-dimension rate along a sequence of scales.

\begin{proposition}[Fractal lower bound]\label{prop:fractal_lower}
For every $\alpha \in (0, d)$ there exist a compact set $E \subset \R^d$ with $\dim_B E = \alpha$, a constant $c > 0$, and a geometric sequence of scales $r_n \to 0^+$ such that, with both clouds sampling $E$ (so $Z_i(r) = E$ for all small $r$): for every $n$ and all sufficiently dense disjoint samples $X, Y \subset E$,
\[
    \Delta\chi(r_n;\, X, Y) \;=\; \beta_0\big(\mathcal{U}(E; r_n)\big) \;\geq\; c\, r_n^{-\alpha}.
\]
Thus the rate of Corollary~\ref{cor:fractal_bound} is attained---without the $\eta$ loss---along a sequence of scales.
\end{proposition}

\begin{proof}
Let $C \subset [0,1]$ be a self-similar Cantor set that is a disjoint union $C = \bigsqcup_{j=1}^{m} C_j$, where each $C_j$ is a $\lambda$-homothetic copy of $C$ with $\lambda = m^{-1/s}$ and $s = \alpha/d \in (0,1)$, and let $E \subset \R^d$ be the $d$-fold product of $C$.
It follows directly from the construction that $\dim_B E = d\, \frac{\log m}{\log (1/\lambda)} = \alpha$.
The level-$n$ cells are $N_n = m^{dn}$ clusters, each of diameter $\sqrt{d}\,\lambda^n$, with pairwise gaps $\geq g\lambda^{n-1}/m$ for a fixed $g > 0$, and $m^{1 - 1/s} \to 0$ ensures diameters are eventually much smaller than gaps.
Set $r_n = 2\sqrt{d}\,\lambda^n$; for large $n$, $2r_n$ is below the inter-cluster gap, so $\mathcal{U}(E; r_n)$ is the disjoint union of the $N_n$ cluster offsets $\mathcal{U}(E \cap Q; r_n)$ over level-$n$ cells $Q$.

Each cluster offset is contractible: every point of the convex hull $K_Q$ of $E \cap Q$ lies within $\mathrm{diam}(E \cap Q) \leq r_n/2$ of $E \cap Q$, so $K_Q \subseteq \mathcal{U}(E \cap Q; r_n)$, and the straight-line homotopy toward the nearest-point projection $\pi_{K_Q}$ stays inside every ball with center in $K_Q$ (the argument of Corollary~\ref{cor:exact_betti}), retracting the cluster offset onto the convex $K_Q$.
Hence $\beta_0(\mathcal{U}(E; r_n)) = N_n = m^{dn} \geq c\, r_n^{-\alpha}$, and $\chi = N_n$ as well.
For the sample statement, let $X, Y \subset E$ be disjoint and $\epsilon$-dense with $\epsilon \leq r_n/2$. Then $\mathcal{U}(E; r_n - \epsilon) \subseteq \mathcal{U}(X; r_n) \cap \mathcal{U}(Y; r_n) \subseteq \mathcal{U}(E; r_n)$, so the sample overlap splits into the same $N_n$ groups; within each cell $Q$ the group is an intersection of two unions of $r_n$-balls with centers in $E \cap Q \subseteq K_Q$, and it contains $K_Q$ (each hull point is within $\mathrm{diam} + \epsilon \leq r_n$ of both $X \cap Q$ and $Y \cap Q$), so the same retraction makes each group contractible.
Hence $\Delta\chi(r_n; X, Y) = N_n$.
\end{proof}

Together, Corollary~\ref{cor:fractal_bound} and Proposition~\ref{prop:fractal_lower} pin the growth rate: a zone of upper box dimension $s$ lets the overlap's component count---and $|\Delta\chi|$, for acyclic components---grow no faster than $r^{-s}$ (up to an arbitrarily small loss, under no-early-merging), and for suitable geometry no slower.

Figure~\ref{fig:fractal_bound} confirms the rate directly. We take two clouds sampling a two-dimensional Cantor dust $E_a = C_a \times C_a$, with $C_a \subset [0,1]$ the two-child Cantor set of ratio $a$---the $m = 2$ specialization of Proposition~\ref{prop:fractal_lower}'s construction. It is totally disconnected, so at the self-similar scales $r_n$ its tube $\mathrm{Tub}_{r_n}(E_a)$ is a disjoint union of contractible clusters and $\chi$ there equals the component count $4^n$; its box dimension is $s(a) = \log 4/\log(1/a)$, and we compute $\Delta\chi(r)$ by the Alpha-complex sweep of Section~\ref{sec:algorithm}. Across a family of dusts---sampled over $s \in [0.54, 0.86]$, i.e.\ ratios $a \in [0.08, 0.20]$ kept well below $1/2$ so the clusters stay strongly separated---the profile climbs as a self-similar staircase through the values $4^n$ at the scales $r_n \asymp a^n$, exactly as Proposition~\ref{prop:fractal_lower} predicts, and its fitted growth exponent equals the box dimension to three decimals. Reaching larger $s$ (ratio $a$ nearer $1/2$) would demand finer sampling and smaller scales---the finite-resolution form of the $r \to 0$ asymptotics---which is why the range is capped.

\begin{figure}[t]
\centering
\includegraphics[width=\textwidth]{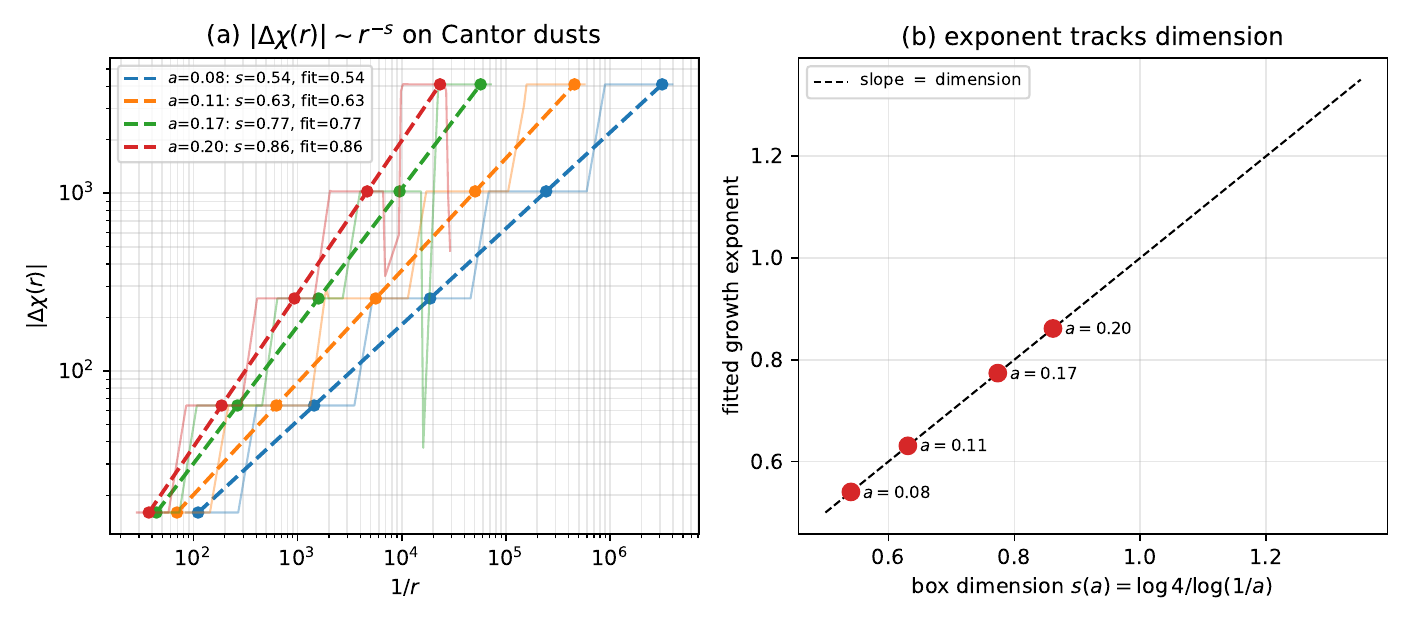}
\caption{Empirical validation of the fractal growth rate (Corollary~\ref{cor:fractal_bound}, Proposition~\ref{prop:fractal_lower}). Two clouds sample a Cantor dust $E_a = C_a \times C_a$ of box dimension $s(a) = \log 4/\log(1/a)$; $\Delta\chi(r)$ is computed by the Alpha-complex sweep. \textbf{(a)} For each dust the profile $|\Delta\chi(r)|$ (thin staircase) sits on the plateau values $4^n$ (dots) at the self-similar scales, hugging the power law $r^{-s}$ (dashed); both axes are logarithmic and the horizontal axis is $1/r$, so $r^{-s}$ appears as a line of slope $s$. \textbf{(b)} The fitted growth exponent (linear axes) equals the box dimension for every dust in the family.}
\label{fig:fractal_bound}
\end{figure}

\subsection{Multiparameter Stability}\label{sec:multiparameter_stability}

Theorem~\ref{thm:algebraic_stability} gives bottleneck, off-critical, and $L^1$ stability on the diagonal; as noted there, $L^1$ does \emph{not} control the peak statistic $\max_r|\Delta\chi|$, and we do not pursue a perturbation guarantee for it here: its statistical use rests on the recovery theory of Sections~\ref{sec:discrimination} and~\ref{sec:consistency}, and its distribution theory is left to future work.

The diagonal stability of Theorem~\ref{thm:algebraic_stability} does extend to the full $k$-parameter surface.

\begin{proposition}[Multiparameter stability]\label{prop:multiparameter_stability}
For $\mathcal{X}, \mathcal{X}' \in \mathbf{PC}_k(\R^d)$ with $\epsilon = \max_i d_H(X_i, X'_i)$:
\begin{enumerate}
    \item \textbf{Multidirectional interleaving.} For all $\mathbf{t}$ with $t_i \geq \epsilon$,
    \[
        \textstyle\bigcap_i \mathcal{U}(X'_i; t_i - \epsilon) \;\subseteq\; \bigcap_i \mathcal{U}(X_i; t_i) \;\subseteq\; \bigcap_i \mathcal{U}(X'_i; t_i + \epsilon),
    \]
    and symmetrically, so the two $k$-parameter intersection filtrations are $\epsilon$-interleaved.
    \item \textbf{Agreement off critical loci.} $\Delta\chi(\mathbf{t}; \mathcal{X}) = \Delta\chi(\mathbf{t}; \mathcal{X}')$ whenever the box $\prod_i [t_i - \epsilon, t_i + \epsilon]$ is free of homologically critical scales of \emph{both} $k$-parameter intersection filtrations (Definition~\ref{def:regular_window}).
    \item \textbf{Windowed $L^1$ bound.} On any bounded window $W \subset \R_{\geq 0}^k$, $\int_W |\Delta\chi(\mathbf{t};\mathcal{X}) - \Delta\chi(\mathbf{t};\mathcal{X}')|\,d\mathbf{t} \leq C_W\,\epsilon$, {where $C_W$ is a finite constant determined by $\sup_{\mathbf{t} \in W}|\Delta\chi(\mathbf{t};\mathcal{X}) - \Delta\chi(\mathbf{t};\mathcal{X}')|$ and the total $(k-1)$-dimensional volume, within $W$, of the critical loci of the two intersection filtrations (both finite, since the clouds are finite; see the proof).}
\end{enumerate}
\end{proposition}

\begin{proof}
Since $d_H(X_i, X'_i) \leq \epsilon$, $\mathcal{U}(X'_i; t_i - \epsilon) \subseteq \mathcal{U}(X_i; t_i) \subseteq \mathcal{U}(X'_i; t_i + \epsilon)$ for $t_i \geq \epsilon$; intersecting over $i$ gives (1).
For (2), write $F_{\mathbf{s}} = \bigcap_i \mathcal{U}(X_i; s_i)$ and $G_{\mathbf{s}} = \bigcap_i \mathcal{U}(X'_i; s_i)$.
By (1), $H_*(F_{\mathbf{t}-\epsilon\mathbf{1}}) \to H_*(G_{\mathbf{t}})$ is injective, being a factor of the isomorphism $H_*(F_{\mathbf{t}-\epsilon\mathbf{1}}) \to H_*(F_{\mathbf{t}+\epsilon\mathbf{1}})$ (regularity of $\mathcal{X}$ on the box), so $\dim H_*(G_{\mathbf{t}}) \geq \dim H_*(F_{\mathbf{t}-\epsilon\mathbf{1}}) = \dim H_*(F_{\mathbf{t}})$ in every degree; exchanging the roles of $\mathcal{X}$ and $\mathcal{X}'$ (regularity of $\mathcal{X}'$) gives the reverse inequalities.
Hence all Betti numbers agree at $\mathbf{t}$ and $\Delta\chi(\mathbf{t};\mathcal{X}) = \Delta\chi(\mathbf{t};\mathcal{X}')$ by the Intersection Theorem.

For (3), consider
\[
    E = \big\{(p, \mathbf{t}) \in \R^d \times \R^k : p \in \textstyle\bigcap_i \mathcal{U}(X_i; t_i)\big\},
\]
and let $\pi : E \to \R^k$ be the projection to the second factor. The identity
\[
    E = \bigcap_{i=1}^{k} \bigcup_{x_i \in X_i} \big\{(p, \mathbf{t}) : \|p - x_i\| \leq t_i\big\}
\]
shows that $E$ is a semi-algebraic subset of $\R^d \times \R^k$, hence so is the projection $\pi$. By Hardt's semi-algebraic triviality theorem \citep{hardt1980semialgebraic}, there is a semi-algebraic subset $\Sigma^0 \subset \R^k$ of codimension one such that $\pi$ restricts to a locally trivial bundle over $\R^k \setminus \Sigma^0$; consequently every $\mathbf{t} \in \R^k \setminus \Sigma^0$ has a neighborhood $V$ on which $H_q(\bigcap_i \mathcal{U}(X_i; t_i)) \to H_q(\bigcap_i \mathcal{U}(X_i; u_i))$ is an isomorphism for all $\mathbf{u} \in V$. Hence the homological critical locus $\Sigma$ of Definition~\ref{def:regular_window} is contained in $\Sigma^0$.

Fix a bounded window $W$ and set $\Sigma_W^0 = \Sigma^0 \cap W$, $\Sigma_W = \Sigma \cap W$. By the cell decomposition of semi-algebraic sets \citep{coste2000semialgebraic}, $\Sigma_W^0 = \bigsqcup_j M_j$ is a finite disjoint union of smooth semi-algebraic manifolds of dimensions $d_j \leq k-1$ with compact closure; the tube formula for a smooth compact submanifold \citep[Theorem~4.8]{gray2004tubes} gives $\mathrm{vol}\big(N_\epsilon(M_j)\big) = O(\epsilon^{\,k - d_j})$, so as $\epsilon \to 0$,
\[
    \mathrm{vol}\big(N_\epsilon(\Sigma_W)\big) \;\leq\; \mathrm{vol}\big(N_\epsilon(\Sigma_W^0)\big) \;\leq\; \sum_j \mathrm{vol}\big(N_\epsilon(M_j)\big) \;=\; O(\epsilon),
\]
the top-dimensional strata ($d_j = k-1$) dominating and any stratum of dimension $< k-1$ contributing $o(\epsilon)$. Since $|\Delta\chi(\cdot;\mathcal{X}) - \Delta\chi(\cdot;\mathcal{X}')|$ is bounded on $W$ by $\sup_W|\Delta\chi(\cdot;\mathcal{X}) - \Delta\chi(\cdot;\mathcal{X}')|$, integrating over $N_\epsilon(\Sigma_W)$ gives $\int_W|\Delta\chi(\cdot;\mathcal{X}) - \Delta\chi(\cdot;\mathcal{X}')|\,d\mathbf{t} \leq C_W\,\epsilon$ with $C_W$ as stated.
\end{proof}
The unconditional guarantees are the bottleneck and $L^1$ bounds of Theorem~\ref{thm:algebraic_stability} and Proposition~\ref{prop:multiparameter_stability}. All of these hold, moreover, in Wasserstein form: for equal-cardinality clouds $W_\infty$ equals the bottleneck matching distance and $d_H \leq W_\infty$, so every stability conclusion above holds verbatim with $\epsilon = \max_i W_\infty(X_i, X'_i)$; the $W_p$ ($p < \infty$) outlier-robust form and its sample complexity are left to future work.

\section{Relative-Homology Refinement}\label{sec:relative}

The Intersection ECP records a single integer $\chi(\bigcap_i \mathcal{U}(X_i; t_i))$ at each scale.
This is exactly what makes it cheap and stable, but it is also lossy: the Euler characteristic can vanish on a topologically nontrivial overlap.
We make the loss precise, isolate the finer invariant that repairs it, and record what is known and open about its stability.
Throughout this section we take $k = 2$ and write $C = \mathcal{U}(X; r) \cap \mathcal{U}(Y; s)$ for the overlap; everything extends to general $k$ as indicated below.

\begin{example}[The $\chi = 0$ degeneracy]\label{rem:degeneracy}
Reading $\chi(C) = 0$ as ``no interaction'' is unsafe: $C$ can be nonempty with rich topology yet $\chi(C) = 0$, the simplest case being $C \simeq S^1$.
A concrete realization in $\R^3$: let $T$ be a torus surface and $\gamma \subset T$ an essential longitude.
Sampling $Y$ on $T$ gives $\mathcal{U}(Y; s) \simeq T^2$ and sampling $X$ on $\gamma$ gives $\mathcal{U}(X; r) \simeq S^1$, with overlap a band $C \simeq S^1$; all three have $\chi = 0$.
Yet $X$ and $Y$ plainly interact---the inclusion $C \simeq S^1 \hookrightarrow \mathcal{U}(Y; s) \simeq T^2$ is not a homotopy equivalence, and the relative homology is nonzero, $H_1(T^2, S^1) = H_2(T^2, S^1) = \Z$.
The two relative classes cancel in the relative Euler characteristic $\chi(T^2, S^1) = -1 + 1 = 0$, exactly as $b_0 - b_1 = 0$ hides the loop of $S^1$.
Thus $\chi(C)$ alone---and even the relative Euler characteristic---can miss genuine interaction.
\end{example}

The repair is to retain the inclusions rather than only the Euler characteristic of $C$.
We say $X$ and $Y$ \emph{do not homotopically intervene} at scales $(r, s)$ if $C = \emptyset$, or both inclusions
\[
    C \hookrightarrow \mathcal{U}(X; r), \qquad C \hookrightarrow \mathcal{U}(Y; s)
\]
are homotopy equivalences.
In particular the relative homology groups $H_*(\mathcal{U}(X; r), C)$ and $H_*(\mathcal{U}(Y; s), C)$ then vanish; conversely, vanishing relative homology forces the inclusions to be homotopy equivalences whenever they are $\pi_1$-isomorphisms (relative Hurewicz), e.g.\ when the pieces are simply connected \citep{hatcher2002algebraic}.
The relative homology is the natural carrier of the interaction: by excision,
\begin{equation}\label{eq:excision}
    H_*\big(\mathcal{U}(X; r),\, C\big) \;\cong\; H_*\big(\mathcal{U}(X; r) \cup \mathcal{U}(Y; s),\; \mathcal{U}(Y; s)\big),
\end{equation}
the homology of the union \emph{relative to} $\mathcal{U}(Y; s)$---precisely what $\mathcal{U}(X; r)$ adds to $\mathcal{U}(Y; s)$ \citep{hatcher2002algebraic}.
Note that~\eqref{eq:excision} is the relative homology group of the union pair, which can record how $\mathcal{U}(Y; s)$ sits inside it.

For $k$ clouds the analogue is the family of pairs $\big\{\big(\bigcup_i \mathcal{U}(X_i; t_i),\; \bigcup_{i \neq j} \mathcal{U}(X_i; t_i)\big) \mid j = 1, \ldots, k\big\}$, one ``leave-one-out'' union for each $j$, measuring what cloud $j$ contributes to the rest.
These capture each cloud's marginal contribution but not the genuinely higher-order interaction carried by the triple and deeper overlaps; a \emph{faithful} $k$-cloud relative invariant is left open (Section~\ref{sec:discussion}).

Letting the scales vary produces, for each cloud left out, a \emph{leave-one-out relative module}.
For $k = 2$, leaving out $Y$ gives
\[
    \mathrm{PH}_*(X\cup Y, Y) \;=\; \big\{\, H_*\big(\mathcal{U}(X; r)\cup\mathcal{U}(Y; s),\ \mathcal{U}(Y; s)\big) \,\big\}_{r,s\geq 0} \;\cong\; \big\{\, H_*(\mathcal{U}(X; r),\, C(r,s)) \,\big\}_{r,s\geq 0},
\]
the last isomorphism by excision~\eqref{eq:excision}; it is a two-parameter persistence module refining $\Delta\chi$, whose (relative) Euler characteristic recovers the discrepancy $\chi(\mathcal{U}(X; r)) - \chi(C)$ while the module itself sees the classes that cancel in it.

A \emph{single} leave-one-out module is not faithful.
In Example~\ref{rem:degeneracy}, $C \simeq S^1 \hookrightarrow \mathcal{U}(X;r) \simeq S^1$ is an equivalence, so $\mathrm{PH}_*(X\cup Y, Y) \cong H_*(\mathcal{U}(X;r), C)$ \emph{vanishes}, and the interaction it misses---the nonzero groups $H_*(T^2, S^1)$---lives in the mirror module $\mathrm{PH}_*(X\cup Y, X)$.
We therefore carry all $k$ leave-one-out pairs at once, in the \emph{symmetrized relative interaction module}: for $\mathcal{X} = \{X_i \mid i = 1, \ldots, k\}$, let
\[
\begin{aligned}
    \mathrm{RPH}_*(\mathcal{X})
    &\;=\; \bigoplus_{i=1}^{k} \Big\{\, H_*\big(\textstyle\bigcup_{j} \mathcal{U}(X_{j};t_{j}),\ \bigcup_{j\neq i}\mathcal{U}(X_{j}; t_{j})\big) \,\Big\}_{\mathbf{t} \geq 0} \\
    &\;\cong\; \bigoplus_{i=1}^{k} \Big\{\, H_*\big(\mathcal{U}(X_i; t_i),\ \mathcal{U}(X_i; t_i) \cap \textstyle\bigcup_{j\neq i}\mathcal{U}(X_{j}; t_{j})\big) \,\Big\}_{\mathbf{t} \geq 0},
\end{aligned}
\]
the $i$-th summand measuring what cloud $i$ contributes to the rest, and the second isomorphism the cloudwise excision~\eqref{eq:excision}.
Retaining every summand, $\mathrm{RPH}_*$ is faithful where a single module is not: the degeneracy of Example~\ref{rem:degeneracy} is recovered by its {$i = Y$} summand $\mathrm{PH}_*(X \cup Y, X) \cong \mathrm{PH}_*(\mathcal{U}(Y), C)$, the nonzero one.
For $k = 2$ it has exactly the two summands $\mathrm{PH}_*(X\cup Y, Y)$ and $\mathrm{PH}_*(X\cup Y, X)$.
It is stable, and the stability is \emph{carried by the unions}.

\begin{theorem}[Stability of the symmetrized relative interaction module]\label{thm:relative_stability}
For tuples $\mathcal{X} = (X_1,\ldots,X_k)$ and $\mathcal{X}' = (X'_1,\ldots,X'_k)$ of finite subsets of $\R^d$ with $\epsilon = \max_i d_H(X_i, X'_i)$, the symmetrized relative interaction modules are stable in the multiparameter interleaving distance:
\[
    d_I\big(\mathrm{RPH}_*(\mathcal{X}),\ \mathrm{RPH}_*(\mathcal{X}')\big) \;\leq\; \epsilon.
\]
The bound holds summand by summand: for each $i$ the leave-one-out summand is $\epsilon$-interleaved with its counterpart, and a direct sum of $\epsilon$-interleavings is an $\epsilon$-interleaving.
\end{theorem}

On the diagonal $t_1 = \cdots = t_k$ each summand is a one-parameter module, and its barcode is within bottleneck distance $\epsilon$ of its counterpart's by the isometry theorem \citep{chazal2016structure}.
Off the diagonal the multiparameter summands admit no barcode \citep{carlsson2009theory, botnan2022introduction}, so the general statement is necessarily in interleaving \citep{lesnick2015theory} rather than bottleneck form.
The proof is in Section~\ref{sec:relative_stability}: it realizes each leave-one-out module through \v{C}ech complexes, identifies it with the corresponding union pair by excision and a Five-Lemma ladder, and interleaves the resulting \v{C}ech pairs by simplicial maps---so the (Hausdorff-stable) unions carry the stability.

\bigskip

The scalar invariants $\beta_q(X,Y;r,s) = \beta_q\big(\mathcal{U}(X;r)\cup\mathcal{U}(Y;s),\, \mathcal{U}(Y;s)\big)$ and their total $\beta_{\mathrm{tot}} = \sum_q \beta_q$ are read directly off the module $\mathrm{PH}_*(X\cup Y, Y)$, and inherit its $L^1$-stability from Theorem~\ref{thm:relative_stability}.

{ This stability is best read through the union side of the excision~\eqref{eq:excision}. Note first that the set-level distinction between $\cup$ and $\cap$ is \emph{not} what drives it: the intersection \emph{module} $\mathrm{PH}_*$ is itself interleaving-stable (Theorem~\ref{thm:algebraic_stability}(1)), even though the intersection \emph{set} $C$ can vary discontinuously in Hausdorff distance. What the union realization buys is that every interleaving map can be built from the (Hausdorff-Lipschitz) unions, so the set-level instability of $C$ never has to be confronted.}

The offset unions $\mathcal{U}(X_i; t_i)$ and $\bigcup_i \mathcal{U}(X_i; t_i)$ are $1$-Lipschitz in Hausdorff distance.
The intersection $C$ is not: at a fixed scale $d_H(C, C')$ admits no bound in terms of $\max(d_H(X,X'), d_H(Y,Y'))$---a single tangency can make one overlap empty and the other not---even though the intersection \emph{module} is stable (Theorem~\ref{thm:algebraic_stability}(1)).
So the instability is of $C$ \emph{as a set}, not of its persistence.
By~\eqref{eq:excision} the relative module is built entirely from the union pair $(\bigcup, \mathcal{U}(Y))$, with $C$ never appearing explicitly, so every interleaving map is an honest inclusion-induced map of a complex into a larger one, and the stability is carried by the (stable) unions.
Making this precise---constructing the interleaving maps from the union filtrations and controlling the induced maps on relative homology---is carried out in Section~\ref{sec:relative_stability}, establishing Theorem~\ref{thm:relative_stability}.

Together these assemble into a refinement tower, of increasing fidelity and cost:
\[
    \underbrace{\Delta\chi = \chi(C)}_{\text{shipped: cheap, stable}}
    \;\rightsquigarrow\;
    \underbrace{\chi(\mathcal{U}(X; r)) - \chi(C)}_{\text{relative Euler char: a cheap necessary test}}
    \;\rightsquigarrow\;
    \underbrace{H_*(\mathcal{U}(X; r), C)
    \oplus H_*(\mathcal{U}(Y; s), C)}
    _{\text{relative module: finer, stable (Thm~\ref{thm:relative_stability})}}.
\]
The first is what we compute with throughout this paper; the second is a cheap necessary condition for non-intervention (it can vanish while the module does not, as in Remark~\ref{rem:degeneracy}); the third is the finer, stable invariant that resolves what $\chi$ cancels, at the cost of the computational simplicity that motivated $\Delta\chi$.
A single leave-one-out module can still miss interaction that its mirror carries (Example~\ref{rem:degeneracy}); the \emph{faithful} object is the symmetrized sum $\mathrm{RPH}_*$ of all $k$ leave-one-out modules, equally stable by Theorem~\ref{thm:relative_stability}.
We adopt $\Delta\chi$ as the shipped invariant and record the module as its principled refinement.

\paragraph{The profile does not rescue the structural degeneracy.}
One might hope that reading $\Delta\chi$ across \emph{all} scales, rather than at a single $\mathbf{t}$, recovers what the Euler characteristic discards at each scale.
It recovers the \emph{accidental} loss: the profile jumps at the critical scales of \emph{all} homological degrees, so two features that cancel in $\chi$ at one scale generically reappear as jumps elsewhere, born and dying at different scales.
It does nothing for the \emph{structural} loss---a cancellation stable across a whole band of scales.
Figure~\ref{fig:saturation} (cubical sublevel filtration of $\max(d_X, d_Y)$; \citealp{gudhi2024}) exhibits both: a $2$D annulus and a $3$D torus shell, each with $\chi = 0$ over a band on which the Intersection profile of the interacting pair is \emph{bit-identical} to that of a non-interacting pair with empty overlap, while the Betti curves separate them outright.
The profile thus records only \emph{when} the interaction's type changes, not \emph{what} it is; no scale resolution undoes a cancellation that holds at every scale.
That structural loss is what the relative module of this section is built to detect.

\begin{figure}[t]
\centering
\includegraphics[width=\textwidth]{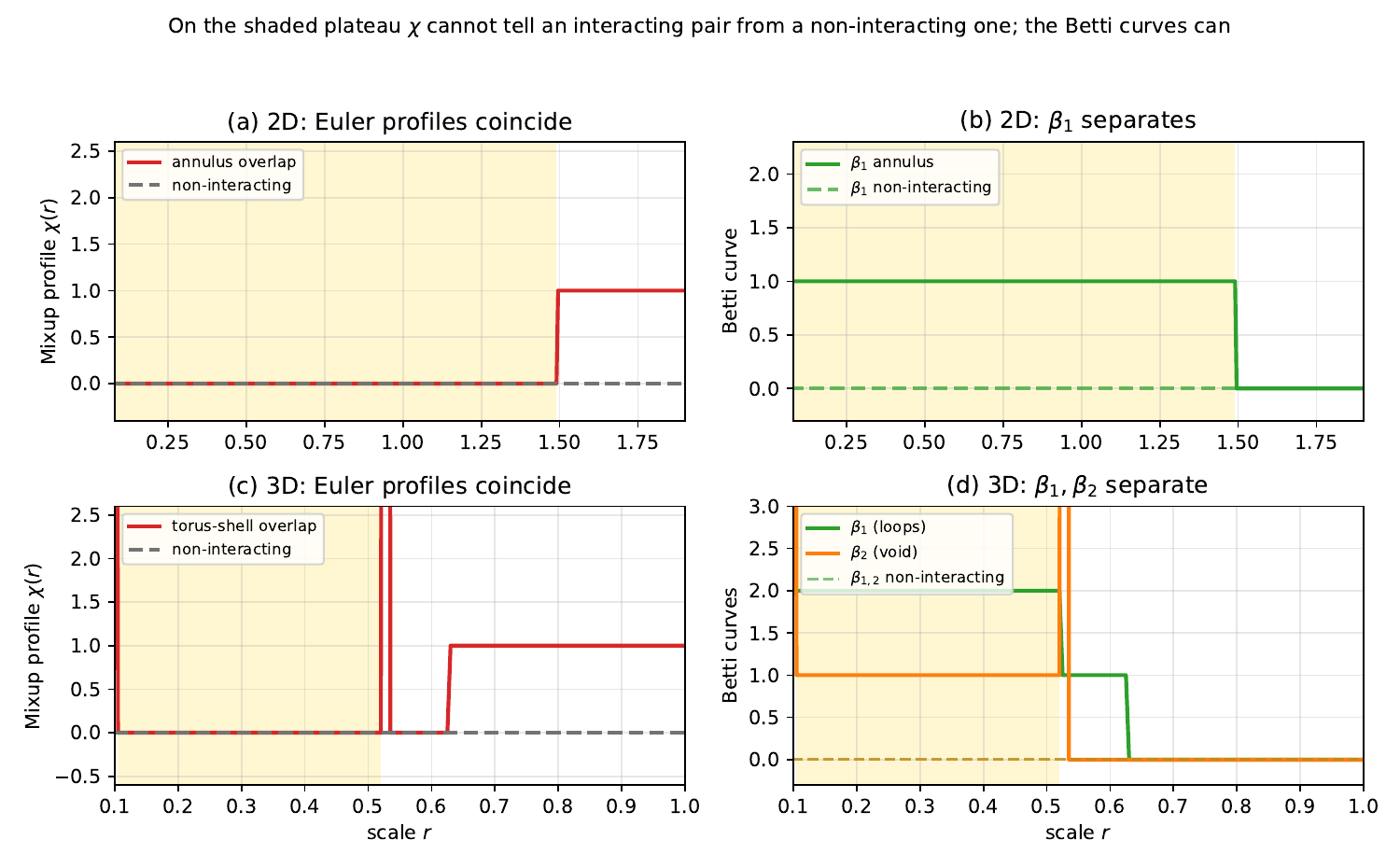}
\caption{The across-scale Intersection profile does not repair the single-scale degeneracy.
Each row compares an \emph{interacting} cloud pair, whose intersection $\mathcal{U}(X;r) \cap \mathcal{U}(Y;r)$ has nontrivial topology with $\chi = 0$, against a \emph{non-interacting} pair with empty overlap.
\textbf{Top (2D):} the intersection is an annulus ($\chi = 1 - 1 = 0$).
\textbf{Bottom (3D):} a torus shell ($\chi = 1 - 2 + 1 = 0$).
Left column: the Intersection Euler profile $\chi(r)$ of the two pairs is \emph{identical} across the shaded plateau (maximum difference $0$).
Right column: the Betti curves ($\beta_1$ in 2D; $\beta_1, \beta_2$ in 3D) separate them throughout the same band.
The profile sees only the boundary events, not the topology present on the plateau.}
\label{fig:saturation}
\end{figure}

\subsection{Construction of the module and proof of stability}\label{sec:relative_stability}

This subsection constructs the leave-one-out relative modules through \v{C}ech complexes and proves the interleaving stability of their symmetrized sum $\mathrm{RPH}_*$ (Theorem~\ref{thm:relative_stability}).
We write out the two-cloud summand $\mathrm{PH}_*(X\cup Y, Y)$; the general case follows summand by summand, as recorded at the end.

For disjoint point clouds $X$, $Y$ and $r, s > 0$, let $\check{C}((X,r)\cup(Y,s))$ denote the nerve
\[
\check{C}((X,r)\cup(Y,s)) = \mathcal{N}(\{B(x,r)\}_{x\in X}\cup \{B(y,s)\}_{y\in Y}),
\]
which contains the \v{C}ech complexes $\check{C}(X,r)$ and $\check{C}(Y,s)$ as subcomplexes.
The Nerve Theorem gives a homotopy equivalence
\[
\check{C}((X,r)\cup(Y,s)) \simeq \mathcal U(X,r) \cup \mathcal U(Y,s).
\]

By the excision isomorphism,\footnote{Strictly speaking, in the application of the excision isomorphism, we need to use the fact that $\mathcal U(X;r) \cup \mathcal U(Y;s)$ is a deformation retract of a slightly bigger set $\mathcal U(X;r+\delta)\cup \mathcal U(Y;s+\delta)$, which we may assume, since $X$ and $Y$ are finite and the balls are Euclidean.} we have
\[
H_\ast(\mathcal U(X;r),\mathcal U(X;r)\cap \mathcal U(Y;s))\cong H_\ast(\mathcal U(X;r)\cup\mathcal U(Y;s),\,\mathcal U(Y;s)).
\]

{
We next identify this union pair with a \v{C}ech pair:
\[
H_\ast\big(\mathcal U(X;r)\cup \mathcal U(Y;s),\, \mathcal{U}(Y;s)\big) \cong H_\ast\big(\check{C}((X;r)\cup(Y;s)),\,\check{C}(Y;s)\big).
\]
Indeed, by the functorial Nerve theorem \citep[Lemma 5.1]{virk2024dowkerdual} the square
\[
\xymatrix{
\mathcal U(Y;s) \ar[r]^{i_{Y}} \ar[d] & \check{C}(Y;s) \ar[d] \\
\mathcal U(X;r)\cup \mathcal U(Y;s) \ar[r]^{i_{X\cup Y}} & \check{C}((X;r)\cup(Y;s)),
}
\]
commutes up to homotopy, where $i_Y$ and $i_{X\cup Y}$ are the homotopy equivalences of the Nerve Theorem and the vertical arrows are inclusions.
Passing to homology yields a ladder of homomorphisms between the long exact sequences of the pairs $(\mathcal U(X;r)\cup\mathcal U(Y;s),\,\mathcal U(Y;s))$ and $(\check{C}((X;r)\cup(Y;s)),\,\check{C}(Y;s))$; since the two nerve maps are isomorphisms on homology, the Five Lemma yields the desired relative isomorphism. These isomorphisms are natural in $(r,s)$---which is what transports the interleaving to the module of Theorem~\ref{thm:relative_stability}---being supplied by the same functorial Nerve theorem \citep[Lemma~5.1]{virk2024dowkerdual}.

}

so they assemble into a two-parameter persistence module
\[
\operatorname{PH}_{\ast}(X\cup Y,Y) \;:=\; \big\{H_\ast\big(\mathcal U(X;r)\cup \mathcal U(Y;s),\,\mathcal U(Y;s)\big) \cong H_\ast\big(\check{C}((X;r)\cup (Y;s)),\,\check{C}(Y;s)\big) \;\big|\; r,s > 0\big\}.
\]
It remains to establish the interleaving bound of Theorem~\ref{thm:relative_stability}.
\begin{proof}[Proof of Theorem~\ref{thm:relative_stability}]
Let $\varepsilon = \max(d_{\mathrm H}(X,Z), d_{\mathrm H}(Y,W))$, and choose maps $f:X\to Z$, $g:Y\to W$, $h:Z\to X$, $k:W\to Y$ with
\[
\max\{ \Vert x-f(x) \Vert, \Vert y-g(y)\Vert, \Vert z-h(z)\Vert, \Vert w-k(w)\Vert \mid x \in X, y\in Y, z\in Z, w\in W \} \leq \varepsilon 
\]
Define $F_{r,s}:\check{C}((X;r)\cup (Y;s)) \to \check{C}((Z;r+\varepsilon)\cup (W;s+\varepsilon))$
on vertices by
\[
F_{r,s}(x) = f(x), \quad F_{r,s}(y) = g(y), \quad x \in X, \ y\in Y.
\]
If, for a point $p\in \mathbb R^d$, $\Vert p-x \Vert \leq r$ and
$\Vert p-y \Vert \leq s$, then we see
\[
\Vert p-f(x) \Vert \leq r+\varepsilon, \quad \Vert p-g(y)\Vert \leq s+\varepsilon,
\]
which implies that $F_{r,s}:X\cup Y \to Z\cup W$ actually defines a simplicial map
$F_{r,s}:\check{C}((X;r)\cup(Y;s)) \to \check{C}((Z;r+\varepsilon)\cup (W;s+\varepsilon))$.
The same applies to the map $Z\cup W \to X\cup Y$ induced by $h$ and $k$ and we obtain a simplicial map
\[
G_{r,s}:\check{C}((Z,r)\cup(W,s)) \to \check{C}((X,r+\varepsilon
)\cup (Y;s+\varepsilon)).
\]
The composition
\[
G_{r+\varepsilon,s+\varepsilon}\circ F_{r,s}:\check{C}((X;r)\cup(Y;s)) \to \check{C}((X;r+2\varepsilon)\cup(Y;s+2\varepsilon))
\]
is homotopic to the inclusion $\iota_{r,s;\varepsilon}: \check{C}((X;r)\cup(Y;s)) \hookrightarrow \check{C}((X;r+2\varepsilon)\cup(Y;s+2\varepsilon))$, because
for each simplex $[z_{0},\ldots, z_{n}]$ of $\check{C}((X;r)\cup(Y;s))$, we see by a similar estimate to the above that
\[
[z_{0},\ldots, z_{n}, G_{r+\varepsilon,s+\varepsilon}\circ F_{r,s}(z_{0}), \ldots, G_{r+\varepsilon,s+\varepsilon}\circ F_{r,s}(z_{n})]
\]
is a simplex of $\check{C}((X;r+2\varepsilon)\cup(Y;s+2\varepsilon))$.

This homotopy
\[
G_{r+\varepsilon,s+\varepsilon} \circ F_{r,s} \simeq \iota_{r,s;\varepsilon}
\]
restricts to a homotopy
\[
(G_{r+\varepsilon,s+\varepsilon}|\check{C}(W;s+\varepsilon)) \circ (F_{r,s}|\check{C}(Y;s))  \simeq \iota_{r,s;\varepsilon}|\check{C}(Y;s):\check{C}(Y;s) \hookrightarrow \check{C}(Y;s+2\varepsilon).
\]
Hence we obtain a homotopy of maps between pairs of spaces:
\[
G_{r+\varepsilon,s+\varepsilon} \circ F_{r,s} \simeq \iota_{r,s;\varepsilon}:(\check{C}((X;r)\cup (Y;s)),\check{C}(Y;s)) \to (\check{C}((X;r+2\varepsilon)\cup (Y;s+2\varepsilon)),\check{C}(Y;s+2\varepsilon)).
\]
Passing to relative homology gives the first half of the $\varepsilon$-interleaving; the other half is obtained symmetrically.

For $k$ clouds the same construction runs on the $k$-fold union complexes $\check{C}\big(\bigcup_j (X_j; t_j)\big)$: the vertex map $\bigsqcup_j f_j$ (each $f_j$ realizing $d_H(X_j, X'_j) \leq \varepsilon$) induces the interleaving simplicial maps, and these restrict to every sub-union $\bigcup_{j \neq i}(X_j; t_j)$, so the argument above gives an $\varepsilon$-interleaving of each leave-one-out summand $\big\{H_*\big(\bigcup_j \mathcal{U}(X_j; t_j),\, \bigcup_{j \neq i}\mathcal{U}(X_j; t_j)\big)\big\}_{\mathbf{t}}$.
The direct sum over $i$ is the $\varepsilon$-interleaving of $\mathrm{RPH}_*$ asserted in Theorem~\ref{thm:relative_stability}.
\end{proof}

\section{Consistency under Sampling}\label{sec:consistency}

A shipped invariant must recover the truth as the sample refines---property~(P2).
We record two complementary guarantees.
The first is qualitative and fully general: for arbitrary compact sets $A, B \subset \R^d$, the (relative) \v{C}ech (co)homology---hence the Betti numbers and Euler characteristic---of $A$, $B$, and the union pair $(A \cup B, B)$ is recovered in the limit from increasingly dense samples; this is the sampling-consistency companion to the relative module of Section~\ref{sec:relative}.
The second is quantitative: under a positive-reach hypothesis and i.i.d.\ sampling, the Intersection ECP recovers the topological interaction of the underlying shapes---persistently at every regular scale, and exactly (pointwise) at scales regular for the sample---after finitely many samples.

\subsection{Qualitative recovery in the inverse limit}\label{sec:consistency_qualitative}
For a compact subset $A$ of $\mathbb R^d$ and 
two sequences $\{r_i\}$ and $\{\delta_i \}$ of positive numbers with $\lim_i r_i = \lim _i \delta_i = 0$, \cite{kawamura} constructs a sequence of finite samples $\{X_i\}$ with $d_{\mathrm H}(X_{i},A) < \delta_{i}$ and of simplicial maps $\{f_{i}^{A}:\check{C}(X_{i+1};r_{i+1}) \to \check{C}(X_{i};r_{i})\}$
such that the inverse sequence
$\{ \check{C}(X_{i};r_{i}), f_{i}^{A}:\check{C}(X_{i+1})\to \check{C}(X_{i})\mid i\geq 1\}$ 
induces isomorphisms
\[
\check{H}_{\ast}(A) \cong \varprojlim \{H_\ast(\check{C}(X_{i};r_{i})); (f_{i}^{A})_\ast \mid i\geq 1 \},\quad
\check{H}^{\ast}(A) \cong \varinjlim \{H^\ast (\check{C}(X_{i};r_{i})); (f_{i}^A)^\ast \mid i\geq 1\}.
\]
The map $f_{i}^{A}:\check{C}(X_{i+1};r_{i+1}) \to \check{C}(X_{i};r_{i})$ is defined subject to the condition on the set $X_{i+1}$ of the vertices of $\check{C} (X_{i+1};r_{i+1})$:
$
\Vert f_{i}^{A}(x_{i+1}) - x_{i+1}\Vert < \delta_i \quad \forall x_{i+1} \in X_{i+1}.
$\\
By an appropriate choice of $r_i$, $\delta_i$ and $X_i$, the map $f_i^A$ is indeed a simplicial map and also we obtain the inclusion $\mathcal{U}(X_{i+1};r_{i+1}) \subset \mathcal{U}(X_{i},r_{i})$ for each $i$.
Furthermore we have the commutative diagrams below by an application of the functorial Nerve theorem 
\citep[Lemma 5.1]{virk2024dowkerdual} (see \cite[Lemma 2.2]{kawamura}):
\[
\xymatrix{
H_\ast(\check{C}(X_{i+1};r_{i+1})) \ar[d]_{(f_{i}^{A})_\ast} & H_\ast(\mathcal U(X_{i+1};r_{i+1})) \ar[l]_{(\iota_{i+1})_\ast} \ar[d]^{(\ell_{i}^{A})_\ast}
&
H^\ast(\check{C}(X_{i+1};r_{i+1})) \ar[r] ^{(\iota_{i+1})^\ast} & H^\ast(\mathcal U(X_{i+1};r_{i+1}))  
\\
H_\ast(\check{C}(X_{i};r_{i})) & H_\ast(\mathcal U(X_{i};r_{i}))  \ar[l]^{(\iota_{i})_{\ast}},
&
H^\ast(\check{C}(X_{i};r_{i})) \ar[u]^{(f_{i}^{A})^\ast} \ar[r]_{(\iota_{i})^\ast} & H^{\ast}(\mathcal U(X_{i};r_{i})) \ar[u]_{(\ell_{i}^{A})^\ast}
}
\]
Where $\iota_{i}$ and $\iota_{i+1}$ are homotopy equivalences of the Nerve theorem and $\ell_i^A$ is the inclusion.
%
These give isomorphisms
\[
\check{H}_{\ast}(A) \cong \varprojlim \{H_\ast(\mathcal{U}(X_{i};r_{i})); (\ell_{i}^{A})_\ast \mid i\geq 1 \},\quad
\check{H}^{\ast}(A) \cong \varinjlim \{H^\ast (\mathcal{U}(X_{i};r_{i})); (\ell_{i}^A)^\ast \mid i\geq 1\}.
\]

Likewise, for two compact subsets $A$ and $B$, we can construct a simplicial map
$\displaystyle
f_{i}^{A,B}: \check{C}( (X_{i+1};r_{i+1})\cup (Y_{i+1};s_{i+1})) \to 
\check{C}( (X_{i};r_{i})\cup (Y_{i};s_{i}))
$
in such a way that
$f_{i}^{A,B}(\check{C}(Y_{i+1};s_{i+1})) \subset \check{C}(Y_{i};s_{i})$ by recalling the defining condition of the map above. Thus $f_i^{A,B}$ is a map between pairs of spaces: 
\[
f_{i}^{A,B}:(\check{C}( (X_{i+1};r_{i+1})\cup (Y_{i+1};s_{i+1})), \check{C}(Y_{i+1};s_{i+1})) \to
(\check{C}( (X_{i};r_{i})\cup (Y_{i};s_{i})), \check{C}(Y_{i};s_{i})).
\]
Writing the inclusion
\[
(\mathcal{U}(X_{i+1};r_{i+1})\cup \mathcal{U}(Y_{i+1};s_{i+1}), \mathcal{U}(Y_{i+1};s_{i+1})) \hookrightarrow
(\mathcal{U}(X_{i};r_{i})\cup \mathcal{U}(Y_{i};s_{i}), \mathcal{U}(Y_{i};s_{i}))
\]
as $\ell_i^{A,B}$, we obtain two inverse sequences of pairs of spaces 
\begin{equation}\label{eq:invseq}
\begin{array}{ll}
\{  (\check{C}( (X_{i};r_{i})\cup (Y_{i};s_{i})), \check{C}(Y_{i};s_{i})),~ f_{i}^{A,B}\mid i\geq 1\}, ~\mbox{and} \\
\{ (\mathcal{U}(X_{i};r_{i})\cup \mathcal{U}(Y_{i};s_{i}), \mathcal{U}(Y_{i};s_{i})),~ \ell_{i}^{A,B}\mid i\geq 1\}
\end{array}
\end{equation}
which give isomorphisms \footnote{Here we use the fact that the inverse sequence (\ref{eq:invseq}) is an $\mathrm{HPOL}$-expansion of the compact pair $(A\cup B, B)$ 
\cite[Chap.I, Section 4]{mardesicsegalshape}. For cohomology the isomorphism also follows from the long exact sequence of pairs, Five lemma and the previous isomorphisms for absolute cohomology.}
\begin{eqnarray*}
\check{H}_{\ast}(A,A\cap B) &\cong&
\check{H}_{\ast}(A\cup B,B) \\
&\cong&
\varprojlim \{H_\ast(\check{C}( (X_{i};r_{i})\cup (Y_{i};s_{i})), \check{C}(Y_{i};s_{i})); (f_{i}^{A,B})_\ast \mid i\geq 1 \}\\
&\cong& 
\varprojlim \{H_\ast(\mathcal{U}(X_{i};r_{i})\cup \mathcal{U}(Y_{i};s_{i}), \mathcal{U}(Y_{i};s_{i})); (\ell_{i}^{A,B})_\ast \mid i\geq 1 \},\\
\check{H}^{\ast}(A,A\cap B) &\cong&
\check{H}^{\ast}(A\cup B,B), \\
&\cong&
\varinjlim \{H^\ast(\check{C}( (X_{i};r_{i})\cup (Y_{i};s_{i})), \check{C}(Y_{i};s_{i})); (f_{i}^{A,B})^\ast \mid i\geq 1 \} \\
&\cong&
\varinjlim \{H^\ast(\mathcal{U}(X_{i};r_{i})\cup \mathcal{U}(Y_{i};s_{i}), \mathcal{U}(Y_{i};s_{i})); (\ell_{i}^{A,B})^\ast \mid i\geq 1 \}
.
\end{eqnarray*}
The Euler characteristics and the Betti numbers of $A, B$ and the pair $(A\cup B, B)$ are also recovered from those of $\mathcal U(X_{i};r_{i})$ and $\mathcal U(Y_{i};s_{i})$ on the basis of the following lemma.
{Here the Betti numbers and Euler characteristic of a general compact set are defined through its \v{C}ech cohomology---$\beta_q(A) = \dim_{\mathbb F} \check{H}^q(A; \mathbb F)$ and $\chi(A) = \sum_q (-1)^q \beta_q(A)$, well defined whenever these groups are finite-dimensional (as they are for compact sets of positive reach, by Lemma~\ref{lem:reach_anr}).}
Due to the existence of the long exact sequence of pairs of compact spaces, \v{C}ech cohomology groups are more workable than \v{C}ech homology groups for general compact subsets, and algebraic inductive limits are easier to handle than algebraic projective limits; for this reason the lemma is stated for cohomology.
A similar consideration has been made in \cite{robins1999}.

Let $\{V_{i},f_{i}:V_{i}\to V_{i+1}\mid i\geq 1\}$
be an inductive sequence of $\mathbb F$-vector spaces and linear maps with $\dim V_{i} < \infty$ for each $i$, and let $V_{\infty} = \varinjlim\{V_{i}; f_{i}\mid i\geq 1\}$.
For $i\leq j,$ let $f_{ji}:V_{i}\to V_{j}$ be the composition
\[
f_{j-1}\circ \cdots \circ f_{i}:V_{i}\to V_{j}.
\]
For an arbitrarily fixed $n\geq 1$, the sequence
\[
\{ \dim f_{in}(V_{n})\mid i\geq n\}
\]
is non-increasing and let $\rho_{n} = \lim_{i\to\infty}\dim f_{in}(V_{n}) = \min_{i\geq n} \dim f_{in}(V_{n})$.
The sequence $\{\rho_{n}\mid n\geq 1\}$ is non-decreasing; set
\[
\rho_{\infty} = \lim_{n\to\infty}\rho_n = \sup_{n}\rho_n  \in [0,\infty].
\]

\begin{lemma}[Dimension of an inductive limit]\label{lem:inductive_limit}
With the notation above, if $\rho_{\infty} < \infty$ then $V_\infty$ is finite-dimensional and $\dim V_\infty = \rho_\infty$.
\end{lemma}
\begin{proof}
From the assumption, there exist an integer $N$ such that $\rho_{\infty} = \rho_{n} = \rho_{N} <\infty$ for each $n\geq N.$

Let $\omega_{1}, \ldots, \omega_{d}$ be linearly independent vectors of $V_\infty$.
Take an integer $n>N$ and vectors $v_{n,1},\ldots, v_{n,d}$ of $V_n$ such that
\[
f_{\infty n}(v_{n,i}) = \omega_i
\]
for each $i$.
For each $m\geq n$, we have
$f_{\infty m}(f_{m n}(v_{n,i})) = \omega_i $ for each $i$, which implies that the vectors $f_{mn}(v_{n,1}),\ldots, f_{mn}(v_{n,d})$ are linearly independent in $f_{mn}(V_{n})$.
Hence we have
$\dim f_{mn}(V_{n}) \geq d$ for each $m$ and thus
\[
\infty > \rho_{\infty} = \rho_{n} = \lim_{m}\dim f_{mn}(V_n ) \geq d.
\]
This implies that $\dim V_\infty$ is finite and 
$\dim V_{\infty} \leq \rho_\infty$.

To prove the reverse inequality, take an integer $I\geq N$ such that, for each $i\geq I$, we have
\[
\dim f_{iN}(V_{N}) = \dim f_{IN}(V_{N}) = \rho_{N} = \rho_{\infty}.
\]
For each $i\geq I$, it readily follows from the definition that the restriction
\[
f_{i}|f_{i~N}(V_{N}) : f_{i~N}(V_{N})\to f_{i+1~N}(V_N) 
\]
is a surjection between vector spaces of the same finite dimension, and hence is an isomorphism.
Hence the vector space $f_{I~N}(V_{N})$ embeds into the limit $V_\infty$ and hence
\[
\rho_{\infty} = \rho_{N} = \dim f_{I~N}(V_{N}) \leq \dim V_{\infty} .
\]
Together with $\dim V_\infty \leq \rho_\infty$ this gives $\dim V_\infty = \rho_\infty$.
\end{proof}

Applying the above to the compact subsets $A$ and $B$ at the beginning of this section, we obtain the following:

\begin{corollary}[Recovery of Betti numbers in the limit]\label{cor:cech_recovery}
For compact subsets $A$ and $B$ of $\mathbb R^d$ and cohomology over the fixed field $\mathbb F$, there are sequences of positive numbers $(r_i ), (s_ i)$, point clouds $(X_{i}), (Y_i )$ with $X_i \subset A$, $Y_i \subset B$, with inclusions
$\ell_{i}^{A}:\mathcal U(X_{i+1}; r_{i+1})
\hookrightarrow
\mathcal U(X_{i}, r_{i})
$
and 
$\ell_{i}^{A,B}:(\mathcal U(X_{i+1}; r_{i+1})\cup \mathcal{U}(Y_{i+1};s_{i+1}), \mathcal{U}(Y_{i+1};s_{i+1})) \hookrightarrow (\mathcal U(X_{i}; r_{i})\cup \mathcal{U}(Y_{i};s_{i}), \mathcal{U}(Y_{i};s_{i}))$ 
%
such that
\[
\begin{array}{ll}
\dim \check{H}^{q}(A) =
\lim_{n\to \infty}(\lim_{i\geq n}
\dim \operatorname{Im}(\ell_{ni}^{A})^{\ast}),   \\
\dim \check{H}^{q}(A,A\cap B) = 
\dim \check{H}^{q}(A\cup B,B) =
\lim_{n\to \infty}(\lim_{i\geq n}
\dim \operatorname{Im}(\ell_{ni}^{A,B})^{\ast}),
\end{array}
\]
if the right hand sides of the above are finite.
\end{corollary}

\bigskip

\subsection{Asymptotic consistency under bounded reach}\label{sec:asymptotic_consistency}

Section~\ref{sec:consistency_qualitative} recovers the (relative) homology of general compact sets qualitatively, in the inverse limit.
Under a positive-reach hypothesis one obtains more: the sample recovers the population interaction persistently, and---at scales regular for the sample filtration---the Intersection ECP equals it \emph{exactly} after finitely many samples.

{
\begin{lemma}[Positive reach gives finite-dimensional homology]\label{lem:reach_anr}
Let $A$ be a compact subset of $\R^d$ of positive reach $\tau > 0$. Then $A$ is a compact ANR, and its singular homology groups are finite-dimensional in every degree.
\end{lemma}
\begin{proof}
Fix $\rho < \tau$ and let $V_\rho(A)$ be the open $\rho$-neighborhood of $A$; as an open subset of $\R^d$ it is an ANR. Since $\rho < \tau$, each $x \in V_\rho(A)$ has a unique nearest point $\pi(x) \in A$, and $\pi : V_\rho(A) \to A$ is a continuous retraction, so $A$---a retract of an ANR---is itself an ANR. Every compact ANR has the homotopy type of a compact polyhedron \citep{west}, so its homology is finite-dimensional in every degree.
\end{proof}
}

{
Before the sampling theorems we record the probabilistic input precisely: the following lemma is the covering-plus-Borel--Cantelli argument that turns i.i.d.\ sampling into almost-sure density.

Recall from Theorem~\ref{thm:geometric_bounds} that $\mathcal{N}(A, \epsilon)$ denotes the $\epsilon$-covering number of a compact set $A$.
\begin{lemma}[Almost-sure covering under i.i.d.\ sampling]\label{lem:sampling_dense}
Let $A \subset \R^d$ be compact and $\epsilon > 0$; put $M = \mathcal{N}(A, \epsilon)$ and fix an $\epsilon$-net $\{a_1, \ldots, a_M\} \subseteq A$ (so $A \subseteq \bigcup_j B(a_j, \epsilon)$). Let $\mu$ be a Borel probability measure on $A$ with $p_\epsilon := \min_j \mu\big(A \cap B(a_j, \epsilon)\big) > 0$, and let $x_1, x_2, \ldots \stackrel{\mathrm{iid}}{\sim} \mu$. Write $X_n = \{x_1, \ldots, x_n\}$ and let $\mathcal{C}_n$ be the event that $X_n \cap \big(A \cap B(a_j, \epsilon)\big) = \emptyset$ for some $j \in \{1, \ldots, M\}$. Then:
\begin{enumerate}
    \item $\mathbb{P}[\mathcal{C}_n] \leq M(1-p_\epsilon)^n \leq M\exp\big(- n\,p_\epsilon\big)$ for every $n$;
    \item almost surely, $X_n$ is $2\epsilon$-dense in $A$ for all sufficiently large $n$.
\end{enumerate}
\end{lemma}
\begin{proof}
(1) For each $j$, independence gives $\mathbb{P}\big[X_n \cap (A \cap B(a_j, \epsilon)) = \emptyset\big] = \big(1 - \mu(A \cap B(a_j, \epsilon))\big)^n \leq (1-p_\epsilon)^n$; a union bound over the $M$ caps gives $\mathbb{P}[\mathcal{C}_n] \leq M(1-p_\epsilon)^n$, and $(1-t)^n \leq e^{-nt}$ for $t \geq 0$ yields the second inequality.
(2) By (1), $\sum_{n} \mathbb{P}[\mathcal{C}_n] \leq M \sum_n (1-p_\epsilon)^n < \infty$, so by the Borel--Cantelli lemma $\mathbb{P}[\mathcal{C}_n \text{ occurs for infinitely many } n] = 0$: almost surely only finitely many $\mathcal{C}_n$ occur, i.e.\ there is a (random) $n_0$ with $X_n \cap (A \cap B(a_j, \epsilon)) \neq \emptyset$ for all $n \geq n_0$ and all $j$. On that event every net point $a_j$ lies within $\epsilon$ of a sample, and as $\{a_j\}$ is $\epsilon$-dense in $A$, $X_n$ is $2\epsilon$-dense in $A$.
\end{proof}

Applying the lemma along a sequence $\epsilon \to 0$ (a countable intersection of almost-sure events) gives $d_H(X_n, A) \to 0$ almost surely---the density input to Theorem~\ref{thm:consistency}. The finite-sample Theorem~\ref{thm:quantitative} instead uses the tail bound~(1) directly, at a fixed sample size and confidence level, which is why its guarantee carries an explicit failure probability rather than holding almost surely.
}

\paragraph{Standing hypotheses for the sampling theorems.}
For $i = 1, \ldots, k$, let $A_i \subset \R^d$ be compact of positive reach $\tau_i > 0$, set $\tau = \min_i \tau_i$, and let $\mu_i$ be a Borel probability measure with $\mathrm{supp}(\mu_i) = A_i$ assigning positive mass to every ball that meets $A_i$ (e.g.\ a density bounded below on $A_i$).
By Lemma~\ref{lem:reach_anr}, each $A_i$ is then a compact ANR with finite-dimensional singular homology, agreeing with \v{C}ech homology.
Assume moreover that $\rho \mapsto T_\rho := \bigcap_i \mathcal{U}(A_i; \rho)$ has finitely many homological critical values in $(0, \tau)$ and that each $T_\rho$ has finite-dimensional homology.
Both sampling theorems below are stated under these hypotheses.

\begin{theorem}[Asymptotic consistency]\label{thm:consistency}
Under the standing hypotheses above, draw $X_i = \{x_{i,1}, \ldots, x_{i,m_i}\} \stackrel{\mathrm{iid}}{\sim} \mu_i$ and write $M_\rho = \bigcap_i \mathcal{U}(X_i; \rho)$.
Fix $r \in (0, \tau)$ regular for $\rho \mapsto T_\rho$. Then a.s.\ $\epsilon_{\max} := \max_i d_H(X_i, A_i) \to 0$, and for all large $\min_i m_i$:
\begin{enumerate}
    \item \textbf{Persistent recovery (unconditional).} The sample recovers the population Betti numbers persistently:
    \[
        \operatorname{rank}\big(H_q(M_r) \to H_q(M_{r + \epsilon_{\max}})\big) = \beta_q(T_r) \quad\text{for every } q,
    \]
    so the $\epsilon_{\max}$-persistent Euler characteristic $\sum_q (-1)^q \operatorname{rank}(H_q(M_r) \to H_q(M_{r+\epsilon_{\max}}))$ equals $\chi(T_r)$.
    \item \textbf{Exact pointwise recovery (at regular scales).} If $r$ is \emph{additionally} a regular value of the sample filtration $\rho \mapsto M_\rho$ (no homological critical value in $[r,\, r + \epsilon_{\max}]$), then $\Delta\chi(r; X_1, \ldots, X_k) = \chi(T_r)$ exactly.
\end{enumerate}
\end{theorem}

\begin{proof}
{$\epsilon_{\max} \to 0$ a.s.\ by Lemma~\ref{lem:sampling_dense}: each $\mu_i$ charges every ball meeting $A_i$, so the mass lower bound $p_\epsilon > 0$ holds, and the lemma makes $X_i$ almost surely eventually $2\epsilon$-dense for every fixed $\epsilon$; intersecting over a sequence $\epsilon \to 0$ gives $d_H(X_i, A_i) \to 0$ a.s., hence $\epsilon_{\max} = \max_i d_H(X_i, A_i) \to 0$ a.s.}

Fix $\delta > 0$ with $[r - \delta, r + \delta]$ free of critical values of $T$ ($r$ is $T$-regular). A.s.\ eventually $\epsilon := \epsilon_{\max} < \delta$; then $X_i \subseteq A_i$ gives the interleaving $T_{\rho - \epsilon} \subseteq M_\rho \subseteq T_\rho$ of Theorem~\ref{thm:discrimination}, so
\[
    T_{r - \epsilon} \;\subseteq\; M_r \;\subseteq\; T_r \;\subseteq\; M_{r + \epsilon} \;\subseteq\; T_{r + \epsilon},
\]
with $H_*(T_{r-\epsilon}) \xrightarrow{\cong} H_*(T_r) \xrightarrow{\cong} H_*(T_{r+\epsilon})$ induced by inclusion, since $T$ has no critical values on $[r-\delta, r+\delta] \supseteq [r-\epsilon, r+\epsilon]$ (recall $\epsilon < \delta$).
The first isomorphism factors through $H_*(M_r)$, so $H_*(M_r) \to H_*(T_r)$ is surjective; the second factors through $H_*(M_{r+\epsilon})$, so $H_*(T_r) \to H_*(M_{r+\epsilon})$ is injective.
The sample map $H_*(M_r) \to H_*(M_{r+\epsilon})$ equals (injective)$\,\circ\,$(surjection onto $H_*(T_r)$), hence has rank $\dim H_*(T_r) = \beta_*(T_r)$---this is~(1).
For~(2), $M$-regularity on $[r, r+\epsilon]$ makes $H_*(M_r) \to H_*(M_{r+\epsilon})$ itself an isomorphism, and (1)'s rank equality then forces $H_*(M_r) \cong H_*(T_r)$; Euler characteristics agree, and the Intersection Theorem gives $\Delta\chi(r) = \chi(M_r) = \chi(T_r)$, an exact integer equality.
\end{proof}

Thus $\Delta\chi$ consistently recovers the underlying interaction $\chi(T_r)$---persistently at every $T$-regular scale, and pointwise-exactly wherever the sample filtration is also regular, its integer range turning approximate stability into exact equality once the perturbation is small enough---so on finite windows it serves as a proxy for the true topological interaction, off a set of sample-artifact scales that the next subsection bounds.
Explicit finite-sample rates follow in Section~\ref{sec:sample_complexity}.

\subsection{Quantitative sample complexity}\label{sec:sample_complexity}

Theorem~\ref{thm:consistency} guarantees eventual exact recovery but not \emph{how many} samples suffice.
We give explicit bounds in the spirit of Niyogi, Smale, and Weinberger \citep{niyogi2008finding}, who showed that $O(\tau^{-k}\log(1/\delta))$ samples from a set of reach $\tau$ recover its homology with probability $\geq 1 - \delta$.
{Throughout, $\mathrm{Crit}(T)$ denotes the set of homological critical values of the population filtration $\rho \mapsto T_\rho = \bigcap_i \mathcal{U}(A_i; \rho)$---the scales at which the \v{C}ech cohomology of $T_\rho$ changes---finite by the standing hypotheses stated before Theorem~\ref{thm:consistency}.}

\begin{theorem}[Quantitative convergence]\label{thm:quantitative}
For each $i$, let $A_i \subset \R^d$ be compact with reach $\tau_i > 0$ and finite $k_i$-dimensional Hausdorff volume $\mathrm{vol}_{k_i}(A_i)$ (so that $\dim_H A_i = k_i$, the density hypothesis below forcing $\mathcal{H}^{k_i}(A_i) > 0$ as well), and let $\mu_i$ be a probability measure on $A_i$ with density bounded below by $f_i > 0$ with respect to the $k_i$-dimensional Hausdorff measure.
Draw $X_i \stackrel{\mathrm{iid}}{\sim} \mu_i$ of size $m_i$, write $M_\rho = \bigcap_i \mathcal{U}(X_i; \rho)$ and $T_\rho = \bigcap_i \mathcal{U}(A_i; \rho)$, set $\tau = \min_i \tau_i$, and assume (as in Theorem~\ref{thm:consistency}) that $\rho \mapsto T_\rho$ has finitely many critical values.
Fix $\delta > 0$ and a scale $r \in (0, \tau)$ regular for the population filtration, and put $\epsilon = \tfrac12 \min\big(r,\ \tau - r,\ \mathrm{dist}(r, \mathrm{Crit}(T))\big)$, so the window $[r-\epsilon, r+\epsilon]$ is free of population critical values.
If, for every $i$,
\begin{equation}\label{eq:m}
        m_i \;\geq\; \frac{2^{k_i}}{f_i\, c_{k_i}\, \epsilon^{k_i}} \left(k_i \log \frac{2}{\epsilon} + \log \frac{\mathrm{vol}_{k_i}(A_i)}{\omega_{k_i}\, \delta}\right),
\end{equation}
where $\omega_k$ is the volume of the unit $k$-ball and $c_{k_i} > 0$ is the dimensional cap-mass constant of \citet[Lemma~5.3]{niyogi2008finding}, then with probability $\geq 1 - k\delta$ the sample is $\epsilon$-dense, and on that event $\Delta\chi$ recovers the population interaction persistently,
\[
    \sum_q (-1)^q \operatorname{rank}\big(H_q(M_r) \to H_q(M_{r+\epsilon})\big) \;=\; \chi\big(\textstyle\bigcap_i \mathcal{U}(A_i; r)\big)
\]
(Theorem~\ref{thm:consistency}(1)), with pointwise-exact equality $\Delta\chi(r; X_1, \ldots, X_k) = \chi(\bigcap_i \mathcal{U}(A_i; r))$ whenever $r$ is in addition a regular value of the sample filtration.
\end{theorem}

\begin{proof}
Each single-cloud tail bound $\Pr[d_H(X_i, A_i) > \epsilon] \leq \delta$ is a covering argument. Take a minimal $(\epsilon/2)$-covering of $A_i$, of $N_i \leq \mathrm{vol}_{k_i}(A_i)/(\omega_{k_i}(\epsilon/2)^{k_i})$ balls: every point of $A_i$ is within $\epsilon/2$ of a net center, so if each cap $B(c, \epsilon/2) \cap A_i$ contains a sample then $d_H(X_i, A_i) \leq \epsilon$. Each cap has $\mu_i$-mass $\geq f_i c_{k_i}(\epsilon/2)^{k_i}$ (the $c_{k_i}$ absorbing the reach and boundary corrections; the flat value $\omega_{k_i}2^{-k_i}$ is unattainable on curved or bounded $A_i$), so it is missed by all $m_i$ samples with probability $\leq (1 - f_i c_{k_i}(\epsilon/2)^{k_i})^{m_i} \leq e^{-m_i f_i c_{k_i}(\epsilon/2)^{k_i}}$; a union bound over the $N_i$ caps gives $\Pr[d_H(X_i, A_i) > \epsilon] \leq \delta$ under the stated $m_i$.

A further union bound over the $k$ clouds makes $\max_i d_H(X_i, A_i) \leq \epsilon$ hold with probability $\geq 1 - k\delta$; since $\epsilon \leq \tfrac12\,\mathrm{dist}(r, \mathrm{Crit}(T))$ the window $[r-\epsilon, r+\epsilon]$ is critical-free, so Theorem~\ref{thm:consistency}(1) gives the persistent recovery, and Theorem~\ref{thm:discrimination} the pointwise-exact equality at any $M$-regular $r$.
\end{proof}

\begin{remark}[Interpreting the sample complexity]\label{rem:nsw_interpret}
The required sample size scales as $O(\epsilon^{-k_i}\log(1/\delta))$ in the \emph{intrinsic} dimension $k_i$ of each $A_i$, not the ambient $d$---the Niyogi--Smale--Weinberger advantage over naive covering in $\R^d$.
The threshold $\epsilon < \tau/2$ is the ``sampling density beyond twice the feature size'' condition of \citet{niyogi2008finding}: below it each $\mathcal{U}(X_i; r)$ matches $\mathcal{U}(A_i; r)$, and the Intersection Theorem (Theorem~\ref{thm:intersection}) carries this to $\Delta\chi(r)$.
\end{remark}

\begin{corollary}[Uniform recovery over a scale band]\label{cor:uniform}
Under the hypotheses of Theorem~\ref{thm:quantitative}, and assuming as in Theorem~\ref{thm:consistency} that $T$ has finitely many critical values in the band, fix $[r_{\min}, r_{\max}] \subset (0, \tau)$ and $\epsilon < \min(r_{\min}, \tau - r_{\max})/2$. If every $m_i$ and $\delta$ satisfy~\eqref{eq:m}, then with probability $\geq 1 - k\delta$
\[
    \Delta\chi(r; X_1, \ldots, X_k) = \chi\big(\textstyle\bigcap_i \mathcal{U}(A_i; r)\big) \qquad \text{for all } r \in [r_{\min}, r_{\max}] \setminus (E_T \cup E_M)
\]
simultaneously, where
\begin{itemize}
    \item $E_T$ is the union of the closed $2\epsilon$-windows $[c-\epsilon, c+\epsilon]$ about the critical values $c$ of $T$ in the band, of Lebesgue measure $\leq 2\epsilon\,\#\mathrm{Crit}(T) = O(\epsilon)$;
    \item $E_M$ is the (sample-dependent) union of the scales straddled by an ephemeral feature of the sample filtration $M$, each an interval of length $\leq 2\epsilon$.
\end{itemize}
\end{corollary}

\begin{proof}
On $\{\max_i d_H(X_i, A_i) \leq \epsilon\}$ (probability $\geq 1 - k\delta$) the filtrations $M, T$ are $\epsilon$-interleaved (Theorem~\ref{thm:discrimination}); the choice of $\epsilon$ places $[r_{\min}, r_{\max}]$ inside $(\epsilon, \tau - \epsilon)$. Off $E_T$ each scale has $\mathrm{dist}(r, \mathrm{Crit}(T)) > \epsilon$, hence a critical-free closed $\epsilon$-window; off $E_M$ it is regular for $M$; so off $E_T \cup E_M$ Theorem~\ref{thm:discrimination} gives exact recovery. Every scale in $E_M$ is straddled by a sample bar of persistence $\leq 2\epsilon$: a longer bar is matched by the interleaving to a population bar (Theorem~\ref{thm:algebraic_stability}(1)) and so marks a genuine, not ephemeral, feature.
\end{proof}

\begin{remark}[The sample-artifact set]\label{rem:artifact_set}
The population part $E_T$ has measure $O(\epsilon)$ unconditionally. The sample part $E_M$ collects transient artifacts of the finite sample; its total length is $2\epsilon$ times the number of ephemeral sample bars meeting the band, an almost-surely finite quantity {under the i.i.d.\ sampling model of Theorem~\ref{thm:quantitative} (finitely many samples give a finite filtration, hence finitely many bars)}. Bounding it \emph{uniformly in the sample size}---hence $|E_M| = O(\epsilon)$---reduces to controlling the transient features of the intersection filtration, which the covering-number bound of Theorem~\ref{thm:geometric_bounds} supplies in the reach-controlled regime. Two guarantees need no such control: the persistent recovery of Theorem~\ref{thm:consistency}(1) holds at \emph{every} $T$-regular scale unconditionally, and removing $E_M$ at a single prescribed scale is exactly the sample-regularity event of Theorem~\ref{thm:consistency}(2).
\end{remark}

\begin{remark}[Exact versus bottleneck stability]\label{rem:ph_connection}
Corollary~\ref{cor:uniform} is the Euler-characteristic analogue of persistence-diagram stability under Hausdorff perturbation \citep{chazal2016structure}---but stronger: where the bottleneck distance is only \emph{bounded} by the Hausdorff distance, the integer-valued $\Delta\chi$ attains \emph{exact equality} over a band of scales---off the $O(\epsilon)$-measure critical-scale set $E_T$ and the transient sample artifacts $E_M$ of Remark~\ref{rem:artifact_set}---after finitely many samples.
{The exactness is more than integer-valuedness alone: it is the \emph{exact} persistent-recovery identity of Theorem~\ref{thm:consistency}(1), which pins the integer $\Delta\chi$ to $\chi(T_r)$ rather than merely bounding a distance---integer-valuedness is what then makes ``close'' mean ``equal.''}
\end{remark}

\section{Realizations, Algorithm, and Complexity}\label{sec:algorithm}

\subsection{Realizations of the Intersection ECP}\label{sec:concrete}

 The Intersection ECP $\mathcal{M}$ admits concrete realizations via specific computational models.
We derive these for $k = 2$; the extension to general $k$ via inclusion-exclusion is straightforward.

\paragraph{Alpha complex realization.}
The first realization is the one we compute with: it expresses $\Delta\chi$ through the (weighted) Alpha complexes of $X$, $Y$, and $X \cup Y$, reducing the Euler integral to a count of simplices.
Recall the \textbf{Alpha complex} $\Alpha(\mathcal{P}; r)$ of a finite cloud $\mathcal{P} \subset \R^d$: the subcomplex of the Delaunay triangulation of $\mathcal{P}$ formed by the simplices of circumradius $\le r$.
It has dimension $\le d$ (its simplices are Delaunay, so have at most $d+1$ vertices; the general-position assumption below makes it a genuine simplicial complex), and by the Nerve Theorem is homotopy equivalent to the ball union $\mathcal{U}(\mathcal{P}; r)$, so $\chi_{\mathcal{P}}(r) := \chi(\Alpha(\mathcal{P}; r)) = \chi(\mathcal{U}(\mathcal{P}; r))$ is computed as the signed simplex count $\sum_{q} (-1)^q \sigma_q(r)$, $\sigma_q(r)$ the number of $q$-simplices at scale $r$.  Throughout we assume the combined point set $X \cup Y$ is in \emph{general position}---no $d+1$ points affinely dependent, and no $d+2$ points cospherical, in the ordinary sense on the diagonal $r=s$ and in the \emph{weighted (power)} sense when $r\neq s$ (no $d+2$ sites on a common sphere in the power metric---the \emph{power distance} of a point $x$ to a weighted site $z$ of weight $w(z)$ being {$\pi_z(x) = \|x - z\|^2 - w(z)$}---equivalently no point power-equidistant from more than $d+1$ sites)---so that the (weighted) Delaunay triangulation is well defined and each Alpha complex is simplicial of dimension $\le d$.
This is generic and can be enforced symbolically by Simulation of Simplicity \citep{edelsbrunner1990simulation}, which perturbs the sites and leaves the computed Euler characteristic unchanged.

\begin{proposition}[Alpha complex formula ($k = 2$)]\label{prop:alpha_formula}
For disjoint $X, Y \subset \R^d$ and scales $r, s \geq 0$:
\begin{equation}\label{eq:alpha_formula}
    \Delta\chi(r, s;\, X, Y) = \chi_X(r) + \chi_Y(s) - \chi(\mathcal{U}(X; r) \cup \mathcal{U}(Y; s)),
\end{equation}
where $\chi_X(r) = \chi(\Alpha(X; r))$ and the union term is computed via the weighted Alpha complex with weights $w(p) = r^2$ for $p \in X$ and $w(p) = s^2$ for $p \in Y$.
On the diagonal $r = s$, this reduces to $\Delta\chi(r) = \chi_X(r) + \chi_Y(r) - \chi_{X \cup Y}(r)$ via the standard Alpha complex.
\end{proposition}

\begin{proof}
By inclusion--exclusion for the valuation $\chi$ (Section~\ref{sec:euler_integral}), $\Delta\chi = \chi_X + \chi_Y - \chi^{\cup}$, and the Nerve Theorem identifies each $\chi_{\mathcal{P}}(r)$ with $\chi(\Alpha(\mathcal{P}; r))$.
\end{proof}
For general $k$, inclusion--exclusion gives the exact $(2^k - 1)$-term formula
\[
    \Delta\chi(\mathbf{t};\, X_1, \ldots, X_k) = \sum_{\emptyset \neq S \subseteq [k]} (-1)^{|S|+1} \chi\!\left(\bigcup_{i \in S} \mathcal{U}(X_i; t_i)\right),
\]
expensive for large $k$ but exact; in practice small $k$ (typically $k = 2$) suffices.

\paragraph{Nerve realization.}
A second realization trades computational efficiency for geometric transparency.
For $k = 2$ the intersection decomposes into convex lens regions,
\[
    \mathcal{U}(X; r) \cap \mathcal{U}(Y; s) = \bigcup_{\substack{x \in X,\, y \in Y \\ \|x - y\| < r + s}} \big(B(x, r) \cap B(y, s)\big),
\]
each contractible ($\chi = 1$) and present only once $\|x - y\| < r + s$, the clouds' combined reach; modeling the intersection by the nerve of these lenses exposes \emph{where} the clouds interact, not just the value of $\Delta\chi$.

From the cell decomposition (Lemma~\ref{lem:cell_decomp}) and the Intersection Theorem (Theorem~\ref{thm:intersection}), we obtain the following.

\begin{proposition}[Nerve formula for Intersection EC ($k = 2$)]\label{prop:nerve}
Let $\mathcal{L}(\mathbf{t}) = \{L_{x,y} = B(x,r) \cap B(y,s) : (x,y) \in X \times Y,\, \|x - y\| < r + s\}$ be the collection of non-empty cross-cloud lenses.
Then
\begin{equation}\label{eq:nerve_formula}
    \Delta\chi(r,s;\, X, Y) = \chi(\mathcal{N}(\mathcal{L}(\mathbf{t}))),
\end{equation}
where $\mathcal{N}(\mathcal{L}(\mathbf{t}))$ is the nerve complex of the lens covering.
\end{proposition}

\begin{proof}
For $k = 2$ the lenses $\mathcal{L}(\mathbf{t})$ are precisely the non-empty convex cells of Lemma~\ref{lem:cell_decomp}, so $\bigcap_i \mathcal{U}(X_i; t_i) \simeq \mathcal{N}(\mathcal{L}(\mathbf{t}))$; the Intersection Theorem then gives $\Delta\chi(r,s) = \chi(\mathcal{N}(\mathcal{L}(\mathbf{t})))$.
\end{proof}

For general $k$, each lens is a $k$-fold intersection $\bigcap_i B(x_i, t_i)$ over a cross-tuple $(x_1, \ldots, x_k) \in \prod_i X_i$, still convex, so the Nerve Theorem continues to apply, giving the analogous nerve model for each floor $C_j(\mathbf{t})$ (Remark~\ref{rem:spectrum_nerve}).

\begin{proposition}[Vertices of the lens nerve]\label{thm:lens_complexity}
For disjoint point clouds $X, Y \subset \R^d$ with $|X| = m$, $|Y| = n$, the nerve $\mathcal{N}(\mathcal{L}(\mathbf{t}))$ of the cross-cloud lens covering (Proposition~\ref{prop:nerve}) has vertex set in bijection with the active cross-pairs:
\begin{equation}\label{eq:nerve_complexity}
    \#\{\text{vertices of } \mathcal{N}(\mathcal{L}(\mathbf{t}))\} = |\{(x,y) \in X \times Y : \|x-y\| < r+s\}| \leq mn.
\end{equation}
Its higher-dimensional simplices are the sub-families of lenses with a common point, and their number is \emph{not} polynomially bounded: if $\ell$ lenses share a common region they span an $(\ell-1)$-simplex with $2^\ell - 1$ nonempty faces, and $\ell$ can be $\Theta(m+n)$ even in general position.
\end{proposition}

\begin{proof}
A vertex is a non-empty lens $B(x,r) \cap B(y,s)$, which is non-empty iff $\|x-y\| < r+s$, giving \eqref{eq:nerve_complexity}.
A $p$-simplex is a set of $p+1$ lenses with $\bigcap_{j}(B(x_j,r) \cap B(y_j,s)) \neq \emptyset$; placing $m+n$ centers so that the origin lies in every ball ($\|x_j\|<r$, $\|y_j\|<s$) makes all $mn$ lenses share the origin, so every sub-family is a simplex and $|\mathcal{N}| = 2^{mn} - 1$.
This already occurs in general position, so no polynomial bound holds.
\end{proof}

The lens nerve is thus \emph{not} a competitive computational route for $\Delta\chi$: its simplex count can be exponential when many lenses overlap, whereas $\chi$ of the same intersection is obtained from the Alpha complex of $X \cup Y$---of worst-case size $O((m+n)^{\lceil d/2 \rceil})$ (Theorem~\ref{thm:algorithm_complexity})---by the inclusion--exclusion formula (Proposition~\ref{prop:alpha_formula}).
We therefore compute $\Delta\chi$ via the Alpha complex; the lens nerve is retained only for geometric interpretation, its $1$-skeleton (the active cross-pairs and their overlaps) encoding the spatial pattern of the interaction, not just its Euler characteristic.

\subsection{The Alpha-Complex Algorithm}

The Alpha-complex realization (Proposition~\ref{prop:alpha_formula}) turns the Intersection ECP into a direct computation. Its defining feature---and the source of its efficiency relative to the mixup barcode---is that the Euler characteristic of an Alpha complex is a \emph{signed simplex count}: it needs a single pass over the simplices and \emph{no} persistence (boundary-matrix) reduction. We give the algorithm on the diagonal for $k = 2$; the off-diagonal and $k$-cloud cases follow by the same template.

For a finite cloud $\mathcal{P} \subset \R^d$, the Alpha filtration assigns to each Delaunay simplex $\sigma$ a birth scale $a(\sigma)$---the least $r$ with $\sigma \in \Alpha(\mathcal{P}; r)$. The Euler characteristic curve is then
\begin{equation}\label{eq:ecc_sweep}
    \chi_{\mathcal{P}}(r) = \chi(\Alpha(\mathcal{P}; r)) = \sum_{\sigma:\, a(\sigma) \leq r} (-1)^{\dim \sigma},
\end{equation}
a step function that starts at $0$ and changes by $(-1)^{\dim\sigma}$ as $r$ crosses each $a(\sigma)$. It is obtained by sorting the simplices on birth scale and accumulating the alternating sum---no homology is computed. Algorithm~\ref{alg:mixup} applies this to $X$, $Y$, and $X \cup Y$ and combines them by inclusion--exclusion.

\begin{algorithm}[t]
\caption{Diagonal Intersection ECP $\Delta\chi(\cdot;\, X, Y)$}\label{alg:mixup}
\KwIn{point clouds $X, Y \subset \R^d$}
\KwOut{the step function $\Delta\chi(\cdot;\, X, Y) : \R_{\geq 0} \to \Z$}
\ForEach{$\mathcal{P} \in \{X,\ Y,\ X \cup Y\}$}{
  compute the Alpha filtration of $\mathcal{P}$ (a standard primitive, from the Delaunay triangulation; e.g.\ \citep{gudhi2024}): Delaunay simplices $\sigma$ with birth scales $a(\sigma)$\;

  sort the simplices by $a(\sigma)$\;
  sweep, recording $\chi_{\mathcal{P}}$ as the running alternating sum $\sum_{a(\sigma)\le r}(-1)^{\dim\sigma}$ of \eqref{eq:ecc_sweep}\;
}
\Return $\Delta\chi(r) = \chi_X(r) + \chi_Y(r) - \chi_{X\cup Y}(r)$ on the merged set of critical scales\;
\end{algorithm}

\begin{theorem}[Correctness]\label{thm:algorithm_correct}
For every $r \geq 0$, Algorithm~\ref{alg:mixup} returns $\Delta\chi(r) = \chi(\mathcal{U}(X;r) \cap \mathcal{U}(Y;r))$.
\end{theorem}
\begin{proof}
The sweep computes $\chi_{\mathcal{P}}(r) = \chi(\Alpha(\mathcal{P}; r)) = \chi(\mathcal{U}(\mathcal{P}; r))$ exactly (Nerve Theorem), so by the Intersection Theorem (Theorem~\ref{thm:intersection}) the returned $\chi_X(r) + \chi_Y(r) - \chi_{X\cup Y}(r)$ equals $\chi(\mathcal{U}(X;r) \cap \mathcal{U}(Y;r))$.
\end{proof}

\begin{theorem}[Complexity]\label{thm:algorithm_complexity}
For $|X| + |Y| = n$ in fixed dimension $d$, Algorithm~\ref{alg:mixup} runs in $O(n^{\lceil d/2\rceil} \log n)$ time and $O(n^{\lceil d/2\rceil})$ space, forming no boundary matrix.
\end{theorem}

\begin{proof}
The three Alpha complexes are subcomplexes of the Delaunay triangulations of $X$, $Y$, and $X\cup Y$.
By the Upper Bound Theorem, the Delaunay triangulation of $m$ points in $\R^d$ has $O(m^{\lceil d/2\rceil})$ faces: lifting the points to the paraboloid in $\R^{d+1}$ identifies its simplices with the lower faces of a convex polytope on $m$ vertices, whose face count is bounded by the cyclic-polytope maximum $O(m^{\lceil d/2\rceil})$ \citep{edelsbrunner2010computational}.
This is where the naive $m^{d+1}$ count of all possible simplices is \emph{halved} in the exponent: the empty-circumball (Delaunay) property makes the complex that sparse.
Hence each complex has $N = O(n^{\lceil d/2\rceil})$ faces.
It is built by the standard lift-and-hull algorithm in $O(n\log n + n^{\lceil d/2\rceil})$ time---the $O(n\log n)$ term for sorting/point location and the $O(n^{\lceil d/2\rceil})$ term to output the faces \citep{edelsbrunner2010computational}.
Sorting these $N$ faces by birth scale for the sweep is a comparison sort of $N$ items, costing $O(N\log N) = O(n^{\lceil d/2\rceil}\log n)$ for fixed $d$.
The sweep itself is linear in $N$, and the output has $O(N)$ critical scales.
No homology is computed, so no $N \times N$ boundary matrix is formed.
\end{proof}

The decisive contrast is with the mixup barcode of \citet{wagner2024mixup}: image-persistence reduction costs $O(N^3)$ in the simplex count---$O(N^\omega)$ with fast matrix multiplication, $\omega < 2.373$---whereas the Euler characteristic needs only the signed count of \eqref{eq:ecc_sweep}. The Euler characteristic beats the $O(N^3)$ reduction for \emph{every} $d \geq 2$.
In the moderate dimensions $d = 5$--$10$ the gap---a factor $N^2 = n^{2\lceil d/2\rceil}$, i.e.\ from $n^{6}$ to $n^{10}$---makes avoiding the reduction not merely faster but decisive.

This cost is essentially optimal.

\begin{proposition}[Optimality]\label{prop:hardness}
For even $d$, the exponent in Theorem~\ref{thm:algorithm_complexity} cannot be lowered: any algorithm that materializes the Alpha filtration must read its $\Theta(n^{\lceil d/2\rceil})$ faces, and at its largest scale the filtration is the full Delaunay triangulation of $X \cup Y$, whose boundary is the convex hull---and computing the convex hull of $n$ points in $\R^d$ requires $\Omega(n^{\lfloor d/2\rfloor})$ operations in the algebraic decision-tree model \citep{edelsbrunner2010computational}.

Since $\lfloor d/2\rfloor = \lceil d/2\rceil$ for even $d$, Algorithm~\ref{alg:mixup} is worst-case optimal there; for odd $d$ a factor-$n$ gap remains, the same gap as between convex-hull and Delaunay complexity.
(No matching lower bound is claimed for $\Delta\chi$ at a \emph{single} scale, which may be easier.)
\end{proposition}

\paragraph{Off-diagonal and $k$ clouds.} Off the diagonal, the union term $\chi(\mathcal{U}(X;r) \cup \mathcal{U}(Y;s))$ is the Euler characteristic of a \emph{weighted} Alpha complex on $X \cup Y$ with weights $r^2$ on $X$ and $s^2$ on $Y$; the same sweep over its (now planar) critical structure yields the full surface $\Delta\chi(r,s)$.
For $k \geq 3$, inclusion--exclusion (Section~\ref{sec:concrete}) gives
\[
    \Delta\chi(\mathbf{t}) = \sum_{\emptyset \neq S \subseteq [k]} (-1)^{|S|+1}\, \chi\Big(\bigcup_{i\in S}\mathcal{U}(X_i; t_i)\Big),
\]
each term an Euler characteristic of a union computed as above, for a total of $O(2^k\, n^{\lceil d/2\rceil}\log n)$. 

\section{Discussion}\label{sec:discussion}

We have isolated a single integer-valued profile, the Intersection ECP $\Delta\chi(\mathbf{t}) = \chi(\bigcap_i \mathcal{U}(X_i; t_i))$, and shown it is not an arbitrary choice but the one the axioms force: among pointwise-Euler interaction profiles it is pinned by separation and normalization (Theorem~\ref{thm:universality}), and dropping even the Euler-integral hypothesis, $\chi$ itself is forced by a multivaluation argument (Theorem~\ref{thm:universality_strong}). Around this invariant we built the theory a measurement needs: bottleneck and $L^1$ stability under perturbation (Theorem~\ref{thm:algebraic_stability}), recovery of the underlying topological interaction under sampling---persistently at every regular scale, and exactly at scales regular for the sample (Theorem~\ref{thm:consistency}, Corollary~\ref{cor:uniform})---and a reduction-free Alpha-complex algorithm running in $O(n^{\lceil d/2\rceil}\log n)$, worst-case optimal in even dimensions among filtration-materializing algorithms (Theorem~\ref{thm:algorithm_complexity}, Proposition~\ref{prop:hardness}). Where the Euler characteristic cancels, the relative-homology refinement of Section~\ref{sec:relative} recovers what it discards.

The unifying view is that $\Delta\chi$ is a \emph{measurement instrument} rather than the richest available descriptor: the coarsest interaction invariant that is at once symmetric, native to every $k$, exactly recoverable, axiomatically forced, and cheap enough to compute and test at scale (Section~\ref{sec:relation_invariants}). That trade is what the applied programs need---regime detection in dynamical systems \citep{majhi2026regime}, and the measurement of class disentanglement in neural representations---neither of which the reduction-bound, pairwise, untested mixup barcode or chromatic six-pack can serve at population scale.

\paragraph{Limitations.} The coarsening is real, not free. A \emph{structural degeneracy}---a balanced overlap such as an annulus or torus shell---has $\Delta\chi = 0$ across a whole band of scales and is invisible to the profile (Example~\ref{rem:degeneracy}, Figure~\ref{fig:saturation}); the relative module detects it, but at the cost of the computational simplicity that motivated $\Delta\chi$. On the sampling side, our pointwise exact-recovery guarantee holds at scales regular for the \emph{random} sample filtration; upgrading it to hold at an arbitrary prescribed scale, equivalently making the uniform exceptional set $O(\epsilon)$ unconditionally rather than only in the reach-controlled regime, requires controlling the ephemeral sample features and is left open (Remark~\ref{rem:artifact_set}).
The sharp component bound carries a no-early-merging hypothesis (Theorem~\ref{thm:geometric_bounds}), and the algorithm's optimality has a factor-$n$ gap in odd dimension (Proposition~\ref{prop:hardness}).

\paragraph{Open problems.} Four stand out.
\emph{(i) A faithful $k$-cloud relative invariant.} The leave-one-out modules of Section~\ref{sec:relative} capture each cloud's marginal contribution but not the genuinely higher-order interaction carried by the triple and deeper intersections; whether the Mayer--Vietoris spectral sequence of the cover $\{\mathcal{U}(X_i; t_i)\}$ assembles these into a faithful invariant is open.
\emph{(ii) An analytic null.} The asymptotic distribution of $\max_r|\Delta\chi|$ under exchangeability---approachable via random-complex limit theory---would calibrate an interaction test without permutations; it is the principal open statistical problem of this line.
\emph{(iii) Sharper sampling.} A deformation-retraction argument controlling the transient features of intersections of unions would remove the sample-regularity caveat above and make the recovery unconditional at every fixed scale.
\emph{(iv) Cheaper $k$-fold computation.} Whether the $2^k$ inclusion--exclusion factor can be removed---e.g.\ by a single weighted power complex on $\bigcup_i X_i$ carrying all $k$ offsets at once---is open; the lens nerve does not achieve it (Proposition~\ref{thm:lens_complexity}).

\appendix
 
\section{Inclusion--Exclusion for Valuations}\label{app:incl-excl}
 
The proof of Theorem~\ref{thm:universality_strong} expands a multivaluation over
finite intersections of balls in each of its arguments.
The underlying fact is the
$m$-fold form of the valuation relation
$\mathbf{1}_{A\cup B}+\mathbf{1}_{A\cap B}=\mathbf{1}_A+\mathbf{1}_B$ of
Section~\ref{sec:euler_integral}; this result is well known \citep{klain1997introduction}, but we record it here for completeness.
 
\begin{lemma}[Inclusion--exclusion for valuations]\label{lem:incl-excl}
Let $\mathcal{L}$ be a lattice of sets (closed under finite $\cup$ and $\cap$),
$G$ an abelian group, and $v\colon \mathcal{L}\to G$ a valuation:
\[
  v(U\cup V)+v(U\cap V)=v(U)+v(V)\qquad(U,V\in\mathcal{L}),
\]
with $v(\emptyset)=0$.
Then for any $C_1,\dots,C_m\in\mathcal{L}$,
\[
  v\!\Big(\bigcup_{i=1}^m C_i\Big)
  =\sum_{\emptyset\neq S\subseteq[m]}(-1)^{|S|+1}\,
   v\!\Big(\bigcap_{i\in S}C_i\Big).
\]
\end{lemma}
 
\begin{proof}
Induction on $m$; the case $m=1$ is trivial.
Write
$A=\bigcup_{i=1}^{m-1}C_i$.
The valuation identity gives
\[
  v(A\cup C_m)=v(A)+v(C_m)-v(A\cap C_m),
  \qquad
  A\cap C_m=\bigcup_{i=1}^{m-1}(C_i\cap C_m).
\]
Apply the inductive hypothesis to the $m-1$ sets $C_i$ and to the $m-1$ sets
$C_i\cap C_m$, using
$\bigcap_{i\in S}(C_i\cap C_m)=\bigcap_{i\in S\cup\{m\}}C_i$.
Each nonempty
$S'\subseteq[m]$ then contributes exactly once with sign $(-1)^{|S'|+1}$:
those with $S'\subseteq[m-1]$ from $v(A)$, the singleton $S'=\{m\}$ from
$v(C_m)$, and those with $m\in S'$, $S'\neq\{m\}$ from $-v(A\cap C_m)$ (the sign
flips since $-(-1)^{|S|+1}=(-1)^{|S\cup\{m\}|+1}$).
\end{proof}

\begin{proof}[Proof of Theorem~\ref{thm:universality_strong}]
Write each $A_i = \bigcup_{c=1}^{m_i} B_{i,c}$ as a finite union of balls, and set $B_{i,S_i} = \bigcap_{c \in S_i} B_{i,c}$ for $\emptyset \neq S_i \subseteq [m_i]$.
Applying the valuation property in each argument in turn (Lemma~\ref{lem:incl-excl}) expands
\[
    I(A_1, \ldots, A_k) = \sum_{\emptyset \neq S_1 \subseteq [m_1]} \cdots \sum_{\emptyset \neq S_k \subseteq [m_k]} (-1)^{\sum_i (|S_i| + 1)}\, I\big(B_{1,S_1}, \ldots, B_{k,S_k}\big).
\]
Each $B_{i,S_i}$ is convex, so the leaf overlap $\bigcap_i B_{i,S_i}$ is convex---empty or contractible.
By Separation and Normalization, together with Topological invariance (which lets the normalization on one contractible set apply to every contractible leaf), $I(B_{1,S_1}, \ldots, B_{k,S_k}) = \chi\big(\bigcap_i B_{i,S_i}\big)$ (namely $1$ if the convex leaf is nonempty, $0$ if empty).
The multivaluation $(A_1, \ldots, A_k) \mapsto \chi(\bigcap_i A_i)$ obeys the \emph{same} inclusion--exclusion expansion---it is a valuation in each argument---with the same leaf values.
Hence
\[
    I(A_1, \ldots, A_k) = \sum_{S_1, \ldots, S_k} (-1)^{\sum_i (|S_i| + 1)}\, \chi\big(\textstyle\bigcap_i B_{i,S_i}\big) = \chi\big(\textstyle\bigcap_i A_i\big),
\]
at every $\mathbf{t}$.
\end{proof}

\bibliographystyle{abbrvnat}
\bibliography{main}

\end{document}